\documentclass[10pt,a4paper,reqno]{amsart}
\usepackage[english]{babel}
\usepackage{latexsym}
\usepackage{amssymb,amsbsy,amsmath,amsfonts,amssymb,amscd}
\usepackage{amsmath, amsthm}
\usepackage[utf8]{inputenc}
\usepackage{epsfig, graphicx}
\usepackage{color}
\usepackage{fullpage}
\usepackage{graphics}
\usepackage{wrapfig}
\usepackage{enumitem}
\usepackage{mathrsfs}
\usepackage{environ}
\usepackage{lipsum}
\usepackage{enumitem} 
\usepackage{tcolorbox}
\usepackage{dsfont}
\usepackage[colorlinks=true,linkcolor=wine-stain,citecolor=blue]{hyperref}
\usepackage{bm}
\usepackage{mathtools,accents}
\usepackage{pgf, tikz}
\usepackage{caption}
\usepackage{calc}
\usepackage{etoolbox}
\usepackage[foot]{amsaddr}
\usepackage{breakcites}

\allowdisplaybreaks
\definecolor{wine-stain}{rgb}{0.7,0,0}

\def\X{\mathrm{X}}

\def\V{\mathrm{V}}
\def\U{\mathrm{U}}
\def\W{\mathrm{W}}

\def\R{\mathbb R}

\def\N{\mathbb N}
\def\P{\mathbb P}
\def\T{\mathbb T}

\def\det{\mathrm{det}}

\def\Ld{\mathrm{L}}
\def\H{\mathrm{H}}

\def\B{\mathrm{B}}
\def\E{\mathrm{E}}
\def\W{\mathrm{W}}
\def\D{\mathrm{D}}

\newcommand{\eps}{\varepsilon}
\newcommand{\ueps}{u_{\varepsilon}}

\newcommand{\rhoeps}{\rho_{\varepsilon}}

\newtcolorbox{dev}{arc=0pt}
\newcounter{compteur}
\counterwithin{compteur}{section}

\definecolor{thmcolor}{rgb}{0.8,0.14,0.2}
\definecolor{defcolor}{rgb}{0.0,0.50,0.0}
\definecolor{excolor}{rgb}{0.50,0.0,0.990}
\definecolor{applicolor}{rgb}{0.50,0.0,0.990}

\newtheoremstyle{thm}
  {\topsep}
  {\topsep}
  {\itshape}
  {0pt}
  {\bfseries}
  {.}
  { }
  {\textcolor{black!100}{\thmname{#1}\thmnumber{ #2}}\thmnote{ (#3)}}

\newtheoremstyle{def}
  {3pt}
  {3pt}
  {}
  {0pt}
  {\bfseries}
  {.}
  { }
  {\textcolor{black!100}{\thmname{#1}\thmnumber{ #2}}\thmnote{ (#3)}}
  
\newtheoremstyle{ex}
  {1pt}
  {1pt}
  {}
  {0pt}
  {\bfseries}
  {.}
  { }
  {\textcolor{black!100}{\thmname{#1}\thmnumber{ #2}\thmnote{ (#3)}}}
  
\theoremstyle{thm}
\newtheorem{theoreme}[compteur]{Theorem}
\newtheorem{corollaire}[compteur]{Corollary}
\newtheorem{proposition}[compteur]{Proposition}
\theoremstyle{appli}

\newtheorem{lemme}[compteur]{Lemma}
\theoremstyle{def}
\newtheorem{definition}[compteur]{Definition}
\newtheorem{notation}[compteur]{Notation}
\theoremstyle{ex}

\newtheorem{remarque}[compteur]{Remark}

\newtheorem{assumption}{Assumption}

\numberwithin{equation}{section}

\newcommand{\nocontentsline}[3]{}
\newcommand{\tocless}[2]{\bgroup\let\addcontentsline=\nocontentsline#1{#2}\egroup}

\makeatletter

\renewcommand{\tocsection}[3]{%
  \indentlabel{\@ifnotempty{#2}{\bfseries\ignorespaces#1 #2\quad}}\bfseries#3}
\renewcommand{\tocsubsection}[3]{%
  \indentlabel{\@ifnotempty{#2}{\ignorespaces#1 #2\quad}}#3}
\renewcommand{\tocsubsubsection}[3]{%
  \indentlabel{\@ifnotempty{#2}{\ignorespaces#1 #2\quad}}#3}

\newcommand\@dotsep{4.5}
\def\@tocline#1#2#3#4#5#6#7{\relax
  \ifnum #1>\c@tocdepth 
  \else
    \par \addpenalty\@secpenalty\addvspace{#2}%
    \begingroup \hyphenpenalty\@M
    \@ifempty{#4}{%
      \@tempdima\csname r@tocindent\number#1\endcsname\relax
    }{%
      \@tempdima#4\relax
    }%
    \parindent\z@ \leftskip#3\relax \advance\leftskip\@tempdima\relax
    \rightskip\@pnumwidth plus1em \parfillskip-\@pnumwidth
    #5\leavevmode\hskip-\@tempdima{#6}\nobreak
    \leaders\hbox{$\m@th\mkern \@dotsep mu\hbox{.}\mkern \@dotsep mu$}\hfill
    \nobreak
    \hbox to\@pnumwidth{\@tocpagenum{\ifnum#1=1\bfseries\fi#7}}\par
    \nobreak
    \endgroup
  \fi}
\AtBeginDocument{%
\expandafter\renewcommand\csname r@tocindent0\endcsname{0pt}
}
\def\l@subsection{\@tocline{2}{0pt}{2.5pc}{5pc}{}}
\def\l@subsubsection{\@tocline{3}{0pt}{4pc}{5pc}{}}
\makeatother

\title{Global derivation of a Boussinesq-Navier-Stokes type system from fluid-kinetic equations}
\author{Lucas Ertzbischoff}
\address{Department of Mathematics, Imperial College London, London, SW7 2AZ, United-Kingdom (\href{mailto:l.ertzbischoff@imperial.ac.uk}{l.ertzbischoff@imperial.ac.uk})}

\begin{document}

\maketitle
\begin{abstract}
We study a hydrodynamic limit of the Vlasov-Navier-Stokes system with external gravity force. We answer a question raised by Han-Kwan and Michel in \cite{HKM} concerning the limit towards a Boussinesq-Navier-Stokes type system. Our work provides a rigorous derivation of such hydrodynamic equations for arbitrarily large times, starting from the previous fluid-kinetic coupling. To do so, we consider a particular spatial geometric setting corresponding to the half-space case. Our proof is based on an absorption effect at the boundary
which leads to crucial decay in time estimates.
\end{abstract}

\renewcommand{\contentsname}{Contents}

{ \hypersetup{linktoc=page, linkcolor=wine-stain}
\tableofcontents
}

\section{Introduction}
In this work, we consider the following Vlasov-Navier-Stokes system set in the half-space:
\begin{align}
\partial_t f +v \cdot \nabla_x f   + {\rm div}_v [f(u-v)-f e_3 ]&=0, \ \ &&t>0, \ \ (x,v) \in \R^3_+ \times \R^3, \label{eq:Vlasov} \\
\partial_t u + (u \cdot \nabla_x)u -  \Delta_x u + \nabla_x p &=\int_{\R^3} f (v-u) \, \mathrm{d}v,  \ \ &&t>0, \ \ x \in  \R^3_+, \label{eq:NS} \\ 
\mathrm{div}_x \, u &=0, \ \ &&t>0, \ \ x \in  \R^3_+,
\label{eq:NS2}
\end{align}
with
\begin{align*}
\R^3_+:=\R^2 \times (0, +\infty), \ \ e_3:=(0,0,1).
\end{align*}
This system of equations accounts for the evolution of a cloud of droplets within an ambient viscous fluid. It belongs to the wide family of fluid-kinetic models (see e.g. \cite{oro,will, D}). More precisely, the particles are described thanks to a distribution function $f(t,\cdot,\cdot) \in \R^+$ on the phase space $\R^3_+ \times \R^3$ while the fluid is described by its velocity $u(t,\cdot) \in \R^3$ and pressure $p(t,\cdot) \in \R$. The surrounding fluid is assumed to be incompressible, homogeneous and viscous, and the monodispersed phase of particles is sufficiently dilute that collisions can be neglected. Here, the third term in the Vlasov equation \eqref{eq:Vlasov} asserts for the acceleration undergone by the particles and which comes from the action of the fluid (drag term $u-v$) and the gravity (external gravity force $-e_3$). The particles also act on the fluid by retroaction, inducing a source term
\begin{align*}
F(t,x)=\int_{\R^3} f(t,x,v) (v-u(t,x)) \, \mathrm{d}v
\end{align*}
in the Navier-Stokes equations \eqref{eq:NS}. This term is referred to as the \textit{Brinkman force}. Note also that the gravity force in the Navier-Stokes equations \eqref{eq:NS} is absorbed in the pressure term. The physical constants are all normalized in \eqref{eq:Vlasov}--\eqref{eq:NS}--\eqref{eq:NS2}.

The system is endowed with the initial conditions
\begin{align*}
u(0, x)=u_0(x), \ \  f(0, x, v)=f_0(x,v), \ \ (x,v) \in \R^3_+ \times \R^3.
\end{align*}
The boundary conditions for the Vlasov-Navier-Stokes system read as follows: we prescribe the homogeneous Dirichlet boundary conditions for the fluid
\begin{align}\label{bcond-fluid}
u(t,\cdot)=0, \text{ on } \partial \R^3_+=\R^2 \times \lbrace 0 \rbrace.
\end{align}
We also introduce the following outgoing/incoming phase-space boundaries:
\begin{align*}
\Sigma^{\pm}&:= \left\lbrace  (x,v) \in \partial \R^3_+ \times \R^3   \mid \pm v \cdot n(x)>0 \right\rbrace,\\
\Sigma_0&:= \left\lbrace  (x,v) \in \partial \R^3_+ \times \R^3   \mid v \cdot n(x)=0 \right\rbrace,\\
\Sigma &:= \Sigma^+ \sqcup \Sigma^- \sqcup \Sigma_0= \partial \R^3_+ \times \R^3,
\end{align*}
where $n(x)$ is the outer-pointing normal vector to the boundary $\partial \R^3_+$ at point $x$.
We prescribe the absorption boundary conditions for the distribution function:
\begin{align}\label{bcond-f}
f(t,\cdot,\cdot)=0, \text{ on } \Sigma^{-}.
\end{align}
%
%

\medskip

In this paper, we are interested in a particular hydrodynamic limit of the system \eqref{eq:Vlasov}--\eqref{eq:NS}--\eqref{eq:NS2}. Roughly speaking, it corresponds to a high-field regime where the particle volume fraction is small compared to the one of the fluid, and where the Stokes number is small. Physically, this means that the particles tend to follow the ambient fluid and that the inertial effects are not very important. After a nondimensionalization based on physical quantities appearing in the system, and that we do not detail here\footnote{We refer to \cite{ErtzbischoffPHD} for more details about this physical scaling procedure. An alternative nondimensionalization can also be found in the PhD thesis \cite{HoferPHD} of Richard Höfer.}, it consists in 

\begin{equation}\label{VNSe-rough}
    \left\{
\begin{aligned}
\partial_t f +v \cdot \nabla_x f   + \dfrac{1}{\eps}{\rm div}_v [f(u-v-e_3) ]&=0, \\
\partial_t u + (u \cdot \nabla_x)u- \Delta_x u + \nabla_x p &=\int_{\R^3} f(v-u) \, \mathrm{d}v, \\ 
\mathrm{div}_x \, u &=0,
\end{aligned}
\right.
\end{equation}
for some parameter $0 <\eps \ll1$. Our main goal is to justify an approximation of the Vlasov-Navier-Stokes system under this regime. This should lead, in some sense to be made precise later, to an hydrodynamic system of the form
\begin{equation}\label{preBNS}
    \left\{
\begin{aligned}
\partial_t \rho + \mathrm{div}_x[\rho (u-e_3)]&=0, \\
\partial_tu+ (u\cdot\nabla_{x}) u-\Delta_{x}u+\nabla_{x}p&=-\rho e_3, \\
\mathrm{div}_x \, u &=0,
\end{aligned}
\right.
\end{equation}
which is a Boussinesq-Navier-Stokes type system without diffusivity. The question of this rigorous passage to the limit has been raised as an open problem by Han-Kwan and Michel in \cite{HKM}. 
In the current article, we establish the global derivation of \eqref{preBNS} from \eqref{VNSe-rough}. A rough version of our main result if the following.

\begin{theoreme}
Assume that $(u_\eps^0, f_\eps^0)$ are smooth and small enough initial data, uniformly in $\eps$. If $f_\eps^0$ is decaying fast enough with respect to $x$ and $v$ then any solution $(u_\eps, f_\eps)$ to the system \eqref{VNSe-rough} with initial data $(u_\eps^0, f_\eps^0)$ satisfies for any $T>0$
$$
        u_{\eps} \underset{\eps \rightarrow 0}{\longrightarrow} u \, \text{ in } \, \Ld^2(0,T;\Ld^2(\R^3_+))  \quad\text{and}  \quad  \int_{\R^3} f_\eps \, \mathrm{d}v \underset{\eps \rightarrow 0}{\longrightarrow} {} \rho \,  \text{ in } \, \Ld^{\infty}(0,T;\Ld^{\infty}(\R^3_+)),$$
where $(u, \rho)$ is a strong solution to \eqref{preBNS}.
  \end{theoreme}

We refer to Subsection \ref{Subsec:AssumpResult} for a more precise version of the statements (see Theorems \ref{thCV}--\ref{th-rate:cvgence}).
As we shall explain later on, this result is \textit{not} a consequence of the hydrodynamic limits of the Vlasov-Navier-Stokes system studied by Han-Kwan and Michel in \cite{HKM}. Their analysis (in the gravity-less case, i.e. without the term $-e_3$ in \eqref{eq:Vlasov}) only allows for a local in time derivation of \eqref{preBNS} starting from  \eqref{VNSe-rough}. Roughly speaking, this comes from a continuous injection of energy in the system coming from the additional gravity term. Our main contribution is thus to justify the previous limit for \textit{arbitrarily large} times. Let us point out that this result will be obtained thanks to the combined mechanism of gravity \textit{and} absorption boundary condition on the half-space.

\subsection{Formal limit}\label{Subsec:formallim}
Let us formally derive the limit macroscopic system \eqref{preBNS}. For any $\eps>0$, we consider a solution $(u_\eps, f_\eps)$ to the system
\begin{equation}
    \label{VNSeps}
    \tag{VNS$_\eps$}
    \left\{
\begin{aligned}
\partial_t f_\eps +  v\cdot\nabla_{x}f_\eps +\frac{1}{\varepsilon} \mathrm{div}_{v}\left[f_\eps (u_\eps -v-e_3)\right]&=0, \ \ &&t>0, \ \ (x,v) \in \R^3_+ \times \R^3,  \\
    \partial_t u_\eps + (u_\eps  \cdot\nabla_{x})u_\eps -\Delta_{x}u_\eps +\nabla_{x}p_\eps &= j_{\eps}-\rho_{\eps} u_\eps, \ \ &&t>0, \ \ x \in \R^3_+, \\
  \mathrm{div}_{x} \, u_\eps &=0, \ \ &&t>0, \ \ x \in \R^3_+,
\end{aligned}
\right.
\end{equation}
where
\begin{align*}
\rho_{\eps}:=\int_{\R^3} f_\eps(t,x,v) \, \mathrm{d}v, \ \ j_{\eps}:=\int_{\R^3} v f_\eps(t,x,v) \, \mathrm{d}v.
\end{align*}
This system is supplemented with initial conditions and with boundary conditions similar to \eqref{bcond-fluid}--\eqref{bcond-f}.

\medskip

Assume that
\begin{align*}
u_\eps \overset{\eps \rightarrow 0}{\longrightarrow} u, \ \ \rho_\eps \overset{\eps \rightarrow 0}{\longrightarrow} \rho.
\end{align*}


Integrating the Vlasov equation on $\R^3$ in velocity yields the conservation of mass while multiplying the Vlasov equation by $v$ and then integrating on $\R^3$ in velocity yields the conservation of momentum for the particles: this reads
\begin{equation*}
    \left\{
\begin{aligned}
\partial_t \rho_\eps + \mathrm{div}_x \,  j_\eps&=0, \\[2mm]
\partial_t j_\eps + \mathrm{div}_x \, \left( \int_{\R^3} v \otimes v f_\eps \, \mathrm{d}v\right) &= \frac{1}{\eps} \left( \rho_\eps (u_\eps-e_3)- j_\eps \right).
\end{aligned}
\right.
\end{equation*}
Assuming the following convergence
$$
 j_\eps \underset{\eps \rightarrow 0}{\rightharpoonup} j,
$$
we deduce that we must have
$$
\rho_\eps (u_\eps-e_3)- j_\eps \underset{\eps \rightarrow 0}{\rightharpoonup} 0.
$$
We thus formally get
$$
j = \rho (u-e_3), 
$$
and then
$$
\partial_t \rho + \mathrm{div}_x[\rho (u-e_3)]=0,$$
as well as the source term $-\rho e_3$ in the Navier-Stokes equations. Note that we have dealt with the convergence of products in a formal way. As we will see in more detail later on, the rigorous convergence of $j_\eps -\rho_\eps u_\eps + \rho_\eps e_3$ will be a crucial issue of the analysis.

Concerning the boundary conditions satisfied by $(\rho,u)$, recall that $u_\eps$ satisfies a Dirichlet boundary condition \eqref{bcond-fluid} and that $f_\eps$ satisfies an absorption boundary condition \eqref{bcond-f}. We thus hope for
\begin{align*}
u(t, \cdot)_{\mid x_3=0}=0,
\end{align*}
at the limit $\eps \rightarrow 0$, without any boundary condition for the density $\rho$. Indeed, the transport equation satisfied by $\rho(t,x)$ in \eqref{BNS} requires a condition on the subset of $\lbrace x_3=0 \rbrace$ which is
\begin{align*}
\Gamma^{-}(t):=\left\lbrace x \in \R^2 \times \lbrace 0 \rbrace \mid  [u(t,x) -e_3]\cdot n(x)<0  \right\rbrace.
\end{align*}
But this set is empty since $u(t, \cdot)_{\mid  \R^2 \times \lbrace 0 \rbrace}=0$ and $(-e_3) \cdot n(x)= (-e_3) \cdot (-e_3)=1$.

\bigskip

All in all, the formal limit of \eqref{VNSeps} as $\eps \to 0$ is the following Boussinesq-Navier-Stokes type system set on $\R^3_+$:
\begin{equation}
    \label{BNS}
    \left\{
\begin{aligned}
\partial_t \rho + \mathrm{div}_x[\rho (u-e_3)]&=0, \ \ &&t>0, \ \ x \in \R^3_+, \\
\partial_tu+ (u\cdot\nabla_{x}) u-\Delta_{x}u+\nabla_{x}p&=-\rho e_3, \ \ &&t>0, \ \ x \in \R^3_+, \\
\mathrm{div}_x \, u &=0, \ \ &&t>0, \ \ x \in \R^3_+,
\end{aligned}
\right.
\end{equation}
with the boundary condition
\begin{align*}
u(t,\cdot)=0, \text{ on } \R^2 \times \lbrace 0 \rbrace.
\end{align*}
At the limit, the system thus consists in a transport of the local density of particles by the flow of the fluid and the gravity, while the action of the particles appears as a forcing term in the Navier-Stokes equations, in the direction of the gravity field. In short, the velocities of the particles align on the sum of the fluid velocity and gravity field.

\bigskip

Let us comment on the system \eqref{BNS} we have just obtained. It formally ressembles to the classical Boussinesq-Navier-Stokes system (without diffusivity), which appears in the literature with a vector field $u$ in the transport equation on $\rho$, and not $u-e_3$. However, since it essentially shares the same features, we shall refer to \eqref{BNS} as a Boussinesq-Navier-Stokes type system.

The Boussinesq-Navier-Stokes system is a standard geophysical fluid dynamics model (see \cite{S-ocea,V-ocea, Maj-ocea}). From the analysis point of view, and because its 2D version retains several key features of 3D incompressible models (see e.g. \cite{MB}), it has recently received significant attention. The \textit{existence theory} (in the less-diffusivity case) has been for instance developed in \cite{HL,Chae,AH,HmidiK,DP, HR}, while \textit{stability} of hydrodynamic equilibria is studied in \cite{DWZZ,TWZZ,MSZ,DS}. We also refer to the so-called \textit{temperature patch (or front) problem} adressed in \cite{DZ,GGJ,CMX}. Note that there is no diffusivity in the transport equation on the density, which makes the mathematical analysis much more challenging than the thermal diffusion case.

\medskip

Note that when the Navier-Stokes equations are replaced by the (steady) Stokes equations in \eqref{BNS} (i.e. neglecting the self-advection term), we obtain
\begin{equation}
    \label{Transport-Stokes}
    \left\{
\begin{aligned}
\partial_t \rho + \mathrm{div}_x[\rho (u-e_3)]&=0, \ \ &&t>0, \ \ x \in \R^3_+, \\
-\Delta_{x} u +\nabla_{x}p&=-\rho e_3, \ \ &&t>0, \ \ x \in \R^3_+, \\
\mathrm{div}_x \, u &=0, \ \ &&t>0, \ \ x \in \R^3_+,
\end{aligned}
\right.
\end{equation}
which is classicaly referred to as the Transport-Stokes system and which appears as an interesting model of sedimentation. On the whole space, this system has been obtained by Höfer in \cite{HoferInertia} from the Vlasov-Stokes system, and by considering the same scaling as ours for the hydrodynamical limit. 

If starting at the microscopic level (with a N--solid particle system coupled with a fluid equation), one can seek to recover the related mesoscopic and macroscopic models. Up to our knowledge, the best results only deal with the direct passage to the macroscopic system \eqref{Transport-Stokes}, by working in a dilute regime where the inertia of the particles is neglected (see the work of Mecherbet \cite{Mech}, Höfer \cite{Hof} and Höfer and Schubert \cite{HS} and the related mean-field techniques). For further results on the Transport-Stokes system \eqref{Transport-Stokes}, we refer to \cite{HS, Mech2, Antoine, MS, Gray}.

The derivation of the Vlasov-Navier-Stokes system \eqref{eq:Vlasov}--\eqref{eq:NS}--\eqref{eq:NS2} from microscopic laws is however an outstanding open problem and results in that direction are still fragmentory. An homogenization procedure has for instance led to the derivation of the sole Brinkman force in the fluid equation (see \cite{DGR,H,HMS, CH, H2}) while \cite{FLR1, FLR2} have adressed the case of an intermediate N--particle-fluid coupling with dissipation in velocity\footnote{We also mention a different approach taken in \cite{BDGR1,BDGR2}, starting from a system of biphase Boltzmann equations for the gas and the droplets, but which still remains partly formal.}.


\subsection{Definitions and notations}
Until the end of this work, we shall refer to the system \eqref{VNSeps} as the VNS system. Recall that for all $\eps>0$, we have set
\begin{align*}
\rho_\eps(t,x):= \int_{\R^3} f_\eps(t,x,v) \, \mathrm{d}v, \ \  j_\eps(t,x):= \int_{\R^3} v f_\eps(t,x,v) \, \mathrm{d}v, \ \ t \geq 0, \ \ x \in \R^3_+.
\end{align*}
We denote by $\mathscr{D}_{\mathrm{div}}(\R^3_+)$ the set of smooth $\R^3$ valued divergence free vector-fields having compact support in $\R^3_+$. For all $q \in (1,+\infty)$, the closures of $\mathscr{D}_{\mathrm{div}}(\R^3_+)$ in $\Ld^q(\R^3_+)$ and in $\H^1(\R^3_+)$ are respectively denoted by $\Ld^q_{\mathrm{div}}(\R^3_+)$ and by $\H^1_{0,\mathrm{div}}(\R^3_+)$. We write $\H^{-1}_{\mathrm{div}}(\R^3_+)$ for the dual of the latter.

If $q \in (1,+\infty)$ is given, any vector field $u \in \Ld^q(\R^3_+)$ can be uniquely decomposed as
\begin{align*}
& u=\widetilde{u} + \nabla p, \\
& \widetilde{u}  \in \Ld^q_{\mathrm{div}}(\R^3_+), \ p \in \Ld^q(\R^3_+), \ \nabla p \in \Ld^q(\R^3_+),
\end{align*}
We recall that the Leray projection $\mathbb{P}_q : u \mapsto \widetilde{u} $ is continuous from $\Ld^q(\R^3_+)$ to $\Ld^q_{\mathrm{div}}(\R^3_+)$.

Considering the following Stokes operator
\begin{align*}
A_q := -\mathbb{P}_q \Delta , \ \  D(A_q):=\Ld^q_{\mathrm{div}}(\R^3_+) \cap \W^{1,q}_0(\R^3_+) \cap \W^{2,q}(\R^3_+),
\end{align*}
we also set for $s \in (1,+\infty)$
\begin{align}\label{domaineDqs}
\mathrm{D}_q^{1-\frac{1}{s},s}(\R^3_+):=\left( D(A_q),\Ld^q_{\mathrm{div}}(\R^3_+) \right)_{1/s,s},
\end{align}
where $( \ , \ )_{1/s,s}$ refers to the real interpolation space of exponents $(1/s,s)$. In the case of the Stokes operator $A_q$, which generates an analytic semigroup $e^{-tA_q}$, the quantity 
\begin{align*}
\Vert u \Vert_{\Ld^q(\R^3_+)}+ \left( \int_0^{\infty} \Vert A_q e^{-t A_q} u \Vert_{\Ld^q(\R^3_+)}^s \, \mathrm{d}t  \right)^{1/s}
\end{align*}
defines an equivalent norm on $\mathrm{D}_q^{1-\frac{1}{s},s}(\R^3_+)$ (see \cite[Chapter 5]{Lunardi}).

\medskip

Next, we define several functionals which are crucial in the analysis of the VNS system.
\begin{definition}
The \textbf{kinetic energy} of the VNS system is defined by
\begin{align}\label{def:KineticEnergy}
\E_{\varepsilon}(t):=\frac{1}{2}\Vert u_{\varepsilon}(t) \Vert_{\Ld^2(\R^3_+)}^2 + \frac{\varepsilon}{2}\int_{\R^3_+ \times \R^3}\vert v \vert^2 f_{\varepsilon}(t,x,v) \, \mathrm{d}x \, \mathrm{d}v.
\end{align}
The \textbf{potential energy} of the VNS system is defined by
\begin{align}\label{def:PotEnergy}
\E_{\varepsilon}^{\mathrm{p}}(t):= \int_{\R^3_+ \times \R^3} x_3 f_{\eps}(t,x,v) \,  \mathrm{d}x \, \mathrm{d}v=\int_{\R^3_+} x_3 \rho_\eps(t,x) \, \mathrm{d}x.
\end{align}
We finally define the \textbf{total energy} of the VNS system as
\begin{align}\label{def:Energy}
\mathcal{E}_{\eps}(t):=\E_{\varepsilon}(t)+\E_{\varepsilon}^{\mathrm{p}}(t),
\end{align}
and the \textbf{dissipation} of the VNS system as
\begin{align}\label{def:Dissipation}
\D_{\varepsilon}(t)&:= \Vert \nabla_x u_{\varepsilon}(t) \Vert_{\Ld^2(\R^3_+)}^2+ \int_{\R^3_+ \times \R^3} \left\vert v -u_{\varepsilon}(t,x) \right\vert^2 f_{\varepsilon}(t,x,v)  \, \mathrm{d}x \, \mathrm{d}v.
\end{align}
\end{definition}
Formally, we can multiply the Navier-Stokes equations in \eqref{VNSeps} by $u_\eps$ and then integrate on $\R^3_+$ with suitable integrations by parts (using the divergence free condition). We can also multiply the Vlasov equation in \eqref{VNSeps} by $\vert v \vert^2/2$ and by $x_3$ and then integrate on $\R^3_+ \times \R^3$ with suitable integrations by parts (and using the absorption boundary condition \eqref{bcond-f}). All in all, we formally obtain
\begin{align*}
\dfrac{\mathrm{d}}{\mathrm{d}t}\mathrm{E}_\eps(t) + \mathrm{D}_\eps(t)\leq -\int_{\R^3_+ \times \R^3} v \cdot e_3 f_{\varepsilon}(t,x,v) \, \mathrm{d}x \, \mathrm{d}v, \ \ \ \  \dfrac{\mathrm{d}}{\mathrm{d}t}\E_{\varepsilon}^{\mathrm{p}}(t) \leq \int_{\R^3_+ \times \R^3} v \cdot e_3 f_{\varepsilon}(t,x,v) \, \mathrm{d}x \, \mathrm{d}v.
\end{align*}
We now define the class of admissible initial data for the VNS system.
\begin{definition}[Initial condition]\label{CIadmissible}
Let $\eps >0$. We shall say that a couple $(u_{\eps}^0,f_{\eps}^0)$ is an admissible initial condition if
\begin{align}
u_{\eps}^0 &\in \Ld^2_{}(\R^3_+), \ \ \mathrm{div}_x \, u_\eps^0=0,  \label{CI:fluid}\\[1mm]
f_{\eps}^0 &\in \Ld^1 \cap \Ld^{\infty}(\R^3_+ \times \R^3), \label{CI-f1}\\[1mm]
f_{\eps}^0 &\geq 0, \ \ \int_{\R^3_+ \times \R^3} f_{\eps}^0(x,v) \, \mathrm{d}x \, \mathrm{d}v =1, \label{CI-f2}\\[1mm]
(x,v)& \mapsto \vert v \vert^2 f_{\eps}^0 (x,v) \in \Ld^1(\R^3_+ \times \R^3), \label{CI-f3} \\[1mm]
(x,v)& \mapsto x_3f_{\eps}^0 (x,v) \in \Ld^1(\R^3_+ \times \R^3). \label{CI-f4}
\end{align}
\end{definition}
In the rest of this article, we will consider global weak solutions to the VNS system which satisfy an energy-dissipation inequality and which are defined in the following sense\footnote{We refer to \cite[Appendix]{E} for more details about the construction of such global weak solutions (for any $ \eps >0$ fixed). The introduction of this reference also provides further information on the Cauchy problem for the VNS system.}.
\begin{definition}[Weak solutions]\label{weak-sol}
Let $\eps>0$. Given an admissible initial condition $(u_{\eps}^0,f_{\eps}^0)$ in the sense of Definition \ref{CIadmissible}. we say that a pair $(u_\eps, f_\eps)$ is a global weak solution to the Vlasov-Navier-Stokes system with boundary conditions (\ref{bcond-fluid})-(\ref{bcond-f}) and with initial condition $(u_{\eps}^0,f_{\eps}^0)$ if
\begin{align*}
&u_\eps\in \Ld^{\infty}_{\mathrm{loc}}(\R^+;\Ld^2_{\mathrm{div}}(\R^3_+))\cap  \Ld^{2}_{\mathrm{loc}}(\R^+;\H^1_{0,\mathrm{div}}(\R^3_+)),\\[1mm]
f_\eps &\in \Ld^{\infty}_{\mathrm{loc}}(\R^+;\Ld^1 \cap \Ld^{\infty}(\R^3_+ \times \R^3)),\\[1mm]
j_\eps-\rho_\eps u_\eps &\in \Ld^{2}_{\mathrm{loc}}(\R^+;\H^{-1}_{\mathrm{div}}(\R^3_+)),
\end{align*}
and if
\begin{itemize}
\item $u_\eps$ is a Leray solution to the Navier-Stokes equations satisfying for any $t \geq 0$ and almost every $0 \leq s  \leq t$ (including $s=0$)
\begin{multline*}
\Vert u_\eps(t) \Vert_{\Ld^2(\R^3_+)}^2  + 2 \int_s ^t \Vert \nabla u_\eps(\tau) \Vert_{\Ld^2(\R^3_+)}^2 \mathrm{d}\tau \\ 
\leq   \Vert u_\eps(s) \Vert^2_{\Ld^2(\R^3_+)} + 2 \int_s ^t \int_{\R^3_+} \left( j_\eps(\tau,x)-\rho_\eps u_\eps(\tau,x) \right) \cdot u_\eps(\tau,x) \, \mathrm{d}x \,  \mathrm{d}\tau.
\end{multline*}
\item $f_\eps$ is a renormalized nonnegative solution to the Vlasov equation with absorption boundary condition.
\end{itemize}
Furthermore, for any $t \geq 0$ and almost every $0 \leq s \leq t$ (including $s=0$), the following inequality holds for all $\eps >0$:
\begin{align}
\E_{\varepsilon}(t) + \int_s^t \D_{\varepsilon}(\tau) \, \mathrm{d}\tau  &\leq \E_{\varepsilon}(s) - \int_s^t \int_{\R^3_+ \times \R^3} v \cdot e_3 f_{\varepsilon}(\tau,x,v) \, \mathrm{d}x \, \mathrm{d}v \, \mathrm{d}\tau\label{ineq:energy1}.
\end{align}
\end{definition}

\bigskip

The notion of renormalized solutions (in the sense of DiPerna and Lions for transport equations \cite{DPL}) for the Vlasov equation allows to consider the trace of $f_\eps$ at the boundary of the half-space (see \cite{Misch}). This also provides some strong stability properties of such solutions. We refer to \cite[Appendix]{E} for further details.

\medskip

As we shall prove later on, the energy inequality \eqref{ineq:energy1} will be improved in Subsection \ref{Subsec:improvineq+conddecay}, by adding the contribution of the potential energy $\E_{\varepsilon}^{\mathrm{p}}$.  The new energy-dissipation inequality shall read
\begin{align*}
\mathcal{E}_{\eps}(t)+ \int_s^t \mathrm{D}_\eps(\tau) \, \mathrm{d}\tau &\leq \mathcal{E}_{\eps}(s).
\end{align*}
We refer to Lemma \ref{IneqEnergyPot} and to \eqref{ineq:energy2} for more details about the obtention of this structural inequality.

%

\begin{notation}
In the whole article, the notation $A \lesssim B$ will always denote the fact that there exists a universal constant $M>0$ independent of all the parameters such that
\begin{align*}
A \leq M B.
\end{align*}
\end{notation}

\subsection{Assumptions and main results}\label{Subsec:AssumpResult}

Let $(u_{\eps}^0)_{\eps>0}$ and $(f_{\eps}^0)_{\eps>0}$ be a family of admissible initial data in the sense of Definition \ref{CIadmissible}. We introduce the following set of assumptions.
\begin{assumption}[Regularity and decay assumption]\label{hypGeneral}
We assume that:
    
    \begin{itemize}
    
    	\item for any $\eps>0$, we have
    	\begin{align}
    	\tag{\textbf{A1-a}}
    	u_{\eps}^0 \in \H^1_0(\R^3_+)\cap \Ld^1(\R^3_+);
\end{align}
	\item there exist $p_0>3$, $s \in (2,3)$ and $p \in (3,p_0)$ such that
        \begin{align}\label{hypGen:regBesov}
        \tag{\textbf{A1-b}}
         \forall\eps>0 ,\qquad u_{\eps}^0\in \mathrm{D}_p^{1-\frac{1}{p},p}(\R^3_+) \cap  \mathrm{D}_3^{\frac{1}{2},2}(\R^3_+) \cap \D_3^{1-\frac{1}{s},s}(\R^3_+),
         \end{align}
         where we refer to \eqref{domaineDqs} for the definition of the previous spaces;
         \item for any $\eps>0$, we have
         \begin{align} \tag{\textbf{A1-c}}
          \vert v \vert^6 f_\eps^0  \in \Ld^1(\R^3_+ \times \R^3);
         \end{align}
        \item there exists $q>3$ such that for any $\eps>0$
            \begin{align}
            	\tag{\textbf{A1-d}}
            	 (1+x_3^q) (1+|v|^q)f_{\eps}^0\in\Ld^1(\R^3; \Ld^{\infty} \cap \Ld^1(\R^3_+)).
            \end{align}
    \end{itemize}
\end{assumption}

\begin{assumption}[Uniform boundedness assumption]\label{hypUnifBoundVNS}
We assume that there exists $M>1$ such that:
\begin{itemize}
\item for any $\eps>0$, we have
\begin{align}\label{hypUnifBoundVNS1}
\tag{\textbf{A2-a}}
\begin{split}
 \Vert{u_{\eps}^0}\Vert_{\Ld^1 \cap \H^1\cap \mathrm{D}_p^{1-\frac{1}{p},p}(\R^3_+)} \leq M, \\
  \Vert  (1+x_3^{q})(1+\vert v \vert^{q}) \, f_\eps^0 \Vert_{\Ld^1(\R^3; \Ld^{\infty} \cap \Ld^1(\R^3_+))} \leq M; \\
\end{split}
\end{align}
where the exponents $(p,q)$ refer to the ones introduced in Assumption \ref{hypGeneral}.

\item the total energy satisfies
\begin{align}\label{hypUnifBoundVNS2}
\tag{\textbf{A2-b}}
 \forall\eps>0 ,\qquad \mathcal{E}_\eps(0) < M.
\end{align}
\end{itemize}

\end{assumption}
\begin{assumption}[Smallness assumption]\label{hypSmallData}
   We assume that:
   
  \begin{itemize}
  \item if $\mathrm{C}_{\star}>0$ is the universal constant given by Proposition \ref{propdatasmall:VNSreg}, then
  \begin{align}\label{hypSmallDataSTRONG}
    \tag{\textbf{A3-a}}
           \forall\eps>0 ,\qquad\Vert{u_{\eps}^0}\Vert_{\H^{1}(\R^3_+)} < \left(\frac{\mathrm{C}_{\star}}{2} \right)^{1/2};
    \end{align}
    
    \item there exist $\eta>0$ and $\eps_0 >0$ such that for all $\eps \in (0, \eps_0)$, we have
    \begin{align}\label{hypSmallDataTech}
    \tag{\textbf{A3-b}}
         \mathcal{E}_\eps(0) + \Vert (1+\vert v \vert) f_\eps^0 \Vert_{\Ld^1(\R^3;\Ld^2(\R^3_+))}+ \Vert (\vert v \vert^{1+\iota}+x_3^{1+\iota}) f_\eps^0 \Vert_{\Ld^1(\R^3;\Ld^2(\R^3_+))} < \eta,
    \end{align}
    for some $\iota>0$.
  \end{itemize}
    \end{assumption}

Of course, one can make the previous assumptions with only $\eps>0$ small enough.
Furthermore, the choice of the exponent $q$ can be made more explicit and actually depends on the exponent $p$. For the sake of readibility, we do note give a precise value and we refer to the proofs where the assumptions will be used. Note that we do not assume that $ \eps \mapsto \mathcal{E}_\eps(0)$ tends to $0$ when $\eps \rightarrow 0$.

\medskip

The main results of our work read as follows. Consider $(u_\eps,f_\eps)$ a global weak solution to the Vlasov-Navier-Stokes system associated to an admissible initial data $(u_{\eps}^0,f_{\eps}^0)$ with $\eps >0$.
\begin{theoreme}\label{thCV}
    Under Assumptions \textbf{\ref{hypGeneral}}--\textbf{\ref{hypUnifBoundVNS}}--\textbf{\ref{hypSmallData}}, and assuming that
    $$
        u_{\eps}^0\underset{\eps \rightarrow 0}{\rightharpoonup}{} u^0\text{ in }w\text{-}\Ld^2(\R^3_+)\qquad\text{and}\qquad \rho_{\eps}^0\underset{\eps \rightarrow 0}{\rightharpoonup}{}\rho^0\text{ in }w^*\text{-}\Ld^\infty(\R^3_+),
    $$
    where
    \begin{align*}
    u^0 \in \Ld^2(\R^3_+), \ \ \mathrm{div}_x \, u^0=0, \ \ \rho^0 \in \Ld^1 \cap \Ld^{\infty}(\R^3_+), \ \ \rho^0 \geq 0,
    \end{align*}
    then, up to a subsequence, $(\ueps)_{\eps>0}$ converges to $u$ in $\Ld^2(0,T;\Ld^2_{\mathrm{loc}}(\R^3_+))$, $(\rhoeps)_{\eps>0}$ converges weakly-$*$ to $\rho$ in $\Ld^\infty((0,T)\times\R^3_+)$  for any $T>0$, where $(\rho,u)$ is a solution of
    \begin{equation}\label{limitRhoU}
        \left\{\begin{aligned}
            &\partial_t\rho+\mathrm{div}_x \, [\rho (u-e_3)]=0,\\
            &\partial_tu+(u\cdot\nabla_x)u-\Delta_xu+\nabla_xp=-\rho e_3,\\
            &\mathrm{div}_x \, u=0,\\
            &\rho|_{t=0}=\rho^0, \ \ u|_{t=0}=u^0,\\
            & u(t,\cdot)_{\mid x_3=0}=0,
        \end{aligned}\right.
    \end{equation}
with
\begin{align}
&\label{limit:u} u \in  \Ld^{\infty}_{\mathrm{loc}}(\R^+;\H^1_{\mathrm{div}}(\R^3_+))\cap \Ld^2_{\mathrm{loc}}(\R^+;\H^2(\R^3_+))\cap \Ld^1_{\mathrm{loc}}(\R^+;\W^{1,\infty}(\R^3_+)), \\[2mm]
& \label{limit:rho}\rho  \in \Ld^{\infty}_{\mathrm{loc}}(\R^+;\Ld^{\infty}(\R^3_+ \times \R^3)), \ \  \rho \geq 0.
\end{align}
\end{theoreme}
\begin{theoreme}\label{th-rate:cvgence}
Let $(u,\rho)$ be any global solution to the system \eqref{limitRhoU} with the regularity \eqref{limit:u}--\eqref{limit:rho}. Under Assumptions \textbf{\ref{hypGeneral}}--\textbf{\ref{hypUnifBoundVNS}}--\textbf{\ref{hypSmallData}}, there exist $\eps_0>0$, $\omega >0$ and $C_M>0$ such that if $T>0$, then for all $t \in [0,T]$ and $\eps \in (0, \eps_0)$
\begin{multline*}
\Vert u_\eps(t)-u(t) \Vert_{\Ld^2(\R^3_+)} +
\Vert  \rho_\eps(t)-\rho(t) \Vert_{\H^{-1}(\R^3_+)} \\
\lesssim e^{C_M(1+T)} \left ( \Vert u_\eps^0-u^0 \Vert_{\Ld^2(\R^3_+)}+ \Vert  \rho_\eps^0-\rho^0 \Vert_{\H^{-1}(\R^3_+)} + \eps^{\frac{1}{2}}(1+T)^{\frac{7}{10}}M^{\omega}\right).
\end{multline*}
\end{theoreme}
\begin{remarque}
Let us clarify the two previous results with respect to the limit system \eqref{limitRhoU}. 
\begin{enumerate}
\item Theorem \ref{thCV} implies a theorem of existence of a global strong solution for the Boussinesq-Navier-Stokes system \eqref{limitRhoU}. Of course, the smallness assumption of the initial data $(u_\eps^0)_\eps$ contained in Assumption \textbf{\ref{hypSmallData}} (which is uniform in $\eps$) is implicitely transferred to the initial data $u^0$ that we choose, thanks to a weak-compactness argument (up to an additional subsequence).
\item Theorem \ref{th-rate:cvgence} implies a theorem of uniqueness of global strong solutions for the Boussinesq-Navier-Stokes system \eqref{limitRhoU}, stated in a very weak form. More precisely, this theorem proves that there is a unique solution to the system for a class of initial conditions $(u^0, \rho^0)$ which can be approached in $\Ld^2(\R^3_+) \times \H^{-1}(\R^3_+)$ by a sequence $(u^0_\eps, \rho^0_\eps)$ satisfying Assumptions \textbf{\ref{hypGeneral}}--\textbf{\ref{hypUnifBoundVNS}}--\textbf{\ref{hypSmallData}}. This is somehow related to a smallness assumption of the initial data $(u^0, \rho^0)$ (recall that we work in dimension $3$).
 
Of course, this result is far from being optimal: we refer to \cite{DP} for further results, stated and proved in the case of $\R^3$. 
\end{enumerate}
\end{remarque}

\begin{remarque}
Let us also point out that our analysis somehow requires solutions to the Navier Stokes equations with \textit{high-regularity}. Since our proof will eventually be based on a weak-compactness argument, such control for the VNS system (uniform in $\eps$) shall be transferred to the solution of the limit system. It explains the regularity obtained in \eqref{limit:u}. 

It seems to be an open and natural problem to obtain the same kind of results by relaxing the regularity and smallness assumptions \ref{hypGeneral} and \ref{hypSmallData} (i.e. only considering weak solution to \eqref{VNSeps} in the sense of Definition \ref{weak-sol}). Note indeed that the limit system \eqref{BNS} admits global weak Leray solutions \cite{DP}.
\end{remarque}

\medskip

Let us finally highlight the \textbf{main original difficulties} that arise in the justification of the hydrodynamic limit of \eqref{VNSeps} towards \eqref{BNS}.

\begin{itemize}
\item First, the proof of Han-Kwan and Michel in \cite{HKM} (for the gravity-less case) cannot directly lead to a global in time result in our context. Indeed, their analysis is crucially based on the \textit{exponential decay} in time of the total (modulated) kinetic energy of the system set on $\T^3$ (i.e. a twisted version of $\E_\eps$). This somehow virtually provides any integrability in time, allowing for global results at some point. However, the decay of the energy on an unbounded domain is not exponential but only polynomial in time, as shown in \cite{HK}.  Furthermore, this decay is actually \textit{not guaranteed} when a constant gravity force is added, because an additional contribution to the energy has to be considered (see \eqref{ineq:energy1} and Remark \ref{rmq:conddecay}). Roughly speaking, the effect of gravity somehow destroys the remarkable energy decay displayed by the Vlasov-Navier-Stokes system without gravity.

\item Even if \cite{E} provides decay in time estimates for the VNS system with a gravity force and on the half-space, these results are of course \textit{not uniform in $\eps$} and are not suitable to justify the hydrodynamic limit when $\eps \rightarrow 0$.
\item On the whole space, it is unclear if the global solutions to the Boussinesq-Navier-Stokes system enjoy some decay in time estimates: more precisely, global existence results for that system such as \cite{DP, HR} only provide at most exponential upper bounds for the fluid velocity. More recent results \cite{BrandoSchon,BrandoCharaf,KW} seem to indicate that obtaining decay in time in the whole space case is actually not possible and that a growth phenomenon can occur.
\item Contrary to the Vlasov-Stokes case with gravity in \cite{HoferInertia}, the presence of the nonlinearity in the 3d Navier-Stokes equations is of course a new difficulty. As a matter of fact, the strategy used in \cite{HoferInertia} is strongly based on the linearity of the Stokes equations, as well as on some explicit representation formula for the solution. In addition, the nonlinearity formally forbids any direct long-time results requiring additional regularity, because of the coupling between $f$ and $u$. Note that we also do not require any regularity assumption on the kinetic part, which only belongs to some Lebesgue spaces. 
\end{itemize}

The previous observations somehow justify the choice of the half-space setting in this article, combined with appropriate boundary conditions (see \eqref{bcond-fluid}--\eqref{bcond-f}). Let us mention that on the half-space, the energy decay of the limiting system can be expected (see e.g \cite{PigongSchon}).

The outcome is also the fact that a delicate bootstrap argument ensuring uniformity in $\eps$ is required. This should prevent some possible growth in time of the solutions. We refer to Subsection \ref{subsection:Strat} for more details about the strategy of proof and the main mechanism allowing to overcome the previous difficulties.


In particular, we provide a particular spatial framework where we can positively answer the question of the global hydrodynamic limit raised in \cite{HKM} in the presence of gravity. Note that treating the case of the whole space still seems to be an open problem.

\subsection{Broad panorama on hydrodynamic limits for fluid-kinetic systems}
In this subsection, we draw a review of results concerning the derivation of hydrodynamical systems from fluid-kinetic models. Let us emphasize the fact that this contains different variants of the Vlasov-Navier-Stokes system, with different scalings in term of $\eps$.
\begin{itemize}
\item Some high friction regime (similar to the one considered in this paper) has been studied in the seminal work of Jabin in \cite{jabin2000macroscopic}: a Vlasov equation without coupling is considered and the fluid velocity is recovered in terms of a convolution of a moment of $f_\eps$ with a smooth kernel (in order to mimick Stokes flow). It highlights the monokinetic behavior $f_\eps(t) \overset{\eps \rightarrow 0}{\rightarrow} \rho(t) \otimes \delta_{v=U(t)}$ which is at stake in the system. The case of a given fluid velocity field with extra-integrability condition is considered in \cite{Jab2}. Goudon and Poupaud \cite{GouPou} then treated the case of a very thin spray for the particles, in which the fluid velocity is given as a fixed random field. In \cite{Gou}, the case of two different regimes is handled by Goudon for a coupled Burgers-Vlasov system: the asymptotic result heavily relies on the one-dimensional setting (see also \cite{CHW} for a 1D compressible model).
\item In the same time, fluid kinetic couplings where a \textit{Fokker-Planck term} smoothes out the kinetic equation (namely, a term $\Delta_v f$ is added on the left-hand side of the Vlasov equation) have been extensively studied: the two main pioneering results have been obtained by Goudon, Jabin and Vasseur in \cite{goudon2004hydrodynamic1,goudon2004hydrodynamic2} in a domain without boundary (and without gravity). In short, their proof is based on the obtention of global entropy bounds for the system. In this context, the distribution function $f_\eps$ tends to converge towards a (local) \textit{Maxwellian} which parameters are solutions to the limit equations. This kind of relaxation is highly linked to the smoothing in velocity in the kinetic equation and does not appear if the Fokker-Planck operator is absent.
  The scaling considered in \cite{goudon2004hydrodynamic1} (called the \textit{light particles} regime by the authors) leads to an advection-diffusion system referred to as the Smoluchowski-Navier-Stokes system, while the scaling considered in \cite{goudon2004hydrodynamic2} (called the \textit{fine particles} regime by the authors) leads to an inhomogeneous incompressible Navier-Stokes system. The second one requires the use of \textit{relative entropy} methods. These results have also been extended by Mellet and Vasseur in \cite{MV} for a compressible Navier-Stokes system and by Su and Yao in \cite{SY} for a non-homogeneous system.

We also refer to the formal analysis performed by Carrillo and Goudon in \cite{carrillo2006stability} and by Carrillo, Goudon and Laffitte in \cite{CGP} for the same system (but with the Euler equations instead of Navier-Stokes ones) where general external potentials and boundary conditions are discussed. 

\item In the case where there is no extra-dissipation in velocity in the kinetic equation (and with the dimension different from $1$), few results were known until now. 
A fine particles regime was derived in \cite{BDM} for a two phase Vlasov-Navier-Stokes system, but still at a formal level. In the direction of the Vlasov-Navier-Stokes system with gravity, the first result has been obtained, to the best of our knowledge, by Höfer in \cite{HoferInertia}: he considered the so-called \textit{inertialess limit} for the Vlasov equation coupled with the steady Stokes system in $\R^3$ and for compactly supported and regular initial distribution functions. His proof relies on a trajectorial analysis and leads to the derivation of the Transport-Stokes system \eqref{Transport-Stokes}. As mentioned previously, we adopt the same scaling in the current article.

\item Very recently, Han-Kwan and Michel have proposed a framework in \cite{HKM} to rigorously handle the high-friction limit of the full Vlasov-Navier-Stokes system in a tridimensional periodic setting (without gravity). They consider three differents regimes for the system, leading to hydrodynamical systems in the limit: following their denomination, 
\textit{light particles} and \textit{light and fast particles} regimes lead to Transport-Navier-Stokes system, while the \textit{fine particles} regime leads to an inhomogeneous Navier-Stokes system. These asymptotic models mainly come from a convergence of the distribution function $f_\eps$ towards a Dirac mass in velocity centered at the fluid velocity limit when $\eps \rightarrow 0$.

As we shall explain later on, their techniques are related to the recent progress concerning the large time behavior of the Vlasov-Navier-Stokes system on the tridimensional torus, performed by Han-Kwan, Moussa and Moyano in \cite{HKMM}. As mentioned in \cite{HKM}, the proof seems to be suitable for a local in time hydrodynamical limit in the gravity case (at least on the torus), that is a local in time derivation of the Boussinesq-Navier-Stokes system \eqref{BNS}. The interesting question of the global in time derivation of \eqref{BNS} (for unbounded domains) thus appears as a natural extension of \cite{HKM}, which left this problem as open. As explained above, serious difficulties due the gravity effect arise when looking for a global in time derivation.
\item Let us finally mention a specific case adressed by Moussa and Sueur in \cite{MoSu}, for a two-dimensional coupling between Vlasov and Euler equations, with gyroscopic effects: in the massless limit for the particles, one recovers the incompressible Euler system, through techniques close to the ones deviced by Brenier in \cite{BreVP2} for the study of the so-called \textit{gyrokinetic limit} of the Vlasov-Poisson system.
\end{itemize}

Let us again emphasize the fact that the hydrodynamical limits for the Vlasov-Navier-Stokes system (without smoothing in the kinetic equation) is closely linked to the monokinetic behavior of the distribution function $f_\eps$ when $\eps \rightarrow 0$, that is a convergence towards a Dirac distribution in velocity. It strongly differs from the Vlasov-Fokker-Planck case where there is a formal convergence towards a local Maxwellian.

Fluid-kinetic systems, such as the Vlasov-Navier-Stokes system, are not the only ones where Dirac masses can appear. Indeed, there are some other Vlasov-type equations displaying singular asymptotic regimes with respect to a small parameter. 
\begin{itemize}
\item A famous example is the quasineutral limit of the Vlasov-Poisson system, corresponding to a regime where the ratio of the Debye length over the typical observation scale is small. In the first part of the important work \cite{BreVP2}, Brenier shows that the solutions of the system have a monokinetic behavior with a velocity solution to the incompressible Euler equations, provided that the initial distribution also converges towards a Dirac mass in velocity. This link between the Vlasov-Poisson system and equations from fluid dynamics (see e.g. \cite{BreVP1}) is still a very active field of research. We refer to \cite{Mas, HK-quasin, HKI} for further results in that direction and to \cite{HK-HDR} for related works.
\item Another interesting example of such monokinetic behavior is given by the kinetic Cucker-Smale equation describing flocking dynamics without Brownian noise. It has been shown in \cite{FK} (see also \cite{CC}) that the solutions to this system converge to a monokinetic distribution with associated density and velocity satisfying a pressureless Euler system with nonlocal flocking dissipation.
\item We also refer to some recent works about the spatially-extended FitzHugh-Nagumo system, which is a mean-field kinetic model describing a neural network as the number of neurons goes to infinity. In \cite{FN1,FN2}, the authors consider a regime of strong local interaction between neurons, which asymptotically leads to a somewhat monokinetic distribution in some of the variables. At the limit, one obtains a macroscopic model which is a reaction-diffusion system (see also \cite{BlauFil}). This hydrodynamic limit shares some similarities with the one described for the kinetic Cucker-Smale model.
\end{itemize}
\subsection{General strategy of proof}\label{subsection:Strat}
We now describe the guiding lines for the proof of Theorems \ref{thCV}--\ref{th-rate:cvgence}. For the sake of conciseness, the ideas we will present are partly formal and we refer to the related sections of this article for rigorous details.

Owing to the analysis of \cite{HKM}, one shall expect that the following convergences should hold, at least in a weak sense:
\begin{align*}
u_\eps(t) \underset{\eps \rightarrow 0}{\longrightarrow} u(t), \ \ \ f_\eps(t) \underset{\eps \rightarrow 0}{\longrightarrow} \rho(t) \otimes \delta_{v=u(t)-e_3},
\end{align*}
where $(\rho,u)$ is a solution to the Boussinesq-Navier-Stokes system \eqref{BNS}. Such convergences will make the formal analysis of Subsection \ref{Subsec:formallim} rigorous and will lead to the result of Theorem \ref{thCV}. Again, the kinetic equation handled here is not of Fokker-Planck type so that the framework we consider here is different from the one studied in \cite{goudon2004hydrodynamic1, goudon2004hydrodynamic2}. 

In view of the expected previous singular limit, we are thus looking for uniform bounds in $\eps$ for $\rho_\eps$ and $u_\eps$. From the Vlasov-Navier-Stokes system, we observe that $(\rho_\eps, u_\eps)$ satisfies the following system
 \begin{equation}\label{systemcons:topasstolim}
        \left\{\begin{aligned}
            &\partial_t\rho_\eps+\mathrm{div}_x \, j_\eps=0,\\
            &\partial_t u_\eps+(u_\eps\cdot\nabla_x)u_\eps-\Delta_x u_\eps+\nabla_x p_\eps=j_\eps-\rho_\eps u_\eps,\\
            &\mathrm{div}_x \, u_\eps=0,
        \end{aligned}\right.
    \end{equation}
therefore a weak compactness argument shall enable us to pass to the limit when $\eps \rightarrow 0$ and to recover \eqref{BNS}. 

Classical energy estimates for the Navier-Stokes equations should provide uniform bounds for $u_\eps$, at least in the energy space where the global Leray solutions belong. 
However, uniform bounds on $\rho_\eps$ are not directly given by \eqref{systemcons:topasstolim} and this constitutes one of the main obstacles of the analysis. In addition, one has to look at the convergence of $j_\eps$ when $\eps \rightarrow 0$, which is expected to be towards $\rho(u-e_3)$. These convergences will allow us to pass to the limit in the source term (i.e. the Brinkman force in the Navier-Stokes equations) and in the first equation of \eqref{systemcons:topasstolim}.

Hence, inspired by the strategy performed by Han-Kwan and Michel in \cite{HKM}, we aim at obtaining for any $T>0$
\begin{itemize}
\item some uniform bounds on the local density $\rho_\eps$ in $\Ld^{\infty}(0,T;\Ld^{\infty}(\R^3_+))$. This issue was already at stake in \cite{HKMM, HK, EHKM} for the study of the large time behavior of the system in different spatial contexts.
\item some results of convergence of $F_\eps:=j_\eps -\rho_\eps u_\eps$ in $\Ld^{2}(0,T;\Ld^{2}(\R^3_+))$ when $\eps \rightarrow 0$.
\end{itemize}

\subsubsection{A Lagrangian framework}
To do so, we rely on the \textit{Lagrangian structure} of the Vlasov equation satisfied by $f_\eps$, introducing the characteristic curves $s \mapsto \left( \mathrm{X}_{\varepsilon}(s;t,x,v),\mathrm{V}_{\varepsilon}(s;t,x,v) \right)$ starting at $(x,v)$ at time $t$ and associated to this equation, namely
\begin{equation}\label{strat:EDO}
\left\{
      \begin{aligned}
        \dot{\mathrm{X}}_{\varepsilon}(s;t,x,v) &=\mathrm{V}_{\varepsilon}(s;t,x,v),\\[2mm]
\dot{\mathrm{V}}_{\varepsilon}(s;t,x,v)&=\frac{1}{\eps} \Big( u_{\eps}(s,\mathrm{X}_{\varepsilon}(s;t,x,v))-e_3 -\mathrm{V}_{\varepsilon}(s;t,x,v)\Big),\\
	\mathrm{X}_{\varepsilon}(t;t,x,v)&=x,\\
	\mathrm{V}_{\varepsilon}(t;t,x,v)&=v,
      \end{aligned}
    \right.
\end{equation}
where $u_\eps(s,\cdot)$ has been by extended $0$ outside the half-space $\R^3_+$. From the method of characteristics for the Vlasov equation, we can infer a key representation formula for the solution $f_\eps$. In view of the absorption boundary condition \eqref{bcond-f} satisfied by $f_\eps(t)$, this function will vanish on the points $(x,v)$ of the phase space such that the trajectory $\sigma \mapsto \mathrm{X}_\eps(\sigma;t,x,v)$ has left the half-space on $[0,t]$.

Thanks to the formulas
\begin{align*}
 \rho_\eps(t,x)= \int_{\R^3} f_\eps(t,x,v) \, \mathrm{d}v, \ \ F_\eps(t,x)&=\int_{\R^3} f_\eps(t,x,v) \left( v-u_\eps(t,x)\right) \, \mathrm{d}v,
\end{align*}
the representation of $f_\eps$ in terms of the characteristic curves will be the starting point towards the desired bound on $\rho_\eps$ and the convergence of $F_\eps$ when $\eps \rightarrow 0$.

We will first perform a change of variable in velocity of the form $$v \mapsto \mathrm{V}_{\varepsilon}(0;t,x,v)$$ in the previous integrals. This procedure will be allowed and will yield a control on $\rho_\eps$, provided that
\begin{align}\label{strat:boundNablatohave}
 \int_0^T \Vert \nabla_x u_\eps(s) \Vert_{\Ld^{\infty}(\R^3_+)} \, \mathrm{d}s \ll 1.
\end{align}
Note that this idea, reminiscent of the work of Bardos and Degond in \cite{BD}, has recently provided many results on the large time behavior of the VNS system \cite{HKMM, HK, EHKM, E} as well as on its high friction limits \cite{HoferInertia, HKM}. Note that obtaining a control like \eqref{strat:boundNablatohave} indicates that some decay of $u_\eps$ seems to be required.

A more careful study of the Brinkman force $F_\eps$ (still based on the trajectories and the previous change of variable, following \cite{HKM}) will show that an additional control similar to \eqref{strat:boundNablatohave} actually ensures the convergence of $F_\eps +\rho_\eps e_3$ towards $0$ when $\eps \rightarrow 0$.
This sufficient control will be satisfied if one obtains \textit{some decay in time of $u_\eps$}.
\subsubsection{Towards decay in time estimates}
As the energy inequality \eqref{ineq:energy1} includes the presence of a gravity term, this only paves the way for a short time result with respect to \eqref{strat:boundNablatohave}. This observation has already been made by Han-Kwan and Michel in \cite{HKM} where the question of the global in time derivation of \eqref{BNS} was left as an open problem (on the whole space).

In the spirit of \cite{E}, we first look for a \textit{conditional decay in time} result for $u_\eps$, which is a somewhat general property of the Navier-Stokes system and which basically requires some decay of the Brinkman force $F_\eps=j_\eps-\rho_\eps u_\eps$. Following Wiegner \cite{Wieg} and Borchers and Miyakawa \cite{BM}, one can prove that the polynomial decay
\begin{align}\label{strat:decay-u}
\Vert u_\eps(s) \Vert_{\Ld^{2}(\R^3_+)} \lesssim \dfrac{1}{(1+s)^{\alpha}}
\end{align}
holds for any $\alpha \in [0, 3/4]$, on a time interval where the source term $F_\eps=j_\eps-\rho_\eps u_\eps$ in the Navier-Stokes equations satisfies the following pointwise decay estimates:
\begin{align}\label{strat:decayF}
\Vert F_\eps(s) \Vert_{\Ld^{2}(\R^3_+)} \lesssim \dfrac{1}{(1+s)^{\frac{7}{4}}}.
\end{align}
Ensuring \eqref{strat:decayF} appears as one the main goals of the analysis\footnote{We also refer to Remark \ref{rmq:conddecay} for a discussion about the possible use of the improved energy-dissipation inequality \eqref{ineq:energy2} including the potential energy.
}.

To obtain \eqref{strat:boundNablatohave} with $T=+\infty$, we will rely on a \textit{bootstrap argument}. The main idea is to interpolate higher order energy estimates for $u_\eps$ with pointwise estimates on its $\Ld^2_x$ norm. Such higher order estimates involving $\D^2_x u_\eps$ are handled thanks to the maximal regularity theory for the unsteady Stokes system. Because of the slow polynomial decay $\eqref{strat:decay-u}$ of the $\Ld^2_x$ norm of $u_\eps$, obtaining integrability results for large time is not directly possible. To overcome this issue, we will look for polynomial weighted in time estimates, following the techniques of \cite{HK,E}.

Roughly speaking, this whole procedure mainly amounts in controlling the Brinkman force $F_\eps=j_\eps-\rho_\eps u_\eps$
with some uniform in $\eps$ bounds and with some decay in time estimates (entailing in particular \eqref{strat:decayF}).

\subsubsection{The use of the absorption at the boundary}

To go further in the analysis of the Brinkman force, we rely on the so-called \textit{exit geometric condition}, defined for the Vlasov equation. This crucial notion stems from the work of Glass, Han-Kwan and Moussa in \cite{GHKM}. It has been recently used in \cite{E} for the study of the large time dynamics of the VNS system on the half-space with gravity, following ideas similar to the one we will enforce here. Let us emphasize the very strong influence of the \textit{geometric control condition} introduced by Bardos, Lebeau and Rauch in their celebrated work \cite{BLR} on the wave equation: the introduction of the exit geometric condition in our context is really reminiscent of this controllability result.

Roughly speaking, this condition referred to as EGC asks for the particles trajectory starting from a compact set to leave the half-space before a fixed time. More precisely, this compact refers to a product of $[0,L]$ in the third spatial direction in $\R^3_+$ by a ball $\overline{\B}(0,R)$ in velocity. This condition is of course related to the vector field $u_\eps$ defining the solutions to the system \eqref{strat:EDO}. The main consequence of an exit geometric condition satisfied by $u_\eps$, in a time $T$ and with respect to $\big( \R^2 \times [0,L] \big) \times \overline{\B}(0,R)$, is that for all $t>T$, we have
\begin{align*}
 \, \mathbf{1}_{\mathrm{X}_{\varepsilon}(0;t,x,v)_3 \leq L} \, \mathbf{1}_{\vert \mathrm{V}_{\varepsilon}(0;t,x,v)\vert \leq R } \,  f_\eps(t,x,v)=0.
\end{align*}
\begin{figure}[!h]
\centering
\begin{tikzpicture}[scale=0.85]
\draw[very thick, ->] (-7,0) -- (8, 0) node[below]{$x_3=0$};
\draw[very thick, ->] (0, 0) -- (0, 5) node[left]{$x_3$};
\draw[dashed] (-7,2.5) -- (8, 2.5)node[above]{$x_3=L$};
\draw[->,red,thick] (3.5,1.5)node[left]{$t=0$} to [out=50,in=310] (4.5,4)node[right]{$\mathscr{T}_2$};
\draw[->,red,thick] (4.5,4)to [out=140,in=90](2,3);
\draw[->,red,thick] (2,3) to [out=270,in=270] (4,3.5);
\draw[->>,red,thick] (4,3.5) to [out=100,in=60] (3.25,3.25)node[left]{$t>T$}; 
\draw[->,blue,thick] (-2,1)node[left]{$t=0$} to [out=50,in=310] (-0.75,4) ;
\draw[->,blue,thick] (-0.75,4) to [out=140,in=90](-5,3)node[right]{$\mathscr{T}_1$};
\draw[->>,blue,thick] (-5,3) to [out=270,in=90] (-3.2,0)node[below]{$t<T$};

\draw[->,blue,thick] (1,1.75)node[above]{$t=0$} to [out=255,in=75] (0.5,1)node[right]{$\mathscr{T}_3$};
\draw[->,blue,thick] (0.5,1) to [out=270,in=220](1.5,0.75);
\draw[->>,blue,thick](1.5,0.75)to [out=40,in=100] (2.5,0)node[below]{$t<T$};

\node[draw,rounded corners,fill=gray!50] (P) at (6.5,0.8) {Phase space $\left(\R^2 \times (0,L) \right) \times \mathrm{B}_v(0,R)$ };
\draw[->,black,thick](-5.5,4.5)node[right]{$-e_3$} to (-5.5,4);
\end{tikzpicture}
\caption{EGC not satisfied in time $T$ w.r.t $\left(\R^2 \times (0,L) \right) \times \mathrm{B}_v(0,R)$
(traj. $\mathscr{T}_2$ not absorbed before $t=T$)}
\label{fig:exemple_tikz}
\end{figure}

This comes from the absorption condition \eqref{bcond-f} satisfied by $f_\eps$ at the boundary of the half-space. In short, this cancellation of the distribution function paves the way for the obtention of decay estimates.

Of course, a major issue is to ensure such a condition and to propagate it throughout the evolution of the system. One can find on Figure \ref{fig:exemple_tikz} an example of a situation where the exit geometric condition in some time $T$ is not satisfied. To ensure it, we will rely on a comparison principle for the system \eqref{strat:EDO}, by looking at the free evolution of the particles still undergoing the gravity effect but without the influence of the fluid (i.e. with $u_\eps \equiv 0$). The analysis of this modified system of ODEs can be adressed explicitly, the trajectories being more or less straight lines in the physical space. This method, in the spirit of the one devised in \cite{E}, shows that an exit geometric condition is available at any time to \eqref{strat:EDO}, provided that 
\begin{align}\label{strat:boundVelocityohave}
 \int_0^{+ \infty} \Vert  u_\eps(s) \Vert_{\Ld^{\infty}(\R^3_+)} \, \mathrm{d}s \ll 1.
\end{align}
This constitutes the key condition to use the geometry of the problem through absorption at the boundary. Let us point out that the presence of the (linear) gravity term in the VNS is \textit{required} to perform such a strategy.

\medskip

Building on the previous ideas, one can hope for some decay of the Brinkman force $F_\eps$. Our method will follow the one introduced in \cite{HKM} and is based on a careful splitting and desingularization (in $\eps$) of $F_\eps$\footnote{While studying the large time behavior of the VNS system with gravity in $\R^3_+$ in \cite{E} (corresponding to $\eps=1$), it is possible to deal independently with each moment $\int_{\R^3} \vert v \vert^k f_\eps(t,x,v) \, \mathrm{d}v$. In particular, it turns out to be sufficient to treat the moments of order $0$ and $1$ to directly get some control on the Brinkman force $j_f -\rho_f u$, because there is no singularity in the expression of the reverse velocity curve.}. 
Combined with the exit geometric condition, obtaining some decay in time estimates should be possible if the initial distribution function $f_\eps^0$ itself enjoys some decay in the phase space. This mainly explains the required mixed-moment type assumption \ref{hypUnifBoundVNS1}. We refer to Section \ref{Section:estimateBrinkman} for more details about this procedure.

\bigskip

All in all, the previous analysis mainly shows that it comes down to obtain
\begin{align}\label{strat:boundLIPtohave}
 \int_0^{+ \infty} \Vert  u_\eps(s) \Vert_{\W^{1,\infty}(\R^3_+)} \, \mathrm{d}s \ll 1,
\end{align}
at least for $\eps$ small enough. We shall follow a bootstrap strategy to ensure such a control. We will rely on the smallness condition \eqref{hypSmallDataTech} to initialize the procedure, as well as to propagate all the estimates until any time.
\subsection{Outline of the paper}
The paper is organized as follows. 

\begin{enumerate}
\item[$-$] In \textbf{Section \ref{Section:Trajectoires}}, we derive some useful information on the system through its Lagrangian structure. 
We define the key concept of exit geometric condition and provide some of its useful properties. We also show how the control \eqref{strat:boundLIPtohave} actually ensures both the change of variable in velocity and the propagation of the exit geometric condition. Furthermore, we provide a first splitting of the Brinkman force leading to its conditional convergence when $\eps$ tends to $0$. This requires a first careful desingularization of this term with respect to $\eps \rightarrow 0$.
\item[$-$] In \textbf{Section \ref{Section:Prelim}}, we collect several preliminar regularity results and estimates on global weak solutions to the VNS system, as well as a conditional theorem of convergence.  First, we derive an improvement of the energy-dissipation inequality satisfied by weak solutions, by considering the total energy \eqref{def:Energy} of the system. Next, we state the aforementioned conditional polynomial decay of the fluid kinetic energy, which is valid provided that the source term itself enjoys some pointwise decay. This enables us to determine sufficient conditions leading to a proof of Theorem \ref{thCV} and which highlight the need of decay in time estimates (see Proposition \ref{IF:Propo}). We also define the notion of \textit{strong existence time} for the Vlasov-Navier-Stokes system, which allows to consider higher order energy estimates for the fluid velocity, provided that the initial data and the source term in the Navier-Stokes equations are small enough.
The previous discussion on our strategy then leads to a bootstrap procedure which consists in obtaining the control \eqref{strat:boundLIPtohave} for $\eps$ small enough.
\item[$-$] \textbf{Section \ref{Section:estimateBrinkman}} provides a family of fine estimates on the Brinkman force, based on the same decomposition as in Section \ref{Section:Trajectoires}. We aim at obtaining a pointwise decay in time of the $\Ld^2_x$ norm of this force (see \eqref{strat:decayF}), as well as some weighted in time bounds in $\Ld^p_t \Ld^p_x$. These polynomial in time estimates are built by propagating the exit geometric condition. Uniform bounds (independent of time and of $\eps$) are obtained thanks to Assumption \textbf{\ref{hypUnifBoundVNS}} and by relying on the conditional polynomial decay of $u_\eps$. 
\item[$-$] The bootstrap argument takes place in \textbf{Section \ref{Section:Boostrap}}. A careful interpolation procedure combined with weigthed in time estimates allows to prove that \eqref{strat:boundLIPtohave} holds for $\eps$ small enough. Passing to the limit in the VNS system is then a direct consequence of Proposition \ref{IF:Propo} and of all the previous uniform estimates on $u_\eps$ and $\rho_\eps$ which will be valid on $\R^+$. This will entail Theorem \ref{thCV} and the Boussinesq-Navier-Stokes system will thus be recovered at any time. The quantitave rates of (strong) convergence announced in Theorem \ref{th-rate:cvgence} are then obtained by performing direct energy estimates.
\end{enumerate}

\medskip

\begin{remarque}
In the rest of the article, we shall state many properties which hold true only for $\eps$ small enough: we will refer to some range of $\eps \in (0, \eps_0)$ where $\eps_0 >0$ may change from one statement to another and can be reduced if necessary. Furthermore, we will always refer to $M>1$ as a constant involved in estimates which are uniform with respect to the parameter $\eps$ and bearing on the initial data (see Assumption \textbf{\ref{hypUnifBoundVNS}}).
\end{remarque}

\section{Particle trajectories}\label{Section:Trajectoires}

\subsection{Lagrangian structure for the Vlasov equation}\label{Subsec:EDO}
We first define the characteristic curves associated to Vlasov equation, which is a transport equation in the phase-space $\R^3_+ \times \R^3$. This provides a useful representation formula for the distribution function $f_\eps$ which takes into account the absorption boundary condition \eqref{bcond-f}.

Let $u_\eps: \R^+ \times \R^3_+ \rightarrow \R^3$ be a given time-dependent vector field, with $\eps >0$ fixed. For $t \geq 0$ and $(x,v) \in \R^3 \times \R^3$, we consider the solution $s \mapsto \big( \mathrm{X}_{\varepsilon}(s;t,x,v), \mathrm{V}_{\varepsilon}(s;t,x,v) \big)$ to the following system of ordinary differential equations:
\begin{equation}\label{EDO-charac}
\left\{
      \begin{aligned}
        \dot{\mathrm{X}}_{\varepsilon}(s;t,x,v) &=\mathrm{V}_{\varepsilon}(s;t,x,v),\\[2mm]
\dot{\mathrm{V}}_{\varepsilon}( (s;t,x,v)&=\frac{1}{\eps} \Big( (Pu_{\eps})(s,\mathrm{X}_{\varepsilon}(s;t,x,v))-e_3 -\mathrm{V}_{\varepsilon}(s;t,x,v)\Big),\\
	\mathrm{X}_{\varepsilon}(t;t,x,v)&=x,\\
	\mathrm{V}_{\varepsilon}(t;t,x,v)&=v,
      \end{aligned}
    \right.
\end{equation}
where the dot means derivative along the first variable. Here, the operator $P$ refers to the extension operator by $0$ outside the half-space, such that
\begin{align*}
P : \W^{1,\infty}_0(\R^3_+) \longrightarrow \W^{1,\infty}(\R^3)
\end{align*}
is bounded. In what follows, we shall use the harmless notation
\begin{align*}
\forall t \geq 0, \ \ (Pu_\eps)(t,\cdot)=P(u_\eps(t,\cdot)),
\end{align*}
because the operator $P$ does not act on the time variable. When there will be no ambiguity, the curves $(\X_\eps, \V_\eps)$ will always refer to solutions to \eqref{EDO-charac}, that is, associated to a velocity field $u_\eps$ solution to the Navier-Stokes equations in the VNS system.

\begin{proposition}\label{prop:ChangeLAGRANGIAN}
Let $\eps>0$. Suppose that $u_\eps \in \Ld^1_{\mathrm{loc}}(\R^+;\mathrm{W}^{1, \infty}(\R^3_+))$. Then for any $s,t \geq0$, the mapping 
\begin{align*}
(x,v) \mapsto \big( \mathrm{X}_{\varepsilon}(s;t,x,v), \mathrm{V}_{\varepsilon}(s;t,x,v) \big)
\end{align*}
is a diffeomorphism of $\R^6$, whose Jacobian value is $e^{\frac{3(t-s)}{\eps}}$ and whose inverse is
\begin{align*}
(x,v) \mapsto \big( \mathrm{X}_{\varepsilon}(t;s,x,v), \mathrm{V}_{\varepsilon}(t;s,x,v) \big).
\end{align*}
\end{proposition}

A solution to the previous system satisfies for all $s,t \geq 0$ and $(x,v) \in \R^3 \times \R^3$
\begin{equation}\label{expr:Zt}
\left\{
      \begin{aligned}
        \mathrm{X}_{\varepsilon}(s;t,x,v)&=x+\eps(1-e^{\frac{t-s}{\eps}})v-\Big(s-t+\eps e^{\frac{t-s}{\eps}}-\eps\Big)e_3 +\int_t^s (1-e^{\frac{\tau-s}{\eps}}) (Pu_{\eps})(\tau,\mathrm{X}_{\eps}(\tau;t,x,v))  \, \mathrm{d}\tau,  \\[2mm]
        \mathrm{V}_{\varepsilon}(s;t,x,v)&= e^{\frac{t-s}{\eps}}v-(1-e^{\frac{t-s}{\eps}})e_3+\dfrac{1}{\eps}\int_t ^s e^{\frac{\tau-s}{\eps}}(Pu_{\eps})(\tau,\mathrm{X}_{\eps}(\tau;t,x,v))   \, \mathrm{d}\tau.   \\
      \end{aligned}
    \right.
\end{equation}
In order to take into account the absorption boundary condition \eqref{bcond-f} that must be satisfied by the distribution function $f_\eps$, we introduce the following backward and forward exit times.
\begin{definition}\label{def:tau}
For $(x,v) \in \R^3_+ \times \R^3$ and for any $t \geq 0$, we set
\begin{align}
\tau^{+}_\eps(t,x,v)&:=\sup \left\lbrace s \geq t \ \mid \forall \sigma \in [t,s], \  \mathrm{X}_\eps(\sigma;t,x,v) \in \R^3_+ \right \rbrace. \label{def:tau+}
\end{align}
We also define
\begin{align*}
\mathcal{O}_\eps^t:=\left\lbrace (x,v) \in \R^3_+ \times \R^3 \mid \forall \sigma \in [0,t], \ \mathrm{X}_\eps(\sigma;t,x,v) \in \R^3_+\right\rbrace.
\end{align*}
\end{definition}

We can now state the following proposition giving a crucial representation formula for the distribution function $f_\eps$. We refer to \cite[Appendix A.2]{EHKM} for a proof.
\begin{proposition}
Let $\eps >0$. If $(u_\eps, f_\eps)$ is a weak solution to the Vlasov-Navier-Stokes system in the sense of Definition \ref{weak-sol} with $u_\eps \in \Ld^1_{\mathrm{loc}}(\R^+;\mathrm{W}^{1, \infty}(\R^3_+))$ then
\begin{align}\label{formule:rep}
f_{\eps}(t,x,v)=e^{\frac{3t}{\eps}} \, \mathbf{1}_{\mathcal{O}^t_{\eps}}(x,v) \, f_\eps^0(\mathrm{X}_{\varepsilon}(0;t,x,v),\mathrm{V}_{\varepsilon}(0;t,x,v)),
\end{align}
where $f_\eps^0$ refers to the admissible initial condition for the Vlasov equation in the sense of Definition \ref{CIadmissible}.
\end{proposition}
For later purposes, note that we have the following identity
\begin{align*}
(\X_\eps,\V_\eps)(0;t,\mathcal{O}_\eps^t)=\left\lbrace  (x,v) \in \R^3_+ \times \R^3  \mid  \tau^{+}_\eps(0,x,v) >t  \right\rbrace,
\end{align*}
therefore for any integrable function $\psi$ on $\R^3_+ \times \R^3$, we can write
\begin{align*}
\int_{\R^3_+ \times \R^3} \psi(x,v) f_{\eps}(t,x,v) \, \mathrm{d}x \, \mathrm{d}v&= \int_{\mathcal{O}_\eps^t} \psi(x,v) e^{\frac{3t}{\eps}}f_\eps^0(\mathrm{X}_{\varepsilon}(0;t,x,v),\mathrm{V}_{\varepsilon}(0;t,x,v)) \, \mathrm{d}x \, \mathrm{d}v \\
&=\int_{\R^3_+ \times \R^3} \psi(\mathrm{X}_{\varepsilon}(t;0,x,v),\mathrm{V}_{\varepsilon}(t;0,x,v)) \, \mathbf{1}_{\tau^{+}_\eps(0,x,v) >t} \, f_\eps^0(x,v) \, \mathrm{d}x \, \mathrm{d}v.
\end{align*}

\subsection{Changes of variable in velocity and space}\label{Subsec:Changevar}
Next, we derive some sufficient conditions on the velocity field $u_\eps$ so that a change of variable with respect to the characteristic curve in velocity holds true. As explained in the introduction, this straightening change of variable has been intensively used for the study of the large time behavior of the system (see \cite{HKMM, HK, EHKM, E}), as well as for its hydrodynamic limits in \cite{HoferInertia,HKM}. In short, this change of variable is valid until time $T$ if a control of the type $ \Vert \nabla_x u_\eps \Vert_{\Ld^1(0,T;\Ld^{\infty}(\R^3_+))} \ll 1$ can be ensured. Thanks to the previous representation formula, this will mainly entail a bound on the local density $\rho_\eps$ of the type 
\begin{align*}
\Vert \rho_\eps \Vert_{\Ld^{\infty}(0,T;\Ld^{\infty}(\R^3_+))} \lesssim \Vert f_\eps^0 \Vert_{\Ld^{1}(\R^3;\Ld^{\infty}(\R^3_+))}.
\end{align*}
As we shall see later on, a change of variable in space will also be performed after the change of variable in velocity. The same control on $\nabla_x u_\eps$ actually also makes such a change of variable admissible.

The first change of variable is the straightening change of variable in velocity.
\begin{lemme}\label{LM:ChgmtVAR-V}
Let $\eps>0$. Assume that there exists some time $t>0$ such that
\begin{align}\label{control:Nablau}
\int_0^t \Vert \nabla_x u_\eps (s) \Vert_{\Ld^{\infty}(\R^3_+)} \, \mathrm{d}s < \delta,
\end{align}
where $\delta>0$ satisfies $\delta e^{\delta}<\frac{1}{9}$. Then for any $x \in \R^3$, the mapping
\begin{align*}
\Gamma_\eps^{t,x}:v \mapsto \mathrm{V}_\eps(0;t,x,v)
\end{align*}
is a $\mathscr{C}^1$-diffeomorphism from $\R^3$ to itself with the bound
\begin{align*}
\forall v \in \R^3, \ \ \det \, ( \mathrm{D}_v \, \Gamma_\eps^{t,x}(v))  \gtrsim e^{\frac{3t}{\eps}},
\end{align*}
where $\gtrsim$ refers to a universal constant.
\end{lemme}
\begin{proof}
We refer to \cite[Lemma 4.5]{E}, having in mind
\begin{align*}
\mathrm{V}_{\varepsilon}(0;t,x,v)&= e^{\frac{t}{\eps}}v-(1-e^{\frac{t}{\eps}})e_3-\dfrac{1}{\eps}\int_0^t e^{\frac{\tau}{\eps}}(Pu_{\eps})(\tau,\mathrm{X}_{\eps}(\tau;t,x,v))   \, \mathrm{d}\tau,
\end{align*}
which is directly inferred from \eqref{expr:Zt}.
\end{proof}
We can now present the first main consequence of the change of variable in velocity, namely a uniform bound for $\rho_\eps$ in $\Ld_t^{\infty} \Ld^{\infty}_x$.
\begin{corollaire}\label{coro:boundrho}
Let $\eps>0$. Assume that there exists some time $t>0$ such that the condition \eqref{control:Nablau} holds. Then we have
\begin{align*}
\Vert \rho_\eps \Vert_{\Ld^{\infty}(0,t;\Ld^{\infty}(\R^3_+)} \lesssim \Vert f_\eps^0 \Vert_{\Ld^1(\R^3; \Ld^{\infty}(\R^3_+))}.
\end{align*}
\begin{proof}
We basically follow \cite[Lemma 3.2]{HK}: we use the representation formula \eqref{formule:rep} to write
\begin{align*}
 \rho_\eps(t,x)= \int_{\R^3} f_\eps(t,x,v) \, \mathrm{d}v &= \int_{\R^3} e^{\frac{3t}{\eps}} \, \mathbf{1}_{\mathcal{O}^t_{\eps}}(x,v) \, f_\eps^0(\mathrm{X}_{\varepsilon}(0;t,x,v),\mathrm{V}_{\varepsilon}(0;t,x,v)) \, \mathrm{d}v \\
 & =\int_{\R^3}  e^{\frac{3t}{\eps}} \, \mathbf{1}_{\mathcal{O}^t_{\eps}}(x,[\Gamma_\eps^{t,x}]^{-1}(w)) \, f_\eps^0(\mathrm{X}_{\varepsilon}(0;t,x,[\Gamma_\eps^{t,x}]^{-1}(w)),w) \vert \det \,  \mathrm{D}_w \, [\Gamma_\eps^{t,x}]^{-1}(w) \vert \, \mathrm{d}w \\
 & \lesssim \int_{\R^3} \, \mathbf{1}_{\mathcal{O}^t_{\eps}}(x,[\Gamma_\eps^{t,x}]^{-1}(w)) \, f_\eps^0(\mathrm{X}_{\varepsilon}(0;t,x,[\Gamma_\eps^{t,x}]^{-1}(w)),w) \, \mathrm{d}w,
\end{align*}
thanks to the change of variable $v \mapsto  \Gamma_\eps^{t,x}(v)$ and the bound of Lemma \ref{LM:ChgmtVAR-V}. Since $$(x, [\Gamma_\eps^{t,x}]^{-1}(w)) \in \mathcal{O}^t_{\eps}  \Longrightarrow  \X_\eps(0;t,x,[\Gamma_\eps^{t,x}]^{-1}(w)) \in \R^3_+,$$ we obtain the bound
\begin{align*}
 \rho_\eps(t,x) & \lesssim \int_{\R^3} \, \mathbf{1}_{\mathcal{O}^t_{\eps}}(x,[\Gamma_\eps^{t,x}]^{-1}(w)) \Vert \, f_\eps^0(\cdot,w) \Vert_{\Ld^{\infty}(\R^3_+)} \, \mathrm{d}w \leq \Vert f_\eps^0 \Vert_{\Ld^1(\R^3; \Ld^{\infty}(\R^3_+))},
\end{align*}
which is the claimed result.
\end{proof}
\end{corollaire}
\begin{remarque}\label{RmqExGamma-1}
From the expression \eqref{expr:Zt}, we can also deduce the following important formula
\begin{align*}
[\Gamma_\eps^{t,x}]^{-1}(w)&=e^{-\frac{t}{\eps}}w-(1-e^{-\frac{t}{\eps}})e_3+\dfrac{1}{\eps}\int_0 ^t e^{\frac{\tau-t}{\eps}}(Pu_{\eps})(\tau,\mathrm{X}_{\eps}(\tau;t,x,[\Gamma_\eps^{t,x}]^{-1}(w)))   \, \mathrm{d}\tau.
\end{align*}
Note that the same strategy as in Corollary \ref{coro:boundrho} cannot be applied for higher-order moments in velocity because this will make some $[\Gamma_\eps^{t,x}]^{-1}(w)$ appear, which diverges with $\eps$. In Subsection \ref{Subsec:SplitBrink} and Section \ref{Section:estimateBrinkman}, we will perform a desingularization procedure, based on the fine structure of the Brinkman force, to get rid of this problem.
\end{remarque}

\begin{notation}\label{Notation:courbechange}
For $\eps>0$, $s,t \geq 0$ and $(x,w) \in \R^6$, we set
\begin{align*}
\widetilde{\mathrm{X}}_\eps^{s;t}(x,w)&:=\mathrm{X}_{\eps}(s;t,x,[\Gamma_\eps^{t,x}]^{-1}(w)).
\end{align*}
When it is necessary, we shall also use the notation $\widetilde{\mathrm{X}}_\eps^{s;t,w}(x):=\widetilde{\mathrm{X}}_\eps^{s;t}(x,w)$.
\end{notation}

We finally state a lemma about a change of variable in space along the trajectories. It will be applied after the previous change of variable in velocity and will be useful when considering $\Ld^p$ norm in space of the Brinkman force. We refer to \cite[Lemma 4.7]{E}.
\begin{lemme}\label{LM:ChgmtVAR-X}
Let $\eps>0$. Assume that there exists some time $t>0$ such that the condition \eqref{control:Nablau} holds. Then for any $0 \leq s \leq t$ and any $w \in \R^3$, the mapping
\begin{align*}
x \mapsto \widetilde{\mathrm{X}}_\eps^{s;t,w}(x)
\end{align*}
is a $\mathscr{C}^1$-diffeomorphism from $\R^3$ to itself and its Jacobian determinant satisfies the following bound from below
\begin{align}\label{bound:DLambda}
\forall x \in \R^3, \ \ \det \, ( \mathrm{D}_x \, \widetilde{\mathrm{X}}_\eps^{s;t,w}(x))  \gtrsim 1,
\end{align}
where $ \gtrsim$ refers to a universal constant.
\end{lemme}

\medskip

\subsection{Towards the convergence of the Brinkman force when $\varepsilon \rightarrow 0$}\label{Subsec:SplitBrink}
Recall the notation
\begin{align*}
F_\eps:=j_\eps-\rho_\eps u_\eps= \int_{\R^3} f_\eps (v-u_\eps) \, \mathrm{d}v
\end{align*}
for the Brinkman force. The main goal of this subsection is to identify sufficient conditions leading to the convergence of the Brinkman force when $\eps \rightarrow 0$. More precisely, we aim at proving a convergence of $F_\eps + \rho_\eps e_3$ towards $0$ in $\Ld^2_t \Ld^2_x$ when $\eps \rightarrow 0$, based on a careful decomposition of this expression.

To do so, we will rely on the tools of Subsections \ref{Subsec:EDO} and \ref{Subsec:Changevar} which are based on particles trajectories. As announced in Subsection \ref{Subsec:Changevar}, we shall make an intensive use of the following change of variable in velocity:
\begin{align*}
v \mapsto \Gamma_\eps^{t,x}(v)=\mathrm{V}_{\varepsilon}(0;t,x,v),
\end{align*}
which should be combined to the representation formula \eqref{formule:rep}. 

The very first step is to write
\begin{align}\label{StartF_eps}
\begin{split}
F_\eps(t,x)&=e^{\frac{3t}{\eps}}\int_{\R^3}  \mathbf{1}_{\mathcal{O}^t_{\eps}}(x,v)f_\eps^0(\mathrm{X}_{\varepsilon}(0;s,x,v),\mathrm{V}_{\varepsilon}(0;s,x,v)) \left( v-u_\eps(t,x)\right) \, \mathrm{d}v \\
&=e^{\frac{3t}{\eps}} \int_{\R^3} \mathbf{1}_{\mathcal{O}^t_{\eps}}(x,[\Gamma_\eps^{t,x}]^{-1}(w))f_\eps^0(\widetilde{\mathrm{X}}_\eps^{0;t}(x,w),w) \left( [\Gamma_\eps^{t,x}]^{-1}(w)-u_\eps(t,x)\right) \vert \det \, \D_w [\Gamma_\eps^{t,x}]^{-1}(w) \vert \, \mathrm{d}w,
\end{split}
\end{align}
As highlighted before, the crucial quantity
\begin{align*}
[\Gamma_\eps^{t,x}]^{-1}(w)&=e^{-\frac{t}{\eps}}w-(1-e^{-\frac{t}{\eps}})e_3+\dfrac{1}{\eps}\int_0 ^t e^{\frac{\tau-t}{\eps}}(Pu_{\eps})(\tau,\mathrm{X}_{\eps}(\tau;t,x,[\Gamma_\eps^{t,x}]^{-1}(w)))   \, \mathrm{d}\tau
\end{align*}
appearing above is singular with respect to the convergence $\eps \rightarrow 0$. 

\medskip

\textbf{Desingularization procedure with respect to $\eps$.} 

In order to get rid of the factor $\eps^{-1}$, we shall perform an integration by parts in time thanks to the exponential factor (coming from friction in the system) appearing in the last integral. This key idea comes from the strategy devised by Han-Kwan and Michel in \cite{HKM}, already inspired by the work of Han-Kwan in \cite{HK}. 
At least formally, we have for all $(t,x,w) \in \R^+ \times \R^3_+  \times \R^3$
\begin{align*}
[\Gamma_\eps^{t,x}]^{-1}(w)&=e^{-\frac{t}{\eps}}w-(1-e^{-\frac{t}{\eps}})e_3+\dfrac{1}{\eps}\int_0 ^t e^{\frac{\tau-t}{\eps}}(Pu_{\eps})(\tau,\widetilde{\mathrm{X}}_\eps^{\tau;t}(x,w))   \, \mathrm{d}\tau \\
&=e^{-\frac{t}{\eps}}w-(1-e^{-\frac{t}{\eps}})e_3+(P u_\eps)(t, \widetilde{\mathrm{X}}_\eps^{t;t}(x,w))-e^{\frac{t}{\eps}} (P u_\eps)(0,\widetilde{\mathrm{X}}_\eps^{0;t}(x,w)) \\
& \quad -\int_0^t e^{\frac{\tau-t}{\eps}} \partial_\tau [P u_\eps](\tau,\widetilde{\mathrm{X}}_\eps^{\tau;t}(x,w)) \, \mathrm{d}\tau \\
& \quad -\int_0^t e^{\frac{\tau-t}{\eps}}  \left( \mathrm{V}_\eps(\tau;t,x,[\Gamma_\eps^{t,x}]^{-1}(w))\cdot \nabla_x \right)[P u_\eps](\tau,\widetilde{\mathrm{X}}_\eps^{\tau;t}(x,w)) \, \mathrm{d}\tau,
\end{align*}
so that 
\begin{align}\label{GammaDesing}
\begin{split}
[\Gamma_\eps^{t,x}]^{-1}(w)&=e^{-\frac{t}{\eps}} \big( w+e_3-(P u_\eps)(0,\widetilde{\mathrm{X}}_\eps^{0;t}(x,w)) \big)+u_\eps(t,x)-e_3 \\
& \quad -\int_0^t e^{\frac{\tau-t}{\eps}} P [\partial_\tau u_\eps](\tau,\widetilde{\mathrm{X}}_\eps^{\tau;t}(x,w)) \, \mathrm{d}\tau \\
& \quad -\int_0^t e^{\frac{\tau-t}{\eps}}  \left( \mathrm{V}_\eps(\tau;t,x,[\Gamma_\eps^{t,x}]^{-1}(w))\cdot \nabla_x \right)[P u_\eps](\tau,\widetilde{\mathrm{X}}_\eps^{\tau;t}(x,w)) \, \mathrm{d}\tau.
\end{split}
\end{align}
As we shall see, the analysis of the last term will require the following formula, coming from the expression of $[\Gamma_\eps^{t,x}]^{-1}(w)$ and $ \mathrm{V}_{\varepsilon}(s;t,x,v)$ from Subsections \ref{Subsec:EDO}--\ref{Subsec:Changevar}:
\begin{align}\label{eq:CharacVGamma-1}
\begin{split}
 \mathrm{V}_{\varepsilon}(s;t,x,[\Gamma_\eps^{t,x}]^{-1}(w))&=e^{\frac{t-s}{\eps}}[\Gamma_\eps^{t,x}]^{-1}(w)-(1-e^{\frac{t-s}{\eps}})e_3+\dfrac{1}{\eps}\int_t^s e^{\frac{\tau-s}{\eps}}(Pu_{\eps})(\tau,\widetilde{\mathrm{X}}_\eps^{\tau;t}(x,w))   \, \mathrm{d}\tau \\
 &= e^{\frac{-s}{\eps}}(w+e_3)-e_3+\dfrac{1}{\eps}\int_0^s e^{\frac{\tau-s}{\eps}}(Pu_{\eps})(\tau,\widetilde{\mathrm{X}}_\eps^{\tau;t}(x,w))   \, \mathrm{d}\tau.
 \end{split}
\end{align}

\begin{remarque}[Justification of the integration by parts]
The previous computations are still partly formal if we do not assume some additional regularity of $u_\eps$. The intervals of time with which we will systematically work in the following will in fact justify this procedure. We refer to Remark \ref{rmk:justifIPP} for the introduction of the related strong existence times.
\end{remarque}

\medskip

The identity \eqref{GammaDesing} on $[\Gamma_\eps^{t,x}]^{-1}(w)$ combined with \eqref{StartF_eps} then provides the following lemma, which is the starting point for our analysis.
\begin{lemme}\label{Split-F1}
For any $\eps>0$, if $u_\eps$ is smooth and if the change of variable in velocity is admissible until time $T$, then for all $(t,x) \in (0,T) \times \R^3_+$
\begin{align*}
\vert F_\eps(t,x)+ \rho_\eps e_3 \vert \lesssim G_\eps^0(t,x) + G_\eps^1(t,x)+G_\eps^2(t,x),
\end{align*}
where \begin{align*}
 G_\eps^0(t,x) &:= e^{-\frac{t}{\eps}}  \int_{\R^3} \mathbf{1}_{\mathcal{O}^t_{\eps}}(x,[\Gamma_\eps^{t,x}]^{-1}(w))f_\eps^0(\widetilde{\mathrm{X}}_\eps^{0;t}(x,w),w)\Big\vert w+e_3-(P u_\eps)(0,\widetilde{\mathrm{X}}_\eps^{0;t}(x,w))\Big\vert \, \mathrm{d}w,\\
G_\eps^1(t,x)&:= \int_{\R^3}\mathbf{1}_{\mathcal{O}^t_{\eps}}(x,[\Gamma_\eps^{t,x}]^{-1}(w)) \int_0^t  e^{\frac{\tau-t}{\eps}} f_\eps^0(\widetilde{\mathrm{X}}_\eps^{0;t}(x,w),w)\Big\vert \partial_\tau [P u_\eps](\tau,\widetilde{\mathrm{X}}_\eps^{\tau;t}(x,w)) \Big\vert \, \mathrm{d}\tau \, \mathrm{d}w, \\[2mm]
G_\eps^2(t,x)&:= \int_{\R^3}\mathbf{1}_{\mathcal{O}^t_{\eps}}(x,[\Gamma_\eps^{t,x}]^{-1}(w)) \int_0^t  e^{\frac{\tau-t}{\eps}} f_\eps^0(\widetilde{\mathrm{X}}_\eps^{0;t}(x,w),w) \\ 
& \qquad \qquad  \qquad  \qquad  \qquad \qquad \qquad \quad  \times \Big\vert \left(\mathrm{V}_\eps(\tau;t,x,[\Gamma_\eps^{t,x}]^{-1}(w)) \cdot  \nabla_x \right) [P u_\eps](\tau,\widetilde{\mathrm{X}}_\eps^{\tau;t}(x,w))  \Big\vert \, \mathrm{d}\tau \, \mathrm{d}w.
\end{align*}
\end{lemme}

We will estimate the three previous terms separately.
\begin{lemme}\label{LM:G0}
Let $T>0$ such that the condition \eqref{control:Nablau} holds at time $T$. Assume that for all $\eps >0$, we have
\begin{align*}
\Vert  f_\eps^0 \Vert_{\Ld^1(\R^3;\Ld^{\infty}(\R^3_+))} + \Vert (1+\vert v \vert^2) f_\eps^0 \Vert_{\Ld^1(\R^3_+ \times \R^3)} & \leq M, \\
\Vert u_\eps^0 \Vert_{\Ld^2(\R^3_+)} &\leq M,
\end{align*}
for some $M>1$ independent of $\eps$. There exists $\mu_2>0$ such that for all $\eps >0$
\begin{align*}
\Vert G_\eps^0 \Vert_{\Ld^2((0,T) \times \R^3_+)} \lesssim \eps^{\frac{1}{2}}M^{\mu_2}.
\end{align*}
\end{lemme}
\begin{proof}
Thanks to Hölder inequality, we have
\begin{align*}
\int_0^T \int_{\R^3_+}  \vert G_\eps^0(t,x) \vert^2 \, \mathrm{d}t \, \mathrm{d}x
&\leq \int_0^T  e^{-\frac{2t}{\eps}}  \Vert  f_\eps^0 \Vert_{\Ld^1(\R^3;\Ld^{\infty}(\R^3_+))}  \int_{\R^3_+ \times \R^3}  \mathbf{1}_{\mathcal{O}^t_{\eps}}(x,[\Gamma_\eps^{t,x}]^{-1}(w))f_\eps^0(\widetilde{\mathrm{X}}_\eps^{0;t}(x,w),w)\\
&  \quad \quad \quad \quad \quad \quad \quad \quad \quad \times  \Big\vert w+e_3-(P u_\eps)(0,\widetilde{\mathrm{X}}_{\eps}^{0;t}(x,w))\Big\vert^2 \, \mathrm{d}x  \, \mathrm{d}w   \, \mathrm{d}t .
\end{align*}
In the last inequality, we have used the following fact: for all $y \in \R^3$ and $w \in \R^3$ we have
\begin{align}\label{equiv:TrajGood}
\begin{split}
(y, [\Gamma_\eps^{t,y}]^{-1}(w)) \in \mathcal{O}^t_{\eps} \Longrightarrow \forall \sigma \in [0,t], \, \X_\eps(\sigma;t,y,[\Gamma_\eps^{t,y}]^{-1}(w)) \in \R^3_+.
 \end{split}
\end{align}
Then, isolating the integral in space, we perform the change of variable in space $x \mapsto \widetilde{\mathrm{X}}_\eps^{0;t}(x,w)$ which is valid thanks to Lemma \ref{LM:ChgmtVAR-X}, and we get for all $t \in [0,T]$ and $w \in \R^3$
\begin{align*}
&\int_{\R^3_+}\mathbf{1}_{\mathcal{O}^t_{\eps}}(x,[\Gamma_\eps^{t,x}]^{-1}(w))f_\eps^0(\widetilde{\mathrm{X}}_\eps^{0;t}(x,w),w) \Big\vert w+e_3-(P u_\eps)(0,\widetilde{\mathrm{X}}_\eps^{0;t}(x,w))\Big\vert^2 \, \mathrm{d}x \\
&= \int_{\R^3} \mathbf{1}_{\R^3_+}([\widetilde{\mathrm{X}}_\eps^{0;t,w}]^{-1}(x)) \mathbf{1}_{\mathcal{O}^t_{\eps}}([\widetilde{\mathrm{X}}_\eps^{0;t,w}]^{-1}(x),[\Gamma_\eps^{t,[\widetilde{\mathrm{X}}_\eps^{0;t,w}]^{-1}(x)}]^{-1}(w)) \\
& \qquad \qquad \qquad \qquad  \qquad \qquad \qquad \qquad  \qquad \qquad \qquad \qquad \times  f_\eps^0(x,w) \Big\vert w+e_3-(P u_\eps)(0,x)\Big\vert^2 \vert \D_x [\widetilde{\mathrm{X}}_\eps^{0;t,w}]^{-1}(x) \vert \, \mathrm{d}x,
\end{align*}
therefore, thanks to \eqref{bound:DLambda}, we get
\begin{align*}
\int_0^T \int_{\R^3_+}  \vert G_\eps^0(t,x) \vert^2 \, \mathrm{d}t \, \mathrm{d}x &\lesssim \int_0^T  e^{-\frac{2t}{\eps}}  \Vert  f_\eps^0 \Vert_{\Ld^1(\R^3;\Ld^{\infty}(\R^3_+))} \\
& \quad \quad \quad \quad \quad \quad \quad \quad \quad \times \left[\Vert (1+\vert v \vert^2) f_\eps^0 \Vert_{\Ld^1(\R^3_+ \times \R^3)}+ \Vert f_\eps^0 \Vert_{\Ld^1(\R^3;\Ld^{\infty}(\R^3_+))} \Vert u_\eps^0 \Vert_{\Ld^2(\R^3_+)}^2 \right] \, \mathrm{d}t \\
&\leq \eps M^{2} + \eps M^{4},
\end{align*}
and this concludes the proof since $M>1$.
\end{proof}
\begin{lemme}\label{LM:G1}
Let $T>0$ such that the condition \eqref{control:Nablau} holds at time $T$. Assume that for all $\eps >0$
\begin{align*}
\Vert f_\eps^0 \Vert_{\Ld^1(\R^3;\Ld^{\infty}(\R^3_+))} &\leq M,
\end{align*}
for some $M>1$ independent of $\eps$. For all $\eps >0$, we have
\begin{align*}
\Vert G_\eps^1 \Vert_{\Ld^2((0,T) \times \R^3_+)} \lesssim \eps M \Vert\partial_t u_\eps \Vert_{\Ld^2((0,T) \times \R^3_+)}.
\end{align*}
\end{lemme}
\begin{proof}
By the Hölder inequality in space and time, we get
\begin{align*}
\int_0^T \int_{\R^3_+} \vert G_\eps^1(t,x) \vert^2 \, \mathrm{d}t \, \mathrm{d}x
& \leq \eps \Vert f_\eps^0 \Vert_{\Ld^1(\R^3; \Ld^{\infty}(\R^3_+))} \int_0^T \int_{\R^3} \int_0^t  e^{\frac{\tau-t}{\eps}}   \\
& \quad \times \left(  \int_{\R^3_+}  \mathbf{1}_{\mathcal{O}^t_{\eps}}(x,[\Gamma_\eps^{t,x}]^{-1}(w))  f_\eps^0(\widetilde{\mathrm{X}}_\eps^{0;t}(x,w),w)\Big\vert \partial_\tau [P u_\eps](\tau,\widetilde{\mathrm{X}}_\eps^{\tau;t}(x,w)) \Big\vert^2  \, \mathrm{d}x \right) \, \mathrm{d}\tau \, \mathrm{d}w \, \mathrm{d}t .
\end{align*}
thanks to Fubini theorem and the same procedure as in the proof of Lemma \ref{LM:G0}. For the integral in space, we use Lemma \ref{LM:ChgmtVAR-X} and perform the change of variable $x'=\widetilde{\mathrm{X}}_\eps^{\tau;t}(x,w)$ 
by first observing that \eqref{equiv:TrajGood} entails
\begin{align*}
([\widetilde{\mathrm{X}}_\eps^{\tau;t,w}]^{-1}(x),[\Gamma_\eps^{t,x}]^{-1}(w)) \in \mathcal{O}^t_{\eps} \Longrightarrow\widetilde{\mathrm{X}}_\eps^{0;t,w}([\widetilde{\mathrm{X}}_\eps^{\tau;t,w}]^{-1}(x)) \in \R^3_+.
\end{align*}
We get for all $w \in \R^3$
\begin{align*}
&\int_{\R^3_+}   \mathbf{1}_{\mathcal{O}^t_{\eps}}(x,[\Gamma_\eps^{t,x}]^{-1}(w))   f_\eps^0(\widetilde{\mathrm{X}}_{\eps}^{0;t}(x,w),w)\Big\vert \partial_\tau [P u_\eps](\tau,\widetilde{\mathrm{X}}_\eps^{\tau;t}(x,w)) \Big\vert^2  \, \mathrm{d}x \\
 &=\int_{\R^3} \mathbf{1}_{[\widetilde{\mathrm{X}}_\eps^{\tau;t,w}]^{-1}(x) \in \R^3_+}  \mathbf{1}_{\mathcal{O}^t_{\eps}}([\widetilde{\mathrm{X}}_\eps^{\tau;t,w}]^{-1}(x),[\Gamma_\eps^{t,x}]^{-1}(w))  
\\
& \qquad \qquad \qquad \qquad  \times f_\eps^0(\widetilde{\mathrm{X}}_\eps^{0;t,w}([\widetilde{\mathrm{X}}_\eps^{\tau;t,w}]^{-1}(x)),w)\Big\vert \partial_\tau [P u_\eps](\tau,x) \Big\vert^2 \vert \det( \D_x [\Xi_\eps^{\tau;t,w}]^{-1}(x))\vert  \, \mathrm{d}x \\
& \lesssim \Vert f_\eps^0(\cdot,w) \Vert_{\Ld^{\infty}(\R^3_+)} \Vert \partial_\tau [P u_\eps](\tau) \Vert_{\Ld^2(\R^3)}^2.
\end{align*}
All in all, this yields
\begin{align*}
\int_0^T \int_{\R^3_+} \vert G_\eps^1(t,x) \vert^2 \, \mathrm{d}t \, \mathrm{d}x & \lesssim \eps \Vert f_\eps^0 \Vert_{\Ld^1(\R^3; \Ld^{\infty}(\R^3_+))} \int_0^T \int_{\R^3} \int_0^t  e^{\frac{\tau-t}{\eps}}   \Vert f_\eps^0(\cdot,w) \Vert_{\Ld^{\infty}(\R^3_+)} \Vert \partial_\tau [P u_\eps](\tau) \Vert_{\Ld^2(\R^3)}^2 \, \mathrm{d}\tau \, \mathrm{d}w \, \mathrm{d}t \\
& \leq \eps ^2\Vert f_\eps^0 \Vert_{\Ld^1(\R^3; \Ld^{\infty}(\R^3_+))}^{2}  \Vert\partial_t u_\eps \Vert_{\Ld^2((0,T) \times \R^3_+)}^{2},
\end{align*}
which concludes the proof.
\end{proof}

\begin{lemme}\label{LM:G2}
Let $T>0$ such that the condition \eqref{control:Nablau} holds at time $T$. Assume that for all $\eps >0$
\begin{align*}
\Vert (1+ \vert v \vert^{2}) f_\eps^0 \Vert_{\Ld^1(\R^3; \Ld^{\infty}( \R^3_+))} &\leq M, \\
\end{align*}
for some $M>1$ independent of $\eps$. Then for all $\eps >0$
\begin{align*}
\Vert G_\eps^2\Vert_{\Ld^2((0,T) \times \R^3_+} &\lesssim  \eps^{\frac{3}{4}} M \Vert  u_\eps \Vert_{\Ld^{\infty}(0,T;\Ld^2(\R^3_+))}^{\frac{1}{2}} \Vert \mathrm{D}^2_x u_\eps\Vert_{\Ld^2((0,T) \times \R^3_+)}^{\frac{1}{2}} T^{\frac{1}{2}} \\
& \quad + \eps M \Vert u_\eps \Vert_{\Ld^{\infty}(0,T;\Ld^2(\R^3_+))}\Vert \nabla_x u_\eps \Vert_{\Ld^2(0,T;\Ld^{\infty}(\R^3_+))}.
\end{align*}
\end{lemme}
\begin{proof}
In view of the formula \eqref{eq:CharacVGamma-1}, we have
\begin{align*}
\vert G_\eps^2(t,x) \vert 
& \leq  \int_{\R^3}\mathbf{1}_{\mathcal{O}^t_{\eps}}(x,[\Gamma_\eps^{t,x}]^{-1}(w)) \int_0^t  e^{\frac{-t}{\eps}} f_\eps^0(\widetilde{\mathrm{X}}_\eps^{0;t}(x,w),w) (1+ \vert w \vert) \Big\vert    \nabla_x  [P u_\eps](s,\widetilde{\mathrm{X}}_\eps^{s;t}(x,w)) \Big\vert \, \mathrm{d}s \, \mathrm{d}w \\
& \quad+ \int_{\R^3}\mathbf{1}_{\mathcal{O}^t_{\eps}}(x,[\Gamma_\eps^{t,x}]^{-1}(w)) \int_0^t  e^{\frac{s-t}{\eps}} f_\eps^0(\widetilde{\mathrm{X}}_\eps^{0;t}(x,w),w)\Big\vert \nabla_x [P u_\eps](s,\widetilde{\mathrm{X}}_\eps^{s;t}(x,w)) \Big\vert \, \mathrm{d}s \, \mathrm{d}w \\
& \quad +  \dfrac{1}{\eps} \int_{\R^3}\mathbf{1}_{\mathcal{O}^t_{\eps}}(x,[\Gamma_\eps^{t,x}]^{-1}(w)) \int_0^t  \int_0^s e^{\frac{s-t}{\eps}} e^{\frac{\tau-s}{\eps}} f_\eps^0(\widetilde{\mathrm{X}}_\eps^{0;t}(x,w),w) \\
& \qquad \qquad \qquad \qquad \times   \left\vert (Pu_{\eps})(\tau,\widetilde{\mathrm{X}}_{\eps}^{\tau;t}(x,w)) \right\vert  \Big\vert \nabla_x[P u_\eps](s,\widetilde{\mathrm{X}}_\eps^{s;t}(x,w)) \Big\vert  \, \mathrm{d}\tau  \, \mathrm{d}s \, \mathrm{d}w \\
& = \mathrm{I}(t,x)+\mathrm{II}(t,x)+\mathrm{III}(t,x).
\end{align*}
We now estimate each term separately. By Fubini theorem and Hölder inequality (in velocity), we have
\begin{align*}
\vert \mathrm{I}(t,x) \vert^2 &\lesssim  e^{\frac{-2t}{\eps}} \Bigg[\int_0^t  \left( \int_{\R^3}\mathbf{1}_{\mathcal{O}^t_{\eps}}(x,[\Gamma_\eps^{t,x}]^{-1}(w))(1+ \vert w \vert^{2})  f_\eps^0(\widetilde{\mathrm{X}}_\eps^{0;t}(x,w),w)  \, \mathrm{d}w \right)^{\frac{1}{2}} \\
& \qquad \qquad \qquad \times \left( \int_{\R^3}\mathbf{1}_{\mathcal{O}^t_{\eps}}(x,[\Gamma_\eps^{t,x}]^{-1}(w))  f_\eps^0(\widetilde{\mathrm{X}}_\eps^{0;t}(x,w),w) \Big\vert    \nabla_x  [P u_\eps](s,\widetilde{\mathrm{X}}_\eps^{s;t}(x,w)) \Big\vert^2 \, \mathrm{d}w \right)^{\frac{1}{2}} \, \mathrm{d}s  \Bigg]^2 \\
& \leq \Vert (1+ \vert v \vert^{2}) f_\eps^0 \Vert_{\Ld^1(\R^3; \Ld^{\infty}(\R^3_+))}e^{\frac{-2t}{\eps}} \Bigg[\int_0^t \eps e^{\frac{s}{\eps}} (1-e^{-\frac{t}{\eps}}) \\
& \qquad  \qquad  \times  \left( \int_{\R^3}\mathbf{1}_{\mathcal{O}^t_{\eps}}(x,[\Gamma_\eps^{t,x}]^{-1}(w))  f_\eps^0(\widetilde{\mathrm{X}}_\eps^{0;t}(x,w),w) \Big\vert    \nabla_x  [P u_\eps](s,\widetilde{\mathrm{X}}_\eps^{s;t}(x,w)) \Big\vert^2 \, \mathrm{d}w \right)^{\frac{1}{2}} \frac{e^{-\frac{s}{\eps}} \, \mathrm{d}s}{\eps(1-e^{-\frac{t}{\eps}})}   \Bigg]^2 \\
& \leq  \eps \Vert (1+ \vert v \vert^{2})  f_\eps^0 \Vert_{\Ld^1(\R^3; \Ld^{\infty}(\R^3_+))} \\ 
& \qquad \qquad \qquad  \times \int_0^t e^{\frac{2(s-t)}{\eps}}  \int_{\R^3} \mathbf{1}_{\mathcal{O}^t_{\eps}}(x,[\Gamma_\eps^{t,x}]^{-1}(w))  f_\eps^0(\widetilde{\mathrm{X}}_\eps^{0;t}(x,w),w) \Big\vert    \nabla_x  [P u_\eps](s,\widetilde{\mathrm{X}}_\eps^{s;t}(x,w)) \Big\vert^2 \, \mathrm{d}w \, \mathrm{d}s.
\end{align*}
Here, we have used Jensen inequality for the probability space
\begin{align*}
\left((0,t), \,  \frac{e^{-\frac{s}{\eps}} \, \mathrm{d}s}{\eps(1-e^{-\frac{t}{\eps}})}  \right).
\end{align*}
Then, we perform the change of variable in space $x'=\widetilde{\mathrm{X}}_\eps^{s;t}(x,w)$ allowed by Lemma \ref{LM:ChgmtVAR-X} together with Fubini Theorem to get
\begin{align*}
 \int_{\R^3_+} \vert \mathrm{I}(t,x) \vert^2 \, \mathrm{d}x  & \lesssim  \eps \Vert (1+ \vert v \vert^{2}) f_\eps^0 \Vert_{\Ld^1(\R^3; \Ld^{\infty}(\R^3_+))}  \\ 
& \qquad \times  \int_0^t e^{\frac{2(s-t)}{\eps}} \Bigg( \int_{\R^3_+ \times \R^3} \mathbf{1}_{[\widetilde{\mathrm{X}}_\eps^{\tau;t,w}]^{-1}(x) \in \R^3_+}  \mathbf{1}_{\mathcal{O}^t_{\eps}}([\widetilde{\mathrm{X}}_\eps^{\tau;t,w}]^{-1}(x),[\Gamma_\eps^{t,x}]^{-1}(w))   \\
& \qquad \qquad \qquad  \times  f_\eps^0(\widetilde{\mathrm{X}}_\eps^{0;t,w}([\widetilde{\mathrm{X}}_\eps^{\tau;t,w}]^{-1}(x)),w) \Big\vert    \nabla_x  [P u_\eps](s,x) \Big\vert^2 \, \mathrm{d}x \, \mathrm{d}w \Bigg)  \, \mathrm{d}s   \\
& \leq  \eps \Vert (1+ \vert v \vert^{2})  f_\eps^0 \Vert_{\Ld^1(\R^3; \Ld^{\infty}(\R^3_+))}  \Vert   f_\eps^0 \Vert_{\Ld^1(\R^3; \Ld^{\infty}(\R^3_+))}     \\
& \qquad \qquad \qquad \qquad \qquad  \times  \int_0^t e^{\frac{2(s-t)}{\eps}}  \Vert  \nabla_x  [P u_\eps](s) \Vert_{\Ld^2(\R^3)}^2 \, \mathrm{d}s .
\end{align*}
Note that since $u_\eps(s) \in \H^1_0(\R^3_+)$, we have $\nabla_x  [P u_\eps](s) = \mathbf{1}_{\R^3_+} \nabla_x  u_\eps(s)$. Thanks to the Gagliardo-Nirenberg-Sobolev inequality (see Theorem \ref{gagliardo-nirenberg} in the Appendix) and the energy inequality \eqref{ineq:energy2}, we have
\begin{align*}
\Vert \nabla_x  u_\eps(s) \Vert_{\Ld^2(\R^3_+)} &\lesssim \Vert \mathrm{D}^2_x u_\eps (s)\Vert_{\Ld^2(\R^3_+)}^{\frac{1}{2}} \Vert  u_\eps(s) \Vert_{\Ld^2(\R^3_+)}^{\frac{1}{2}}\lesssim  \Vert  u_\eps \Vert_{\Ld^{\infty}(0,T;\Ld^2(\R^3_+))}^{\frac{1}{2}} \Vert \mathrm{D}^2_x u_\eps (s)\Vert_{\Ld^2(\R^3_+)}^{\frac{1}{2}},
\end{align*}
from which we infer
\begin{align*}
 \int_{\R^3_+} \vert \mathrm{I}(t,x) \vert^2 \, \mathrm{d}x &\lesssim  \eps \Vert (1+ \vert v \vert^{2}) f_\eps^0 \Vert_{\Ld^1(\R^3; \Ld^{\infty}(\R^3_+))}  \Vert  f_\eps^0 \Vert_{\Ld^1(\R^3; \Ld^{\infty}(\R^3_+))}  \\
& \qquad \qquad \qquad \qquad \qquad \times \Vert  u_\eps \Vert_{\Ld^{\infty}(0,T;\Ld^2(\R^3_+))}  \int_0^t  e^{\frac{2(s-t)}{\eps}} \Vert   \mathrm{D}^2_x u_\eps(s) \Vert_{\Ld^2(\R^3_+)} \, \mathrm{d}s. 
\end{align*}
Since
\begin{align*}
\int_0^t  e^{\frac{2(s-t)}{\eps}} \Vert   \mathrm{D}^2_x u_\eps(s) \Vert_{\Ld^2(\R^3_+)} \, \mathrm{d}s \leq \Vert \mathrm{D}^2_x u_\eps\Vert_{\Ld^2((0,t) \times \R^3_+)} \left( \int_0^t  e^{\frac{4(s-t)}{\eps}}  \, \mathrm{d}s \right)^{\frac{1}{2}},
\end{align*}
we obtain
\begin{align*}
\Vert \mathrm{I}(t) \Vert_{\Ld^2(\R^3_+)}  &  \lesssim \eps^{\frac{3}{4}} \Vert (1+ \vert v \vert^{2})  f_\eps^0 \Vert_{\Ld^1(\R^3; \Ld^{\infty}(\R^3_+))}^{\frac{1}{2}}  \Vert f_\eps^0 \Vert_{\Ld^1(\R^3; \Ld^{\infty}(\R^3_+))}^{\frac{1}{2}} \Vert  u_\eps \Vert_{\Ld^{\infty}(0,T;\Ld^2(\R^3_+))}^{\frac{1}{2}} \Vert \mathrm{D}^2_x u_\eps\Vert_{\Ld^2((0,t) \times \R^3_+)}^{\frac{1}{2}},
\end{align*}
and an integration in time yields
\begin{align*}
\Vert \mathrm{I} \Vert_{\Ld^2(0;T,\Ld^2(\R^3_+)}   \lesssim \eps^{\frac{3}{4}} \Vert (1+ \vert v \vert^{2})  f_\eps^0 \Vert_{\Ld^1(\R^3; \Ld^{\infty}(\R^3_+))}^{\frac{1}{2}}  \Vert f_\eps^0 \Vert_{\Ld^1(\R^3; \Ld^{\infty}(\R^3_+))}^{\frac{1}{2}} \Vert  u_\eps \Vert_{\Ld^{\infty}(0,T;\Ld^2(\R^3_+))}^{\frac{1}{2}} \Vert \mathrm{D}^2_x u_\eps\Vert_{\Ld^2((0,T) \times \R^3_+)}^{\frac{1}{2}} T^{\frac{1}{2}}.
\end{align*}
Let us turn to the control of $\mathrm{II}(t)$. The very same procedure as for $\mathrm{I}(t)$ leads to
\begin{align*}
\vert \mathrm{II}(t,x) \vert^2
& \leq \Vert   f_\eps^0 \Vert_{\Ld^1(\R^3; \Ld^{\infty}(\R^3_+))} \\
& \qquad  \times \Bigg[\int_0^t e^{\frac{2(s-t)}{\eps}} \left( \int_{\R^3}\mathbf{1}_{\mathcal{O}^t_{\eps}}(x,[\Gamma_\eps^{t,x}]^{-1}(w))  f_\eps^0(\widetilde{\mathrm{X}}_\eps^{0;t}(x,w),w) \Big\vert    \nabla_x  [P u_\eps](s,\widetilde{\mathrm{X}}_\eps^{s;t}(x,w)) \Big\vert^2 \, \mathrm{d}w \right)^{\frac{1}{2}} \, \mathrm{d}s  \Bigg]^2 \\
& = \Vert   f_\eps^0 \Vert_{\Ld^1(\R^3; \Ld^{\infty}(\R^3_+))}  \Bigg[\int_0^t \frac{\epsilon}{2}(1-e^{-\frac{2t}{\eps}}) \\
&  \qquad  \qquad  \left( \int_{\R^3}\mathbf{1}_{\mathcal{O}^t_{\eps}}(x,[\Gamma_\eps^{t,x}]^{-1}(w)) f_\eps^0(\widetilde{\mathrm{X}}_\eps^{0;t}(x,w),w) \Big\vert    \nabla_x  [P u_\eps](s,\widetilde{\mathrm{X}}_\eps^{s;t}(x,w)) \Big\vert^2 \, \mathrm{d}w \right)^{\frac{1}{2}} \frac{e^{\frac{2(s-t)}{\eps}} \, \mathrm{d}s}{\frac{\epsilon}{2}(1-e^{-\frac{2t}{\eps}})} \Bigg]^2 \\
&\lesssim \eps \Vert    f_\eps^0 \Vert_{\Ld^1(\R^3; \Ld^{\infty}(\R^3_+))} \\
& \qquad  \times \int_0^t e^{\frac{2(s-t)}{\eps}} \left( \int_{\R^3}\mathbf{1}_{\mathcal{O}^t_{\eps}}(x,[\Gamma_\eps^{t,x}]^{-1}(w))  f_\eps^0(\widetilde{\mathrm{X}}_\eps^{0;t}(x,w),w) \Big\vert    \nabla_x  [P u_\eps](s,\widetilde{\mathrm{X}}_\eps^{s,t}(x,w)) \Big\vert^2 \, \mathrm{d}w \right) \, \mathrm{d}s.
\end{align*}
The same change of variable in space as in the case of $\mathrm{I}(t)$ and the same computations give us
\begin{align*}
\int_{\R^3_+} \vert \mathrm{II}(t,x) \vert^2 \, \mathrm{d}x \lesssim \eps \Vert    f_\eps^0 \Vert_{\Ld^1(\R^3; \Ld^{\infty}(\R^3_+))}^2 \Vert  u_\eps \Vert_{\Ld^{\infty}(0,T;\Ld^2(\R^3_+))}  \Vert \mathrm{D}^2_x u_\eps\Vert_{\Ld^2((0,t) \times \R^3_+)} \left( \int_0^t  e^{\frac{4(s-t)}{\eps}}  \, \mathrm{d}s \right)^{\frac{1}{2}},
\end{align*}
and then
\begin{align*}
 \Vert \mathrm{II}(t) \Vert_{\Ld^2(\R^3_+)}& \lesssim \eps^{\frac{3}{4}} \Vert  f_\eps^0 \Vert_{\Ld^1(\R^3; \Ld^{\infty}(\R^3_+))}  \Vert  u_\eps \Vert_{\Ld^{\infty}(0,T;\Ld^2(\R^3_+))}^{\frac{1}{2}} \Vert \mathrm{D}^2_x u_\eps\Vert_{\Ld^2((0,t) \times \R^3_+)}^{\frac{1}{2}}.
 \end{align*}
 We thus obtain
 \begin{align*}
 \Vert \mathrm{II} \Vert_{\Ld^2(0;T,\Ld^2(\R^3_+)}   \lesssim \eps^{\frac{3}{4}} \Vert  f_\eps^0 \Vert_{\Ld^1(\R^3; \Ld^{\infty}(\R^3_+))}  \Vert  u_\eps \Vert_{\Ld^{\infty}(0,T;\Ld^2(\R^3_+))}^{\frac{1}{2}} \Vert \mathrm{D}^2_x u_\eps\Vert_{\Ld^2((0,T) \times \R^3_+)}^{\frac{1}{2}} T^{\frac{1}{2}}.
 \end{align*}
 For the last term $\mathrm{III}(t)$, we start writing
 \begin{align*}
\vert \mathrm{III}(t,x) \vert^2 &=\Bigg\vert  \dfrac{1}{\eps} \int_{\R^3}\mathbf{1}_{\mathcal{O}^t_{\eps}}(x,[\Gamma_\eps^{t,x}]^{-1}(w))  f_\eps^0(\widetilde{\mathrm{X}}_{\eps}^{0;t}(x,w),w) \int_0^t e^{\frac{s-t}{\eps}} \Big\vert \nabla_x[P u_\eps](s,\widetilde{\mathrm{X}}_\eps^{s;t}(x,w)) \Big\vert \\  
& \qquad \qquad \qquad  \qquad \qquad \qquad \qquad \qquad \times \left( \int_0^s  e^{\frac{\tau-s}{\eps}}   \left\vert (Pu_{\eps})(\tau,\widetilde{\mathrm{X}}_{\eps}^{\tau;t}(x,w)) \right\vert    \, \mathrm{d}\tau \right)  \, \mathrm{d}s \, \mathrm{d}w \Bigg\vert^2 \\
& \leq  \Bigg\vert \dfrac{1}{\eps} \int_{\R^3}\mathbf{1}_{\mathcal{O}^t_{\eps}}(x,[\Gamma_\eps^{t,x}]^{-1}(w))  f_\eps^0(\widetilde{\mathrm{X}}_{\eps}^{0;t}(x,w),w) \int_0^t e^{\frac{s-t}{\eps}} \Big\vert \nabla_x[P u_\eps](s,\widetilde{\mathrm{X}}_\eps^{s;t}(x,w)) \Big\vert \\
& \qquad \qquad \qquad \qquad \qquad \qquad \times  \Big( \int_0^s  e^{\frac{\tau-s}{\eps}}  \, \mathrm{d}\tau \Big)^{\frac{1}{2}}\left( \int_0^s  e^{\frac{\tau-s}{\eps}}   \left\vert (Pu_{\eps})(\tau,\widetilde{\mathrm{X}}_{\eps}^{\tau;t}(x,w)) \right\vert^2    \, \mathrm{d}\tau \right)^{\frac{1}{2}}  \, \mathrm{d}s \, \mathrm{d}w \Bigg\vert^2 \\
\end{align*}
hence we deduce
\begin{align*}
\vert \mathrm{III}(t,x) \vert^2&\lesssim \eps^{-1} \Bigg[  \int_{\R^3} \int_0^t \mathbf{1}_{\mathcal{O}^t_{\eps}}(x,[\Gamma_\eps^{t,x}]^{-1}(w)) f_\eps^0(\widetilde{\mathrm{X}}_{\eps}^{0;t}(x,w),w) e^{\frac{s-t}{\eps}} \Big\vert \nabla_x[P u_\eps](s,\widetilde{\mathrm{X}}_\eps^{s;t}(x,w)) \Big\vert \\  
& \qquad \qquad \qquad \qquad \qquad \qquad \qquad \qquad \times  \left( \int_0^s  e^{\frac{\tau-s}{\eps}}   \left\vert (Pu_{\eps})(\tau,\widetilde{\mathrm{X}}_{\eps}^{\tau;t}(x,w)) \right\vert^2    \, \mathrm{d}\tau \right)^{\frac{1}{2}}  \, \mathrm{d}s \, \mathrm{d}w \Bigg]^2 \\
& \leq \eps^{-1}\left( \int_0^t  \int_{\R^3} \mathbf{1}_{\mathcal{O}^t_{\eps}}(x,[\Gamma_\eps^{t,x}]^{-1}(w)) f_\eps^0(\widetilde{\mathrm{X}}_{\eps}^{0;t}(x,w),w)  e^{\frac{s-t}{\eps}} \Big\vert \nabla_x[P u_\eps](s,\widetilde{\mathrm{X}}_\eps^{s;t}(x,w)) \Big\vert^{2} \, \mathrm{d}s \, \mathrm{d}w  \right) \\
& \times  \left(  \int_0^t \int_{\R^3} \mathbf{1}_{\mathcal{O}^t_{\eps}}(x,[\Gamma_\eps^{t,x}]^{-1}(w)) f_\eps^0(\widetilde{\mathrm{X}}_{\eps}^{0;t}(x,w),w) e^{\frac{s-t}{\eps}} \left( \int_0^s  e^{\frac{\tau-s}{\eps}}   \left\vert (Pu_{\eps})(\tau,\widetilde{\mathrm{X}}_{\eps}^{\tau;t}(x,w)) \right\vert^2    \, \mathrm{d}\tau \right) \, \mathrm{d}s \, \mathrm{d}w  \right),
\end{align*}
where we have used Hölder inequality in velocity and time in the last inequality. By \eqref{equiv:TrajGood},
we have
\begin{align*}
\int_{0}^T \int_{\R^3_+} \vert \mathrm{III}(t,x) \vert^2 \, \mathrm{d}x \, \mathrm{d}t &\lesssim \eps^{-1} \Vert  f_\eps^0\Vert_{\Ld^1(\R^3; \Ld^{\infty}(\R^3_+))} \left[ \int_0^T \left( \int_0^t  e^{\frac{s-t}{\eps}}  \Vert \nabla_x u_\eps(s) \Vert_{\Ld^{\infty}(\R^3_+)}^{2} \, \mathrm{d}s \right) \, \mathrm{d}t \right] \\
& \quad \times \underset{t \in (0,T)}{\sup} \Bigg\{ \int_0^te^{\frac{s-t}{\eps}}  \int_0^s  e^{\frac{\tau-s}{\eps}}  \\
& \quad  \times \int_{\R^3_+ \times \R^3} \mathbf{1}_{\mathcal{O}^t_{\eps}}(x,[\Gamma_\eps^{t,x}]^{-1}(w))  f_\eps^0(\widetilde{\mathrm{X}}_{\eps}^{0;t}(x,w),w)      \left\vert (Pu_{\eps})(\tau,\widetilde{\mathrm{X}}_{\eps}^{\tau;t}(x,w)) \right\vert^2 \, \mathrm{d}x \, \mathrm{d}w  \, \mathrm{d}\tau  \, \mathrm{d}s  \Bigg\} .
\end{align*}
For the term between braces, we perform the change of variable $x'=\widetilde{\mathrm{X}}_{\eps}^{\tau;t}(x,w)$ and we get
\begin{align*}
& \int_0^te^{\frac{s-t}{\eps}}  \int_0^s  e^{\frac{\tau-s}{\eps}}  \int_{\R^3_+ \times \R^3} \mathbf{1}_{\mathcal{O}^t_{\eps}}(x,[\Gamma_\eps^{t,x}]^{-1}(w))  f_\eps^0(\widetilde{\mathrm{X}}_{\eps}^{0;t}(x,w),w)      \left\vert (Pu_{\eps})(\tau,\widetilde{\mathrm{X}}_{\eps}^{\tau;t}(x,w)) \right\vert^2 \, \mathrm{d}x \, \mathrm{d}w  \, \mathrm{d}\tau \, \mathrm{d}s \\
 & \lesssim \Vert  f_\eps^0 \Vert_{\Ld^1(\R^3 ; \Ld^{\infty}(\R^3_+)) }\int_0^te^{\frac{s-t}{\eps}}  \int_0^s  e^{\frac{\tau-s}{\eps}} \Vert u_\eps(\tau) \Vert_{\Ld^2(\R^3_+)}^2 \, \mathrm{d}\tau \, \mathrm{d}s \\
 & \lesssim \eps^2 \Vert \vert v \vert^kf_\eps^0 \Vert_{\Ld^1(\R^3 ; \Ld^{\infty}(\R^3_+)) } \Vert u_\eps \Vert_{\Ld^{\infty}(0,T;\Ld^2(\R^3_+))}^2.
\end{align*}

For the term in brackets, we write
\begin{align*}
\int_0^T  \left( \int_0^t  e^{\frac{s-t}{\eps}}  \Vert \nabla_x u_\eps(s) \Vert_{\Ld^{\infty}(\R^3_+)}^{2} \ \mathrm{d}s \right) \, \mathrm{d}t \leq \eps  \int_0^T \Vert \nabla_x u_\eps(s) \Vert_{\Ld^{\infty}(\R^3_+)}^{2}  \, \mathrm{d}s.
\end{align*}
This yields
\begin{align*}
\int_{0}^T \int_{\R^3_+} \vert \mathrm{III}(t,x) \vert^2 \, \mathrm{d}x \, \mathrm{d}t &\lesssim \eps^{2} \Vert  f_\eps^0 \Vert_{\Ld^1(\R^3 ; \Ld^{\infty}(\R^3_+)) }^2 \Vert u_\eps \Vert_{\Ld^{\infty}(0,T;\Ld^2(\R^3_+))}^2 \int_0^T \Vert \nabla_x u_\eps(s) \Vert_{\Ld^{\infty}(\R^3_+)}^{2}   \, \mathrm{d}s.
\end{align*}
We obtain the result by gathering all the terms together.
\end{proof}

In view of the previous uniform bounds, we get the following result which eventually quantifies the convergence for $j_\eps-\rho_\eps u_\eps$ when $\eps \rightarrow 0$.
\begin{corollaire}\label{coro:condBrinkCV}
Let $\big((u_\eps,f_\eps)\big)_{\eps >0}$ be a family of global weak solutions to the VNS system which are smooth. Let $T>0$ such that the condition \eqref{control:Nablau} holds at time $T$. Assume that for all $\eps>0$, we have 
\begin{align*}
\Vert u_\eps^0 \Vert_{\Ld^2(\R^3_+)}+\Vert (1+\vert v \vert^2) f_\eps^0 \Vert_{\Ld^{1}(\R^3; \Ld^{\infty} \cap \Ld^1(\R^3_+))} & \leq M,
\end{align*}
for some $M>1$ independent of $\eps$. Then there exists $\mu_2 >0$ such that for all $\eps>0$, we have
\begin{align*}
\Vert j_\eps-\rho_\eps u_\eps +\rho_\eps e_3 \Vert_{\Ld^2((0,T) \times \R^3_+)} & \lesssim \eps^{\frac{1}{2}}M^{\mu_2} + \eps M \Vert\partial_t u_\eps \Vert_{\Ld^2((0,T) \times \R^3_+)} \\
& \quad + \eps^{\frac{3}{4}} M \Vert  u_\eps \Vert_{\Ld^{\infty}(0,T;\Ld^2(\R^3_+))}^{\frac{1}{2}} \Vert \mathrm{D}^2_x u_\eps\Vert_{\Ld^2((0,T) \times \R^3_+)}^{\frac{1}{2}} T^{\frac{1}{2}} \\
& \quad + \eps M \Vert u_\eps \Vert_{\Ld^{\infty}(0,T;\Ld^2(\R^3_+))}\Vert \nabla_x u_\eps \Vert_{\Ld^2(0,T;\Ld^{\infty}(\R^3_+))}.
\end{align*}
\end{corollaire}

\subsection{Exit geometric condition and absorption on the half-space}\label{Subsection:EGCabs}
We eventually define and study the \textit{exit geometric condition}, ensuring that particles starting in a given area of the phase-space leave the half-space before a prescribed time. This will be the key tool to analyse the absorption effect at the boundary, of crucial importance in Section \ref{Section:estimateBrinkman}.
This notion is reminiscent of an important idea used in \cite{GHKM} and has also been revisited in \cite{E} for the study of the large time behavior. In short, it requires a control of the type $ \Vert  u_\eps \Vert_{\Ld^1(0,T;\Ld^{\infty}(\R^3_+))} \ll 1$ in order to hold true and is truly based on the presence of the gravity term in the kinetic equation.
\begin{definition}\label{DEF:EGC}
Let $\eps>0$ and $L,R>0$. We say that a vector field $\U \in \Ld^1_{\mathrm{loc}}(\R^+;\mathrm{W}^{1,\infty}_0(\R^3_+))$ satisfies the exit geometric condition (EGC) in time $T \geq 0$ with respect to $\big(\R^2 \times (0,L)\big) \times \B(0,R)$ if
\begin{align}\label{EGC-T}
\underset{\substack{x \in \R^2 \times (0,L) \\v \in \B(0,R)}}{\sup} \ \tau^{+}_{\U, \eps}(0,x,v)<T,
\end{align}
where $\tau^{+}_{\U, \eps}$ refers to Definition \eqref{def:tau} for the characteristic curves $(\mathrm{X_{\U, \eps}},\mathrm{V_{\U, \eps}})$ of the Vlasov equation associated to a velocity field $\U$ in \eqref{EDO-charac}. 

In what follows, we shall say that $\U$ satisfies $\mathrm{EGC}_{\eps}^{L,R}(T)$.
\end{definition}

\medskip

As a consequence of the representation formula \eqref{formule:rep}, we obtain the following proposition which highlights the effect of an EGC on the solution to the Vlasov equation. We refer to \cite[Proposition 5.2]{E} for a proof.
\begin{proposition}\label{PropoABS-form}
Suppose that a velocity field $\U \in \Ld^1_{\mathrm{loc}}(\R^+;\mathrm{W}^{1,\infty}_0(\R^3_+))$ satisfies $\mathrm{EGC}_{\eps}^{L,R}(T)$ for some fixed $L, R>0$. Then, if $f_\eps$ is the solution to the Vlasov equation associated to $U$, we have for almost every $(x,v) \in \R^3_+ \times \R^3$ and any $t > T$
\begin{align}\label{formuleRep:split3}
\begin{split}
f_\eps(t,x,v)&\leq e^{\frac{3t}{\eps}} \,  \mathbf{1}_{\mathcal{O}^t_{\eps,\U}}(x,v) \, \mathbf{1}_{\vert \V^{0;t}_{\eps,\U}(x,v) \vert >R}  \, f_\eps^0(\X^{0;t}_{\eps,\U}(x,v),\V^{0;t}_{\eps,\U}(x,v)) \\[2mm]
& \quad +  e^{\frac{3t}{\eps}} \,  \mathbf{1}_{\mathcal{O}^t_{\eps,\U}}(x,v) \,  \mathbf{1}_{ \X^{0;t}_{\eps,\U}(x,v)_3 >L}  \, f_\eps^0(\X^{0;t}_{\eps,\U}(x,v),\V^{0;t}_{\eps,\U}(x,v)),
\end{split}
\end{align}
where
$$\mathcal{O}^t_{\eps,\U}:=\left\lbrace (x,v) \in \R^3_+ \times \R^3 \mid \forall \sigma \in [0,t], \ \mathrm{X}_\eps(\sigma;t,x,v) \in \R^3_+\right\rbrace.$$
\end{proposition}

\medskip
The main task is now to find a sufficient condition which can ensure that a vector field satisfies an EGC. We rely on a stability principle, comparing the whole system of curves \eqref{EDO-charac} for the Vlasov equation with velocity field $u_\eps$ (solution to the Navier-Stokes equations) to the same version without the fluid velocity.
We thus consider the following characteristic curves $(\mathrm{X}^g_\eps,\mathrm{V}^g_\eps)$ for the Vlasov equation associated with the vector field $(x,v) \mapsto (v,-e_3-v)$:
\begin{equation}\label{EDO-charac-grav}
\left\{
      \begin{aligned}
        \dot{\mathrm{X}}^g_\eps(s;t,x,v) &=\mathrm{V}^g_\eps(s;t,x,v),\ \ \ &&\dot{\mathrm{V}}^g_\eps (s;t,x,v)= \frac{1}{\eps}\left(-e_3-\mathrm{V}^g_\eps(s;t,x,v)\right),\\[2mm]
	\mathrm{X}^g_\eps(t;t,x,v)&=x, \ \ &&\mathrm{V}^g_\eps(t;t,x,v)=v.
      \end{aligned}
    \right.
\end{equation}
This corresponds to the free evolution of the particles, without coupling with the fluid phase and with the sole presence of the gravity field. In view of the simpler form of that system, we hope for precise information on the absorption at the boundary for \eqref{EDO-charac}. Indeed, we have
\begin{equation}\label{expr:Zt-grav}
\left\{
      \begin{aligned}
        \mathrm{X}_{\varepsilon}^g(t;s,x,v)&=x+\eps(1-e^{\frac{s-t}{\eps}})(v+e_3)-(t-s)e_3,  \\[2mm]
        \mathrm{V}_{\varepsilon}^g(t;s,x,v)&= e^{\frac{s-t}{\eps}}(v+e_3)-e_3,   \\
      \end{aligned}
    \right.
\end{equation}
so that
\begin{align*}
\left\{
      \begin{aligned}
        \mathrm{X}_{\varepsilon}^g(t;s,x,v)_3&=x_3+\eps(1-e^{\frac{s-t}{\eps}})(v_3+1)-(t-s),  \\[2mm]
        \mathrm{V}_{\varepsilon}^g(t;s,x,v)_3&= e^{\frac{s-t}{\eps}}(v_3+1)-1.   \\
      \end{aligned}
    \right.
\end{align*}
%
In particular, we will precisely quantify how one can ensure an EGC for \eqref{expr:Zt-grav}. Coming back to the full system will be possible thanks to the following stability result, mainly inspired from \cite{GHKM,E}. It will allow us to transfer to \eqref{EDO-charac} any EGC satisfied in finite time, provided that the $\Ld^1_t \Ld^{\infty}_x$ norm of the vector field $\U$ defining the curves is small enough.

\begin{lemme}\label{LM:perturbEGC}
Let $\alpha>0$. There exists a constant $\kappa_{\alpha}>0$ such that the following holds for all $\eps \in(0,1)$. Suppose that the trivial vector field (related to $(\mathrm{X}^g_\eps,\mathrm{V}^g_\eps)$) satisfies $\mathrm{EGC}_{\eps}^{L,R}(T)$, where $L,R>0$ are given. Then, any vector field $\U \in \Ld^1_{\mathrm{loc}}(\R^+;\mathrm{W}^{1,\infty}_0(\R^3_+))$ such that
\begin{align}\label{smallnessU-L1Linfty}
\int_0^{T+\alpha} \Vert \U(s) \Vert_{\Ld^{\infty}(\R^3_+)} \, \mathrm{d}s \leq \kappa_{\alpha},
\end{align}
satisfies $\mathrm{EGC}_{\eps}^{L,R}(T+\alpha)$.
\end{lemme}
\begin{proof}
The proof of \cite[Lemma 5.4]{E} actually applies \textit{mutatis mutandis} when one assumes $\eps \in(0,1)$.
\end{proof}

\medskip

Before going further, we also observe that for all $(x,v) \in \left(\R^2 \times (0,L) \right) \times \B(0,R)$ and for all $t \geq 0$ , we have
\begin{align*}
 \X^g_\eps(t;0,x,v)_3&=x_3 + \eps(1-e^{-\frac{t}{\eps}})(v_3+1)-t \\
& \leq L + \eps(1-e^{-\frac{t}{\eps}})(1+R)-t \\
&<L+\eps(1+R)-t.
\end{align*}
We thus infer the following lemma.
\begin{lemme}\label{LM:t_0(R)}
Let $\eps>0$. If $L,R>0$ are given, the trivial vector field $\mathrm{U} \equiv 0$ (associated to $(\mathrm{X}^g_\eps,\mathrm{V}^g_\eps)$) satisfies $\mathrm{EGC}_{\eps}^{L,R}(t_\eps^g(L,R))$ where
\begin{align}\label{def:t_0(R)}
t_\eps^g(L,R):= L+\eps(1+R).
\end{align}
\end{lemme}


\begin{definition}
For $\eps>0$ and $L,R>0$, we set for $s>0$
\begin{align*}
\ell_{\eps}^L(s):=\dfrac{1}{2}(s-\eps(1-e^{-\frac{s}{\eps}}))-L, \ \ r_\eps^R(s):=\dfrac{1}{2}\left(\dfrac{s}{\eps(1-e^{-\frac{s}{\eps}})} -1\right)-R.
\end{align*}
\end{definition}

We observe that for any $\eps>0$
\begin{align*}
\ell_{\eps}^L(t_\eps^g(L,R))>-L, \ \ r_{\eps}^R(t_\eps^g(L,R))>-R.
\end{align*}
Furthermore, the functions $\ell_{\eps}^L$ and $r_{\eps}^R$ are increasing on $\R^+$ and diverge towards $+\infty$ when $t\rightarrow + \infty$.

\medskip

Useful information about the EGC for the free system \eqref{EDO-charac-grav} of curves $(\mathrm{X}^g_\eps,\mathrm{V}^g_\eps)$ are then gathered in the following lemma.

\begin{lemme}\label{EGCt:reverse}
Let $\eps \in(0,1)$. Let $L,R >0$ such that
\begin{align*}
t_\eps^g(L,R)<t_0
\end{align*}
for some $t_0>0$ independent of $\eps$. 
Then for all $t \geq t_0$ the trivial vector field $\mathrm{U} \equiv 0$ (associated to ($\X^g_\eps,\V^g_\eps)$) satisfies $\mathrm{EGC}_{\eps}^{L+\ell_{\eps}^L(t),R+r_{\eps}^R(t)}(t)$. Furthermore, there exists $C=C(t_0)>0$ independent of $\eps$ such that 
\begin{align*}
\forall s \geq t_0, \ \ \dfrac{1}{L+\ell_{\eps}^L(s)} \leq \dfrac{C}{1+s}, \ \ \dfrac{1}{R+r_{\eps}^R(s)} \leq \dfrac{C}{1+s}.
\end{align*}
\end{lemme}
\begin{proof}
In view of the previous remark, we have
\begin{align*}
\forall t \geq t_0, \ \ \ell_{\eps}^L(t_0)>-L, \ \ r_{\eps}^R(t_0)>-R.
\end{align*}
Now, let $t \geq t_0$. We observe that if $x_3 \in (0, L+\ell_{\eps}^L(t))$ and $\vert v \vert \in [0,R+r_\eps^R(t))$ 
 then 
\begin{align*}
\X^g_\eps(t;0,x,v)_3=x_3+\eps(1-e^{-\frac{s}{\eps}})(v_3+1)-t<0.
\end{align*}
This implies that the trivial vector field $\mathrm{U} \equiv 0$ satisfies $\mathrm{EGC}_{\eps}^{L+\ell_{\eps}^L(t),R+r_{\eps}^R(t)}(t)$. Indeed, the function $s \mapsto \X^g(s;0,x,v)_3$ is strictly decreasing after the first time it vanishes. Finally, tedious but basic computations show that the functions
\begin{align*}
s \mapsto \dfrac{1+s}{L+\ell_{\eps}^L(s)}, \ \ \text{and} \ \ s \mapsto \dfrac{1+s}{R+r_\eps^R(s)}
\end{align*}
are positive and nonincreasing on $[t_0,+\infty)$ therefore we have
\begin{align*}
\forall s \geq t_0, \ \ \dfrac{1+s}{L+\ell_{\eps}^L(s)} \leq \dfrac{1+t_0}{L+\ell_{\eps}^L(t_0)}, \ \ \dfrac{1+s}{L+r_\eps^R(s)} \leq\dfrac{1+t_0}{L+r_\eps^R(t_0)}.
\end{align*}
A Taylor expansion at $\eps \rightarrow 0$ then shows that the two previous r.h.s are continuous and uniformly bounded by some constant independent of $\eps \in (0,1)$: there exists $C(t_0)>0$, independent of $\eps>0$ such that 
\begin{align*}
\forall s \geq t_0, \ \ \dfrac{1+s}{L+\ell_{\eps}^L(s)} \leq C(t_0), \ \ \dfrac{1+s}{L+r_\eps^R(s)} \leq C(t_0).
\end{align*}
The proof is then complete.
\end{proof}
\begin{remarque}\label{RMK:EGCepsdiminue}
We also have the following link between two EGC related to different parameters $\eps$: if $t \geq t_0 > t_{\eps_0}^g(L,R)$ then Lemma \ref{EGCt:reverse} ensures that $\U=0$ satisfies $\mathrm{EGC}_{\eps_0}^{L+\ell_{\eps_0}^L(t),R+r_{\eps_0}^R(t)}(t)$ and in addition, $\U=0$ satisfies $\mathrm{EGC}_{\eps}^{L+\ell_{\eps_0}^L(t),R+r_{\eps_0}^R(t)}(t)$ for any 
$$\eps \in \left(0, \frac{\eps_0}{2 \eps_0+1}\right).$$
Indeed, if $\eps$ is given in the previous interval we know from Lemma \ref{LM:t_0(R)} that $\U=0$ satisfies $$\mathrm{EGC}_{\eps}^{L+\ell_{\eps_0}^L(t),R+r_{\eps_0}^R(t)}\Big(t_\eps^g \big( L+\ell_{\eps_0}^L(t),R+r_{\eps_0}^R(t) \big) \Big).$$ It thus remains to prove that 
\begin{align*}
t_\eps^g \big( L+\ell_{\eps_0}^L(t),R+r_{\eps_0}^R(t) \big)<t,
\end{align*}
that is
\begin{align*}
\dfrac{L+\ell_{\eps_0}^L(t)}{t}+ \eps \dfrac{1+ R+r_{\eps_0}^R(t) }{t}<1.
\end{align*}
We observe that $s \mapsto \dfrac{L+\ell_{\eps_0}^L(s)}{s}$ is increasing on $\R^+$ and tends to $\dfrac{1}{2}$ as $t \rightarrow + \infty$, while $s \mapsto \dfrac{1+ R+r_{\eps_0}^R(s) }{s}$ is increasing on $\R^+$ and tends to $\dfrac{1}{2 \eps_0}$ as $t \rightarrow + \infty$, therefore 
\begin{align*}
\dfrac{L+\ell_{\eps_0}^L(t)}{t}+ \eps \dfrac{1+ R+r_{\eps_0}^R(t) }{t}<\frac{1}{2}+ \eps \frac{2\eps_0+1}{2 \eps_0}<1,
\end{align*}
by our choice of $\eps_0$.
\end{remarque}

\section{Preliminary results  on the solutions to the Vlasov-Navier-Stokes system}\label{Section:Prelim}
In this section, we mainly exhibit sufficient conditions ensuring the convergence of $(u_\eps, \rho_\eps)$ when $\eps \rightarrow 0$, as well as some several non-uniform (in $\eps$ and $T$) estimates for the VNS system paving the way for a local in time analysis.
\begin{itemize}
\item In Subsection \ref{Subsec:improvineq+conddecay}, we derive an improvement of the energy-dissipation inequality \eqref{ineq:energy1} by considering the contribution of the potential energy. We also state a conditional result about the polynomial decay of the fluid kinetic energy, whenever the Brinkman force enjoys some pointwise decay in $\Ld^2_x$. We finally show how one can obtain the conclusion of Theorem \ref{thCV}, assuming some non-trivial controls on $u_\eps$.
\item Subsection \ref{Subsec:local+strongtime} then introduces the definition of the so-called \textit{strong existence times} which are useful to propagate extra regularity on the weak solutions to the Navier-Stokes equations. Such intervals of strong existence times make some additional integrability results on the system available. 
\item In Subsection \ref{Subsec:launchboot}, we mainly introduce the bootstrap strategy which will be at the heart of the proof of Theorems \ref{thCV}--\ref{th-rate:cvgence}. To do so, we consider the greatest time until which the controls \eqref{control:Nablau} and \eqref{smallnessU-L1Linfty} on $u_\eps$ hold true.
\end{itemize}
\subsection{Decay of the energy functionals and conditional results}\label{Subsec:improvineq+conddecay}
First recall the definition \eqref{def:PotEnergy} of the potential energy $\E_{\varepsilon}^{\mathrm{p}}$. We will use this functional to balance the last term coming from the gravity field in \eqref{ineq:energy1}, thanks to the following lemma.


\begin{lemme}\label{IneqEnergyPot}
Let $\big((u_\eps,f_\eps)\big)_{\eps >0}$ be a family of global weak solutions to the Vlasov-Navier-Stokes system. For any $t \geq 0$ and almost every $0 \leq s \leq t$ (including $s=0$), we have for all $\eps >0$
\begin{align*}
\E_{\varepsilon}^{\mathrm{p}}(t) \leq \E_{\varepsilon}^{\mathrm{p}}(s)+\int_s^t \int_{\R^3_+ \times \R^3}v_3 f_{\eps}(\tau,x,v) \, \mathrm{d}v  \, \mathrm{d}x  \, \mathrm{d}\tau.
\end{align*}
\end{lemme}
\begin{proof}
We rely on the strong stability results from DiPerna-Lions theory about transport equations, $f_\eps$ being a renormalized solution to the Vlasov equation (see e.g. \cite{Misch}). We thus write the proof as if $u_\eps$ and $f_\eps^0$ were smooth and compactly supported. In particular, the characteristic curves are classicaly defined. Thanks to the representation formula given in \eqref{formule:rep}, we have
\begin{align*}
f_{\eps}(t,x,v)=e^{\frac{3t}{\eps}} \mathbf{1}_{\mathcal{O}^t_{\eps}}(x,v)f_\eps^0(\mathrm{X}_{\varepsilon}(0;t,x,v),\mathrm{V}_{\varepsilon}(0;t,x,v)).
\end{align*}
This yields
\begin{align*}
\E_{\varepsilon}^{\mathrm{p}}(t)=e^{\frac{3t}{\eps}} \int_{\mathcal{O}^t_{\eps}} x_3  f_\eps^0(\mathrm{X}_{\varepsilon}(0;t,x,v),\mathrm{V}_{\varepsilon}(0;t,x,v))  \, \mathrm{d}x \, \mathrm{d}v = \int_{\R^3_+ \times \R^3 } \mathrm{X}_\eps(t;0,x,v)_3 \, \mathbf{1}_{\tau^+(0,x,v)>t} \,  f_\eps^0(x,v) \, \mathrm{d}x \, \mathrm{d}v,
\end{align*}
by the change of variables $(x,v) \mapsto (\mathrm{X}_{\varepsilon}(t;0,x,v),\mathrm{V}_{\varepsilon}(t;0,x,v))$ (see Proposition \ref{prop:ChangeLAGRANGIAN}). In view of 
\begin{align*}
\frac{\mathrm{d}}{ \mathrm{d}\tau} \mathrm{X}_\eps(\tau;0,x,v)_3 =\mathrm{V}_\eps(\tau;0,x,v)_3,
\end{align*}
we know that 
\begin{align*}
\mathrm{X}_\eps(t;0,x,v)_3=\mathrm{X}_\eps(s;0,x,v)_3  + \int_s^t \mathrm{V}_\eps(\tau,0,x,v)_3 \, \mathrm{d}\tau,
\end{align*}
therefore for all $s<t$, we have by Fubini theorem
\begin{align*}
\E_{\varepsilon}^{\mathrm{p}}(t) &\leq \int_{\R^3_+ \times \R^3 } \mathrm{X}_\eps(s;0,x,v)_3 \mathbf{1}_{\tau^+(0,x,v)>t}  f_\eps^0(x,v) \, \mathrm{d}x \, \mathrm{d}v \\
& \qquad  \qquad   \qquad   \qquad   \qquad   \qquad  \qquad  \qquad  + \int_s^t \int_{\R^3_+ \times \R^3 } \mathrm{V}_\eps(\tau;0,x,v)_3 \,  \mathbf{1}_{\tau^+(0,x,v)>t} \,   f_\eps^0(x,v) \, \mathrm{d}x \, \mathrm{d}v \, \mathrm{d}\tau \\
 &\leq \int_{\R^3_+ \times \R^3 } \mathrm{X}_\eps(s;0,x,v)_3 \,  \mathbf{1}_{\tau^+(0,x,v)>s} \,   f_\eps^0(x,v) \, \mathrm{d}x \, \mathrm{d}v \\
 & \qquad  \qquad   \qquad   \qquad   \qquad   \qquad  \qquad  \qquad  + \int_s^t \int_{\R^3_+ \times \R^3 } \mathrm{V}_\eps(\tau;0,x,v)_3 \,  \mathbf{1}_{\tau^+(0,x,v)>\tau} \, f_\eps^0(x,v) \, \mathrm{d}x \, \mathrm{d}v \, \mathrm{d}\tau.
\end{align*}
Performing the reverse changes of variable in the two last integrals, we eventually obtain the result.
\end{proof}
Combining Lemma \ref{IneqEnergyPot} with the energy-dissipation inequality \eqref{ineq:energy1} satisfied by any weak solution to the system (in the sense of Definition \ref{weak-sol}), we obtain the following result.
\begin{proposition}
Let $\big((u_\eps,f_\eps)\big)_{\eps >0}$ be a family of global weak solutions to the Vlasov-Navier-Stokes system. For any $t \geq 0$ and almost every $0 \leq s \leq t$ (including $s=0$), we have for all $\eps >0$
\begin{align}
\mathcal{E}_{\eps}(t)+ \int_s^t \mathrm{D}_\eps(\tau) \, \mathrm{d}\tau &\leq \mathcal{E}_{\eps}(s). \label{ineq:energy2}
\end{align}
\end{proposition}

\medskip
Following the seminal result of Wiegner \cite{Wieg} and Borchers and Miyakawa \cite{BM}, we now derive the following conditional result concerning the large time behavior of the fluid kinetic energy. Note that this result only deals with the Navier-Stokes part of the system, treating the Brinkman force as a fixed source term. We refer to \cite[Theorem 3.1]{E} for the details of the proof.
\begin{theoreme}\label{cond:decay}
Let $\big((u_\eps,f_\eps)\big)_{\eps >0}$ be a family of global weak solutions to the Vlasov-Navier-Stokes system. Let $T>0$ and assume that there exists $\eps_0>0$ such that 
\begin{align}
\label{HYPdecay:BRINK} \forall \eps  \in (0,\eps_0), \ \  \forall s \in [0,T], \ \ \Vert j_\eps(s)-\rho_\eps u_\eps(s) \Vert_{\Ld^2(\R^3_+)}  \lesssim  \dfrac{K}{(1+s)^{7/4}},
\end{align}
for some some $K>0$ independent of $T$ and $\eps$. Then there exists a nonnegative nondecreasing continuous function $\Psi$ cancelling at $0$ and independent of $T$ and $\eps$ such that
\begin{align}
\label{eq:decaythmcond}
\forall \eps  \in (0,\eps_0), \ \ \forall s \in [0,T], \ \ \Vert u_\eps(s) \Vert_{\Ld^2(\R^3_+)}^2 \leq \frac{\Psi \left (\Vert  u_{\eps}^0 \Vert_{\Ld^{1}\cap \Ld^2(\R^3_+)}^2+K \right)}{(1+t)^{\frac{3}{2}}}.
\end{align}
\end{theoreme}

\begin{remarque}\label{rmq:conddecay}
In view of the improved energy-dissipation inequality \eqref{ineq:energy2}, it may be reasonable to obtain a conditional (polynomial) decay result on the total energy $\mathcal{E}_\eps$, which takes into account the whole coupling between the Vlasov equation and the Navier-Stokes equations. In the gravity-less case and in the whole space, this idea has been empowered by Han-Kwan in \cite{HK} under the condition of a (uniform) bound on $\rho_\eps$ in $\Ld^{\infty}_t\Ld^{\infty}_x$. This strategy relies on a fine algebraic structure of the whole system. However, in the gravity case on the half-space, an adaptation of this result would require an additional assumption of the potential energy which reads
\begin{align*}
\forall s \in [0,T], \ \ \mathrm{E}^{\mathrm{p}}_{\eps}(s)  \lesssim  \dfrac{1}{(1+s)^{3/2}}.
\end{align*}
Indeed, it seems difficult to control the dissipation of the system from below by a part of the potential energy therefore one must assume \textit{a priori} that this energy has some decay. 

That is why we shall rather use the conditional result stated in Theorem \ref{cond:decay}, which only deals with the decay of the kinetic energy of the fluid part. As we shall see later on, this will be enough for our purpose.
\end{remarque}

\bigskip

We now state a conditional proposition which emphasizes some sufficient conditions leading to the proof of Theorem \ref{thCV}. We mainly combine a classical weak-compactness type argument with the conditional convergence of the Brinkman force provided by Corollary \ref{coro:condBrinkCV}.
\begin{proposition}\label{IF:Propo}
Let $\big((u_\eps,f_\eps)\big)_{\eps >0}$ be a family of global weak solutions to the Vlasov-Navier-Stokes system which are smooth and such that $u_\eps \in  \Ld^1_{\mathrm{loc}}(\R^+;\mathrm{W}^{1,\infty}_0(\R^3_+))$. Let $T>0$. 
Assume that for all $\eps >0$, we have
\begin{align}
\label{C1}\tag{\textbf{C1}} \Vert \nabla_x u_\eps \Vert_{\Ld^1(0,T;\Ld^{\infty}(\R^3_+))} &< \delta  , \\
\label{C2}\tag{\textbf{C2}}\Vert \partial_t u_\eps \Vert_{\Ld^2(0,T;\Ld^2(\R^3_+))} &\leq M, \\
\label{C3}\tag{\textbf{C3}} \mathcal{E}_{\eps}(0)+\Vert (1+\vert v \vert^2) f_\eps^0 \Vert_{\Ld^{\infty} \cap \Ld^1(\R^3_+ \times \R^3)} & \leq M, \\
\label{C4}\tag{\textbf{C4}} \Vert \nabla_x u_\eps \Vert_{\Ld^2(0,T;\Ld^{\infty}(\R^3_+))} &\leq C_T,
\end{align}
where $M>1$ is independent of $\eps$ and $T$, where $C_T>0$ is independent of $\eps$ and where $0<\delta e^{\delta}<1/9$. Then the convergence results stated in Theorem \ref{thCV} hold true on $[0,T]$.
\end{proposition}
\begin{proof}
As the sequence $(u_\eps)_{\eps}$ is bounded in $\Ld^2(0,T;\H^1_0(\R^3_+))$ (thanks to the energy inequality \eqref{ineq:energy2} and \eqref{C3}) there exists $u \in \Ld^2(0,T;\H^1_0(\R^3_+))$ such that, up to a subsequence that we shall not denote here, we have
\begin{align*}
u_\eps  \xrightharpoonup[\eps \to 0]{}  u, \ \ \text{in} \ \ w\text{-}\Ld^2(0,T;\H^1_0(\R^3_+)).
\end{align*}
On the other hand, the sequence $(\rho_\eps)_{\eps}$ is bounded in $\Ld^{\infty}(0,T;\Ld^{\infty}(\R^3_+))$ thanks to Corollary \ref{coro:boundrho} and \eqref{C1}--\eqref{C3}. Therefore there exists $\rho \in \Ld^{\infty}(0,T;\Ld^{\infty}(\R^3_+))$ such that, up to a subsequence, we have
\begin{align*}
\rho_\eps \xrightharpoonup[\eps \to 0]{}  \rho, \ \ \text{in} \ \ w^{\ast}\text{-}\Ld^{\infty}(0,T;\Ld^{\infty}(\R^3_+)).
\end{align*}

Now, let $K$ be a compact subset of $\R^3_+$. By the Aubin-Lions lemma (which holds because $(\partial_t u_\eps)_{\eps}$ is bounded in $\Ld^2(0,T;\Ld^2(\R^3_+))$ thanks to \eqref{C2}), we deduce that, up to another extraction, $(u_\eps)_{\eps}$ converges strongly to $u$ in $\Ld^2((0,T) \times K)$. In particular, we have the convergence of the product
\begin{align*}
\rho_\eps u_\eps \xrightharpoonup[\eps \to 0]{}  \rho u, \ \ \text{in} \ \ w\text{-}\Ld^2((0,T) \times K).
\end{align*}


In a second part, we use the conservation of mass for the particles which yields 
\begin{align}\label{eq:massCOND}
\partial_t \rho_{\eps} + \mathrm{div}_x \,  j_{\eps}=0, \ \ \text{in} \ \ \mathscr{D}'([0,T) \times \R^3_+).
\end{align}
Using Corollary \ref{coro:condBrinkCV}, we observe that the Brinkman force $(F_\eps)_\eps=(j_{\eps}-\rho_{\eps}u_\eps)_\eps$ converges to $-\rho e_3$ in $\Ld^2((0,T)\times \R^3_+)$ when $\eps \rightarrow 0$ (because of the energy inequality \eqref{ineq:energy2} and \eqref{C1}--\eqref{C2}--\eqref{C3}--\eqref{C4}) and thus in $\Ld^2((0,T)\times K)$. This implies that
\begin{align*}
j_\eps \xrightharpoonup[\eps \to 0]{}  \rho (u-e_3), \ \ \text{in} \ \ w\text{-}\Ld^2((0,T) \times K).
\end{align*}
This last convergence and the previous convergence of $(\rho_\eps)_{\eps}$ allow one to pass to the limit in the equation \eqref{eq:massCOND}: we get
\begin{align*}
\partial_t \rho + \mathrm{div}_x \, \left[ \rho(u-e_3) \right]=0, \ \ \text{in} \ \ \mathscr{D}'((0,T) \times \R^3_+).
\end{align*}
Finally, the aforementioned strong convergence of $(u_\eps)_{\eps}$ towards $u$ is enough to classically pass to the limit in the divergence-free condition, in the l.h.s of the Navier-Stokes equations and in the source term $j_{\eps}-\rho_{\eps}u_\eps$ thanks to the previous convergence of the Brinkman force. 

\end{proof}
\begin{remarque}\label{rmk:conditionIFthm}
The previous result provides a guideline for the proof of Theorem \ref{thCV} in large time.

In view of the Assumption \textbf{\ref{hypUnifBoundVNS}} we shall make on the initial data, it goes without saying that the conditional hypothesis \eqref{C1} and \eqref{C4} of Proposition \ref{IF:Propo} are the most difficult to obtain and constitute the main issue of the analysis.

Thanks to an interpolation argument of the type
$$ \Vert \nabla_x u_\eps(s) \Vert_{\Ld^{\infty}(\R^3_+)} \lesssim \Vert \D^2_x u_\eps(s) \Vert_{\Ld^{p}(\R^3_+)}^{\beta_p} \Vert u_\eps(s) \Vert_{\Ld^{2}(\R^3_+)}^{1-\beta_p}, \ \  p>3, \ \ \beta_p \in (0,1),
$$
 we will essentially prove that a decay of $u_\eps$ under the form
$$ \forall s \in [0,T], \ \ \Vert u_\eps(s) \Vert_{\Ld^2(\R^3_+)}^2 \leq \frac{1}{(1+t)^{\frac{3}{2}}}$$
 shall imply \eqref{C1} and \eqref{C4} on $(0,T)$. In view of Theorem \ref{cond:decay}, this will be ensured provided that the Brinkman force satisfies a decay like 
\begin{align}
\label{C5}\tag{\textbf{C5}} \forall s \in [0,T], \ \ \Vert j_\eps(s)-\rho_\eps u_\eps(s) \Vert_{\Ld^2(\R^3_+)}  \lesssim  \dfrac{K}{(1+s)^{7/4}},
\end{align}
where $K>0$ is independent of $T$ and $\eps$.

Because of the slow decay in time of $u_\eps$, we will actually need a refined argument requiring, through maximal regularity estimates (see Section \ref{AnnexeMaxregStokes} in Appendix), a polynomial decay in time of the Brinkman force of the form
\begin{align*}
\label{C6}\tag{\textbf{C6}} \Vert (1+t)^{\gamma} (j_\eps-\rho_\eps u_\eps) \Vert_{\Ld^p(0,T; \Ld^p(\R^3_+))} \lesssim K, \ \ p>3.
\end{align*}
In short, we have
\begin{align*}
\eqref{C5} \ \text{and} \ \eqref{C6} \ \text{on} \ (0,T) \Longrightarrow \eqref{C1} \ \text{and} \ \eqref{C4} \ \text{on} \ (0,T).
\end{align*}
We refer to the beginning of Subsection \ref{Subsec:Weighted} for more details about the previous implication.

According to the conditional Proposition \ref{IF:Propo}, ensuring \eqref{C5}--\eqref{C6} is the key of our analysis and will lead to a proof of Theorem \ref{thCV}. The main mechanism leading to \eqref{C5}--\eqref{C6} will be the combination of the absorption boundary condition and the gravity effect.
\end{remarque}
\subsection{Local estimates and strong existence times}\label{Subsec:local+strongtime}
We now introduce the notion of interval of strong existence for the Navier-Stokes equations. It will enable us to consider higher regularity estimates for the system, which will be crucial in the final bootstrap strategy. It is based on a parabolic smoothing effect for the equations and roughly states the instantaneous gain of two derivatives (in space) for a solution to the Navier-Stokes equations. This requires that the forcing term (i.e. the Brinkman force $j_\eps -\rho_\eps u_\eps$) and the initial data enjoy some smallness properties. We refer to (a small variant of) \cite[Theorem A.8]{E} for a proof. 

\medskip

Recall the notation
$$F_\eps=j_\eps -\rho_\eps u_\eps$$
for the Brinkman force, $(u_\eps, f_\eps)$ being any global weak solution to the VNS system.
\begin{proposition}\label{propdatasmall:VNSreg}
There exists a universal constant $\mathrm{C}_{\star}$ such that the following holds. Let $\big((u_\eps,f_\eps)\big)_{\eps >0}$ be a family of global weak solutions to the Vlasov-Navier-Stokes system. Assume that for some $T>0$, one has
\begin{align}\label{datasmall:VNSreg}
\Vert  u_\eps^0 \Vert_{\H^1(\R^3_+)}^2 +  \int_0^T \Vert F_\eps(s) \Vert_{\Ld^2(\R^3_+)}^2 \, \mathrm{d}s +\int_0^T \Vert F_\eps(s) \Vert_{\Ld^2(\R^3_+)} \, \mathrm{d}s< \mathrm{C}_{\star}.
\end{align}
Then one has 
\begin{align*}
u_\eps &\in \Ld^{\infty}(0,T; \H^1(\R^3_+)) \cap \Ld^{2}(0,T; \H^2(\R^3_+)), \\
\partial_t u_\eps &\in \Ld^{2}(0,T; \Ld^2(\R^3_+)),
\end{align*}
and for all $t \in [0,T]$
\begin{multline}\label{ineq:VNSreg}
\Vert \nabla_x u_\eps(t) \Vert_{\Ld^2(\R^3_+)}^2 + \int_0^t \Vert \mathrm{D}^2_x u_\eps(s) \Vert_{\Ld^2(\R^3_+)}^2 \, \mathrm{d}s +\int_0^t \Vert \partial_t u_\eps(s) \Vert_{\Ld^2(\R^3_+)}^2 \, \mathrm{d}s   \\
\lesssim \Vert \nabla_x u_\eps^0 \Vert_{\Ld^2(\R^3_+)}^2 + \int_0^t \Vert F_\eps(s) \Vert_{\Ld^2(\R^3_+)}^2 \, \mathrm{d}s,
\end{multline}
where $\lesssim$ only depends on $\mathrm{C}_{\star}$.
\end{proposition}
Recall that we work under Assumption \ref{hypSmallDataSTRONG} on the initial data ensuring that 
\begin{align*}
\forall\eps>0, \qquad\Vert{u_{\eps}^0}\Vert_{\H^{1}(\R^3_+)}^2 < \frac{\mathrm{C}_{\star}}{2}.
\end{align*}
In order to prove that the smallness condition (\ref{datasmall:VNSreg}) is satisfied for all times, we now introduce the notion of \textit{strong existence times} for the Vlasov-Navier-Stokes system.
\begin{definition}[Strong existence time]\label{strongtime}
Let $\eps >0$. A real number $T\geq 0$ is a \emph{strong existence time} (for a global weak solution $(u_\eps,f_\eps)$) whenever the inequality
\begin{align*}
  \int_0^T \Vert F_\eps(s) \Vert_{\Ld^2(\R^3_+)}^2 \, \mathrm{d}s +\int_0^T \Vert F_\eps(s) \Vert_{\Ld^2(\R^3_+)} \, \mathrm{d}s < \frac{\mathrm{C}_{\star}}{2},
\end{align*}
holds.
\end{definition}

\begin{definition}
For all $\eps>0$ and $T>0$, we set 
\begin{align*}
\Upsilon_\eps^0 (T):=\Vert u_0^\eps \Vert_{\H^1(\R^3_+)}^2+ \int_0^T \Vert F_\eps(s) \Vert_{\Ld^2(\R^3_+)}^2 \, \mathrm{d}s +\int_0^T \Vert F_\eps(s) \Vert_{\Ld^2(\R^3_+)} \, \mathrm{d}s.
\end{align*}
\end{definition}
\begin{remarque}\label{bound-psi:eps}
If $T$ is a strong existence time in the sense of Definition \ref{strongtime}, then 
\begin{align*}
\forall\eps>0, \ \forall t \in [0,T], \ \  \Upsilon_\eps^0 (t) < \mathrm{C}_{\star},
\end{align*}
which is therefore a uniform bound in $\eps$ and $t$. Thus, in what follows, we shall use the harmless notation $ \Upsilon_\eps^0$ without mentioning the time $t$.
\end{remarque}
\begin{remarque}
Note that the parabolic smoothing used in \cite{HKM} for the torus case is rather based on the standard Fujita-Kato type smallness assumption for the Navier-Stokes system, namely requiring that the initial data $u_\eps^0$ is small in $\dot\H^{\frac{1}{2}}(\T^3)$. It should be possible to relax the $\H^1$ assumption of \eqref{datasmall:VNSreg} but, for the sake of simplicity, we have preferred avoiding such technical details.
\end{remarque}




\medskip

A straightforward reformulation of Proposition \ref{propdatasmall:VNSreg} combined with Sobolev embedding leads to the following result.
\begin{corollaire}\label{D_tu/D2:uL2}
For any finite strong existence times $T>0$ of a global weak solution $(u_\eps,f_\eps)$ , we have
\begin{align*}
\Vert \partial_t u_\eps \Vert_{\Ld^{2}(0,T; \Ld^2(\R^3_+))}^2 + \Vert\mathrm{D}^2_x u_\eps \Vert_{\Ld^{2}(0,T; \Ld^2(\R^3_+))}^2 &\lesssim 2 \mathrm{C}_{\star}, \\[2mm]
\Vert u_\eps \Vert_{\Ld^{\infty}(0,t; \Ld^6(\R^3_+))}^2 \lesssim \Vert \nabla_x u_\eps \Vert_{\Ld^{\infty}(0,t; \Ld^2(\R^3_+))}^2 &\lesssim \mathrm{C}_{\star}.
\end{align*}
\end{corollaire}

\begin{remarque}\label{rmk:justifIPP}
Let us explain how one can now make the computation of Subsection \ref{Subsec:SplitBrink} rigorous, justifying in particular the integration by parts in time at the heart of the desingularization in $\eps$ of the Brinkman force (see \eqref{GammaDesing}). The main point is that we shall perform this procedure on intervals of time which are strong existence times so that $\partial_t u_\eps \in \Ld^2_T \Ld^2_x$ on theses intervals, in view of Proposition \ref{propdatasmall:VNSreg}. The exponential factor in the integral involved in $\Gamma_\eps^{t,x}$ being as smooth as we want, the computation is allowed thanks to well known properties of Sobolev functions in time with value in Banach spaces.
\end{remarque}

\medskip

Dealing with strong existence times also provides some useful integrability estimates on the solutions to the Vlasov-Navier-Stokes system. We shall use the following ones.
\begin{corollaire}\label{coro:estimatesSTRONG}
Let $T$ be a finite strong existence time of a global weak solution $(u_\eps,f_\eps)$. Then
\begin{itemize}
\item for any $p \in [1,6]$, we have 
\begin{align*}
j_\eps-\rho_\eps u_\eps \in \Ld^p(0,T;\Ld^p(\R^3_+));
\end{align*}
\item there exists $\varsigma>0$ and $\mu>0$ such that for all $p \in (3,3+\varsigma)$, we have for all $t \in (0,T)$
\begin{align*}
\Vert (u_\eps \cdot \nabla_x)u_\eps(t) \Vert_{\Ld^2(\R^3_+)} &\lesssim (\Upsilon_\eps^0)^{\frac{1}{2}}\mathcal{E}_\eps(0)^{\frac{1}{4}} \Vert \D^2_x u_\eps(t) \Vert_{\Ld^2(\R^3_+)}, \\[2mm]
\Vert (u_\eps \cdot \nabla_x)u_\eps(t) \Vert_{\Ld^p(\R^3_+)} &\lesssim (\Upsilon_\eps^0)^{\varsigma_p} \mathcal{E}_\eps(0)^{\mu} \Vert \D^2_x u_\eps(t) \Vert_{\Ld^p(\R^3_+)},
\end{align*}
for some $\varsigma_p>0$;
\item for the exponent $p$ given in Assumption \textbf{\ref{hypGeneral}}, we have
 \begin{align}
 \label{convecInLPLP} (u_\eps \cdot  \nabla_x )u_\eps &\in \Ld^p(0,T;\Ld^p(\R^3_+)), \\[2mm]
 \label{D_t/2uIn LPLP} \partial_t u_\eps, \, \D^2_x u_\eps &\in \Ld^p(0,T;\Ld^p(\R^3_+));
 \end{align}
\item we have
\begin{align} \label{nablauInL1Linfty}
\nabla_x u_\eps \in \Ld^1(0,T; \Ld^{\infty}(\R^3_+)).
\end{align}
\end{itemize}
\end{corollaire}
\begin{proof}
We refer to \cite{HK} and \cite{E}. Note that we use Assumption \ref{hypGen:regBesov} to ensure such integrability results.
\end{proof}
\subsection{Bootstrap procedure}\label{Subsec:launchboot}

Before setting up a bootstrap procedure, we aim at obtaining further local in time integrability results. We mainly refer to \cite[Subsection 4.2]{E}. Indeed, the proofs performed in \cite{E} are the same, with a fixed parameter $\eps$ appearing at some points. All the local in time estimates actually blow up with $\eps \rightarrow 0$ but this is harmless since we shall not use quantitative estimates for the moment, arguing only with integrability properties.
\begin{proposition}
Let $\eps >0$. Suppose that $\vert v \vert^6 f_\eps^0  \in \Ld^1(\R^3_+ \times \R^3)<\infty$. For any global weak solution $(u_\eps,f_\eps)$, we have
\begin{align}
\label{integL2:Brink}&F_\eps \in \Ld^2_{\mathrm{loc}}(\R^+; \Ld^2(\R^3_+)), \\[2mm]
\label{integu:L1infty}&u_\eps \in \Ld^1_{\mathrm{loc}}(\R^+; \Ld^{\infty}(\R^3_+)),\\[2mm]
&\rho_\eps \in \Ld^{\infty}_{\mathrm{loc}}(\R^+; \Ld^{\infty}(\R^3_+)).
\end{align}

\begin{remarque}\label{Exist:strongtimes}
In particular, under Assumption \textbf{\ref{hypGeneral}}, for all $\eps >0$ there exists $T_\eps>0$ such that
\begin{align*}
\left(1+\sqrt{T_\eps}\right)\int_0^{T_\eps} \Vert F_\eps(s) \Vert_{\Ld^2(\R^3_+)}^2 \, \mathrm{d}s < \frac{\mathrm{C}_{\star}}{2},
\end{align*}
and this also means that for all $\eps>0$, there exists a positive strong existence time.
\end{remarque}

\end{proposition}

\textbf{Until the end of this work, we consider a fixed family $\big((u_\eps,f_\eps)\big)_{\eps >0}$ of global weak solutions to the Vlasov-Navier-Stokes system, in the sense of Definition \ref{weak-sol}, associated to an admissible initial data $(u_\eps^0, f_\eps^0)$ and satisfying Assumptions \textbf{\ref{hypGeneral}}--\textbf{\ref{hypUnifBoundVNS}}--\textbf{\ref{hypSmallData}}}.

\medskip

In view of the conditional Proposition \ref{IF:Propo} and Remark \ref{rmk:conditionIFthm}, we will follow a strategy based on a bootstrap argument and which requires the following definition.

\begin{definition}\label{def:tstar}
Let $\alpha \in (0,1)$ be fixed. For any $\eps>0$, we set
\begin{align}
t^{\star}_\eps :=\sup \left\lbrace \text{strong existence times } t >0 \text{ such that } \int_{0} ^t \Vert  u_\eps(s) \Vert_{\mathrm{W}^{1,\infty}(\R^3_+)}  \, \mathrm{d}s < \delta^{\star} \right\rbrace,
\end{align}
where $\delta^{\star}:=\min(\kappa_{\alpha}, \delta)$ is defined in the following way: $\kappa_{\alpha}$ refers to the constant of Lemma \ref{LM:perturbEGC} and $\delta$ is choosen such that $\delta e^{\delta}<\frac{1}{9}$ (see in particular Lemma \ref{LM:ChgmtVAR-V}).
\end{definition}
\begin{lemme}
For all $\eps>0$, we have $t^{\star}_\eps>0$. 
\end{lemme}
\begin{proof}
According to Remark \ref{Exist:strongtimes}, there exists a strong existence time $T_\eps>0$. By \eqref{nablauInL1Linfty} of Corollary \ref{coro:estimatesSTRONG}, we know that $\nabla_x u_\eps \in \Ld^1(0,T_\eps;\Ld^{\infty}(\R^3_+))$ while $ u_\eps \in \Ld^1(0,T_\eps;\Ld^{\infty}(\R^3_+))$ by \eqref{integu:L1infty}, therefore a continuity in time argument shows there exists $\underline{T_\eps} \in (0,T_\eps)$ such that
\begin{align*}
\int_{0} ^{\underline{T_\eps}} \Vert  u_\eps(s) \Vert_{\Ld^{\infty}(\R^3_+)}  \, \mathrm{d}s < \frac{\delta^{\star}}{2}, \ \ \int_{0} ^{\underline{T_\eps}} \Vert  \nabla_x u_\eps(s) \Vert_{\Ld^{\infty}(\R^3_+)}  \, \mathrm{d}s < \frac{\delta^{\star}}{2}.
\end{align*}
Since $\underline{T_\eps}$ is still a strong existence time, this concludes the proof by definition of $t^{\star}_\eps$.
\end{proof}
Our main goal is now to show that $t^{\star}_{\eps} = + \infty$, at least for any $\eps$ small enough. This will require a finite number of use of Assumptions \textbf{\ref{hypUnifBoundVNS}}--\textbf{\ref{hypSmallData}} bearing on the initial data (see in the following Sections \ref{Section:estimateBrinkman}--\ref{Section:Boostrap}) and we will be able to consider global weak solutions arising from such initial data.

\section{Estimates and decay of the Brinkman force}\label{Section:estimateBrinkman}
The purpose of this section is twofold: having in mind the strategy described in Subsection \ref{subsection:Strat} and in the end of Subsection \ref{Subsec:improvineq+conddecay} (see in particular Remark \ref{rmk:conditionIFthm}), we want to provide
\begin{itemize}
\item pointwise decay in time estimates for $F_\eps$ in $\Ld^2_x$, thanks to the absorption effect at the boundary. Since we do not yet have access to the conditional decay in time of $u_\eps$ in $\Ld^2_x$ provided by Theorem \ref{cond:decay}, we shall rely on the energy inequality \eqref{ineq:energy2}. This first step is performed in Subsection \ref{Subsect:FpointL2}. Note that we shall start by this very first procedure in order to ensure the polynomial decay of the kinetic energy of the fluid afterwards.
\item decay in time estimates for $F_\eps$ in $\Ld^p_t \Ld^p_x$, thanks to the absorption effect at the boundary and the polynomial decay in time of $u_\eps$ in $\Ld^2_x$ provided by Subsection \ref{Subsect:FpointL2} and Theorem \ref{cond:decay}. These estimates are derived in Subsection \ref{Subsect:F-LpLp}.

\end{itemize}
The main starting point in order to establish such estimates is the use of the Lagrangian framework of Section \ref{Section:Trajectoires}.
We shall refine the computations of Subsection \ref{Subsec:SplitBrink} for the Brinkman force.
Note that the statements of that subsection only provided bounds ensuring the convergence of $F_\eps + \rho_\eps  e_3$ when $\eps \rightarrow 0$.

\medskip

As in Subsection \ref{Subsec:SplitBrink}, we start writing  
\begin{align*}
F_\eps(t,x)&=e^{\frac{3t}{\eps}} \int_{\R^3} \mathbf{1}_{\mathcal{O}^t_{\eps}}(x,[\Gamma_\eps^{t,x}]^{-1}(w))f_\eps^0(\widetilde{\mathrm{X}}_{\eps}^{0;t}(x,w),w) \left( [\Gamma_\eps^{t,x}]^{-1}(w)-u_\eps(t,x) \right) \vert \det \, \D_w [\Gamma_\eps^{t,x}]^{-1}(w) \vert \, \mathrm{d}w,
\end{align*}
and perform a splitting of the integral thanks to the identity \eqref{GammaDesing} on $[\Gamma_\eps^{t,x}]^{-1}$ . In view of Definition \ref{def:tstar} and Lemma \ref{LM:ChgmtVAR-V}, this procedure will be valid for times $t <t_{\eps}^{\star}$.

\medskip

To go further and obtain some decay estimates from the previous expression of $F_\eps$, we shall rely on the absorption condition at the boundary which is encoded in the indicator $\mathbf{1}_{\mathcal{O}^t_{\eps}}(x,[\Gamma_\eps^{t,x}]^{-1}(w))$. Our strategy is crucially based upon the exit geometric condition of Subsection \ref{Subsection:EGCabs}. It will provide some quantitative decay in time estimates for $F_\eps$, thanks to the decay in the phase space of $f_\eps^0$ itself.

\medskip

\medskip

\textbf{Use of the absorption.} 

In view of Lemma \ref{EGCt:reverse}, assume that there exists $T_0>0$ such that for $\eps$ small enough, we have $T_0<t_\eps^{\star}$ and such that for all $t \in (T_0,t_{\eps}^{\star})$
\begin{align}\label{gen:abs1}
u_\eps \ \  \text{satisfies} \ \ \mathrm{EGC}_{\eps}^{1+\Ld(t),1+\mathrm{R}(t)}(t),
\end{align}
 for some continuous and positive functions $\Ld$ and $\mathrm{R}$ satisfying
\begin{align}\label{gen:abs2}
\forall t \in (T_0,t_\eps^{\star}), \ \ \dfrac{1}{1+\Ld(t)} \lesssim \dfrac{1}{1+t}, \ \  \dfrac{1}{1+\mathrm{R}(t)} \lesssim \dfrac{1}{1+t}.
\end{align}

Then, according to Proposition \ref{PropoABS-form}, we can write
\begin{align*}
f_\eps(t,x,v) \leq f_\eps^{\natural}(t,x,v)+f_\eps^{\flat}(t,x,v),
\end{align*}
where 
\begin{align*}
f_\eps^{\natural}(t,x,v)&:=e^{\frac{3t}{\eps}} \mathbf{1}_{\mathcal{O}^t_{\eps}}(x,v) \,  \mathbf{1}_{\vert \V_\eps(0;t,x,v) \vert >1+\mathrm{R}(t)}  \,  f_\eps^0(\X_\eps(0;t,x,v),\V_\eps(0;t,x,v)), \\[2mm]
f_\eps^{\flat}(t,x,v)&:=e^{\frac{3t}{\eps}} \mathbf{1}_{\mathcal{O}^t_{\eps}}(x,v) \,  \mathbf{1}_{ \X_\eps(0;t,x,v)_3 >1+\mathrm{L}(t)}  \,  f_\eps^0(\X_\eps(0;t,x,v),\V_\eps(0;t,x,v)).
\end{align*}
From the previous splitting, we can infer
\begin{align*}
\vert F_\eps(t,x) \vert  &\leq\int_{\R^3} f_\eps^{\natural}(t,x,v) \vert v-u_\eps(t,x) \vert  \, \mathrm{d}v +\int_{\R^3} f_\eps^{\flat}(t,x,v) \vert v-u_\eps(t,x) \vert  \, \mathrm{d}v  \\
& =\int_{\R^3} e^{\frac{3t}{\eps}} \mathbf{1}_{\mathcal{O}^t_{\eps}}(x,v) \,  \mathbf{1}_{\vert \V_\eps(0;t,x,v) \vert >1+\mathrm{R}(t)}  \,  f_\eps^0(\X_\eps(0;t,x,v),\V_\eps(0;t,x,v)) \vert v-u_\eps(t,x) \vert  \, \mathrm{d}v \\
& \quad + \int_{\R^3} e^{\frac{3t}{\eps}} \mathbf{1}_{\mathcal{O}^t_{\eps}}(x,v) \,  \mathbf{1}_{ \X_\eps(0;t,x,v)_3 >1+\mathrm{L}(t)}  \,  f_\eps^0(\X_\eps(0;t,x,v),\V_\eps(0;t,x,v)) \vert v-u_\eps(t,x) \vert  \, \mathrm{d}v.
\end{align*}
Arguing exactly as in Subsection \ref{Subsec:SplitBrink}, we have the following splitting lemma for the Brinkman force $F_\eps$, which in the same spirit as that of Lemma \ref{Split-F1}..
\begin{lemme}\label{Split-F2}
Assume that \eqref{gen:abs1}--\eqref{gen:abs2} hold with respect to a time $T_0$. For any $\eps>0$, if $T \in (T_0,t_\eps^{\star})$ is a strong existence time, then for all $(t,x) \in (T_0,T) \times \R^3_+$
\begin{align*}
\vert F_\eps(t,x) \vert &\lesssim \sum_{i=0}^2   F_\eps^{\natural,i}(t,x) + \sum_{i=0}^2  F_\eps^{\flat,i}(t,x),
\end{align*}
where 
\begin{align*}
F_\eps^{\natural,0}(t,x)&:=e^{-\frac{t}{\eps}}  \int_{\R^3} \mathbf{1}_{\mathcal{O}^t_{\eps}}(x,[\Gamma_\eps^{t,x}]^{-1}(w)) \,  \mathbf{1}_{\vert w \vert >1+\mathrm{R}(t)}  \, f_\eps^0(\widetilde{\mathrm{X}}_{\eps}^{0;t}(x,w),w)\Big\vert w+e_3-(P u_\eps)(0,\widetilde{\mathrm{X}}_{\eps}^{0;t}(x,w))\Big\vert \, \mathrm{d}w \\
& \quad  +  \int_{\R^3} \mathbf{1}_{\mathcal{O}^t_{\eps}}(x,[\Gamma_\eps^{t,x}]^{-1}(w)) \,  \mathbf{1}_{\vert w \vert >1+\mathrm{R}(t)}  \, f_\eps^0(\widetilde{\mathrm{X}}_{\eps}^{0;t}(x,w),w) \, \mathrm{d}w, \\[2mm]
F_\eps^{\natural,1}(t,x)&:= \int_{\R^3}\mathbf{1}_{\mathcal{O}^t_{\eps}}(x,[\Gamma_\eps^{t,x}]^{-1}(w)) \,  \mathbf{1}_{\vert w \vert >1+\mathrm{R}(t)}   \int_0^t  e^{\frac{\tau-t}{\eps}} f_\eps^0(\widetilde{\mathrm{X}}_{\eps}^{0;t}(x,w),w)\Big\vert \partial_\tau [P u_\eps](\tau,\widetilde{\mathrm{X}}_\eps^{\tau;t}(x,w)) \Big\vert \, \mathrm{d}\tau \, \mathrm{d}w, \\[2mm]
F_\eps^{\natural,2}(t,x)&:= \int_{\R^3}\mathbf{1}_{\mathcal{O}^t_{\eps}}(x,[\Gamma_\eps^{t,x}]^{-1}(w)) \,  \mathbf{1}_{\vert w \vert >1+\mathrm{R}(t)}  \, \int_0^t  e^{\frac{\tau-t}{\eps}} f_\eps^0(\widetilde{\mathrm{X}}_{\eps}^{0;t}(x,w),w) \\
& \qquad \qquad  \qquad  \qquad  \qquad \qquad \qquad \quad   \times \Big\vert \left(\mathrm{V}_\eps(\tau;t,x,[\Gamma_\eps^{t,x}]^{-1}(w)) \cdot \nabla_x \right) [P u_\eps](\tau,\widetilde{\mathrm{X}}_\eps^{\tau;t}(x,w))  \Big\vert \, \mathrm{d}\tau \, \mathrm{d}w,
\end{align*}
and
\begin{align*}
F_\eps^{\flat,0}(t,x)&:=e^{-\frac{t}{\eps}}  \int_{\R^3} \mathbf{1}_{\mathcal{O}^t_{\eps}}(x,[\Gamma_\eps^{t,x}]^{-1}(w))  \,  \mathbf{1}_{ \widetilde{\mathrm{X}}_{\eps}^{0;t}(x,w)_3 >1+\mathrm{L}(t)}  \, f_\eps^0(\widetilde{\mathrm{X}}_{\eps}^{0;t}(x,w),w)\Big\vert w+e_3-(P u_\eps)(0,\widetilde{\mathrm{X}}_{\eps}^{0;t}(x,w))\Big\vert \, \mathrm{d}w \\
& \quad + \int_{\R^3} \mathbf{1}_{\mathcal{O}^t_{\eps}}(x,[\Gamma_\eps^{t,x}]^{-1}(w))  \,  \mathbf{1}_{ \widetilde{\mathrm{X}}_{\eps}^{0;t}(x,w)_3 >1+\mathrm{L}(t)}  \, f_\eps^0(\widetilde{\mathrm{X}}_{\eps}^{0;t}(x,w),w) \, \mathrm{d}w, \\[2mm]
F_\eps^{\flat,1}(t,x)&:= \int_{\R^3}\mathbf{1}_{\mathcal{O}^t_{\eps}}(x,[\Gamma_\eps^{t,x}]^{-1}(w)) \,  \mathbf{1}_{ \widetilde{\mathrm{X}}_{\eps}^{0;t}(x,w)_3 >1+\mathrm{L}(t)} \int_0^t  e^{\frac{\tau-t}{\eps}} f_\eps^0(\widetilde{\mathrm{X}}_{\eps}^{0;t}(x,w),w)\Big\vert \partial_\tau [P u_\eps](\tau,\widetilde{\mathrm{X}}_\eps^{\tau;t}(x,w)) \Big\vert \, \mathrm{d}\tau \, \mathrm{d}w, \\[2mm]
F_\eps^{\flat,2}(t,x)&:= \int_{\R^3}\mathbf{1}_{\mathcal{O}^t_{\eps}}(x,[\Gamma_\eps^{t,x}]^{-1}(w)) \,  \mathbf{1}_{ \widetilde{\mathrm{X}}_{\eps}^{0;t}(x,w)_3 >1+\mathrm{L}(t)} \int_0^t  e^{\frac{\tau-t}{\eps}} f_\eps^0(\widetilde{\mathrm{X}}_{\eps}^{0;t}(x,w),w) \\
& \qquad \qquad  \qquad  \qquad  \qquad \qquad \qquad \quad  \times \Big\vert \left(\mathrm{V}_\eps(\tau;t,x,[\Gamma_\eps^{t,x}]^{-1}(w)) \cdot \nabla_x \right) [P u_\eps](\tau,\widetilde{\mathrm{X}}_\eps^{\tau;t}(x,w))  \Big\vert \, \mathrm{d}\tau \, \mathrm{d}w.
\end{align*}
\end{lemme}
We shall estimate the contribution of all these terms, by performing a change of variable in space based on Lemma \ref{LM:ChgmtVAR-X}.

%


\medskip

\textbf{Choice of the initial time.}

In this entire section, we consider a time $0 \leq T_0<t_\eps^{\star}$ (for $\eps$ small enough) such that all the subsequent estimates are performed on $(T_0,t_\eps^{\star})$. There are mainly two cases:
\begin{itemize}
\item if $T_0=0$, we do not use the absorption effect and do not rely on the exit geometric condition (namely by dropping the indicators involving $\mathrm{L}(t)$ and $\mathrm{R}(t)$ from the previous formulas). 
\item if $0<T_0<t_\eps^{\star}$ is such that the general conditions of absorption \eqref{gen:abs1} and \eqref{gen:abs2} are satisfied, we shall refer to $T_0$ as \textit{a starting time of absorption} (with respect to the functions $\Ld$ and $\mathrm{R}$). Of course, this is the interesting case in which we hope for some decay estimates to hold (in large time). Note that if we restrict ourselves to uniform in time bounds (without any decay), this procedure comes back to the case $T_0=0$.
\end{itemize}
Later on, we shall specify a time after which the absorption effect can be handled (see Definition \ref{def:T_0}). In what follows, we will mainly state all the results on $(T_0,t_\eps^{\star})$ where $T_0$ is a starting time of absorption. When it is useful for our purpose, we shall mention the result in the case $T_0=0$.

\subsection{Pointwise in time estimates of the Brinkman force in $\Ld^2_x$}\label{Subsect:FpointL2}
In the current subsection, we are not yet allowed to use the conditional decay of the kinetic energy of the fluid stated in Theorem \ref{cond:decay}. Nevertheless, we will obtain pointwise decay estimates for the Brinkman force, the main tool being the energy inequality \eqref{ineq:energy2}.

\begin{lemme}\label{LM:F0-pointL2}
For all $t \in (T_0,t^{\star}_\eps)$ and any $k \geq 0$, we have
\begin{align*}
 \Vert F_\eps^{\natural,0}(t) \Vert_{\Ld^2(\R^3_+)} &\lesssim \frac{e^{\frac{-t}{\eps}}}{(1+t)^{k}} \Vert \vert v \vert^{k} f_\eps^0 \Vert_{\Ld^1(\R^3;\Ld^{\infty}(\R^3_+))}^{\frac{1}{2}}\Big[\Vert (1+\vert v \vert^{k+2}) f_\eps^0 \Vert_{\Ld^1(\R^3_+ \times \R^3)}  + \Vert \vert v \vert^{k} f_\eps^0 \Vert_{\Ld^1(\R^3;\Ld^{\infty}(\R^3_+))} \Vert u_\eps^0 \Vert_{\Ld^2(\R^3_+)}^2 \Big]^{\frac{1}{2}} \\
& \quad + \frac{1}{(1+t)^{k}} \Vert \vert v \vert^k   f_\eps^0 \Vert_{\Ld^1(\R^3;\Ld^2(\R^3_+))}, \\[2mm]
 \Vert F_\eps^{\flat,0}(t) \Vert_{\Ld^2(\R^3_+)}  &\lesssim \frac{e^{\frac{-t}{\eps}}}{(1+t)^{k}} \Vert x_3^{k} f_\eps^0 \Vert_{\Ld^1(\R^3;\Ld^{\infty}(\R^3_+))}^{\frac{1}{2}}\Big[\Vert (1+\vert v \vert^{2}) x_3^{k} \, f_\eps^0 \Vert_{\Ld^1(\R^3_+ \times \R^3)}  + \Vert  x_3^{k} f_\eps^0 \Vert_{\Ld^1(\R^3;\Ld^{\infty}(\R^3_+))} \Vert u_\eps^0 \Vert_{\Ld^2(\R^3_+)}^2 \Big]^{\frac{1}{2}}\\
& \quad + \frac{1}{(1+t)^{k}} \Vert  x_3^{k} \,    f_\eps^0 \Vert_{\Ld^1(\R^3;\Ld^2(\R^3_+))}.
\end{align*}
\end{lemme}
\begin{proof}
We focus on the proof of the first estimate, the proof of the second one being similar. We have for all $t \in (T_0,t^{\star}_\eps)$
\begin{align*}
\Vert F_\eps^{\natural,0}(t)\Vert_{\Ld^2(\R^3_+)} 
&\leq \frac{e^{-\frac{t}{\eps}}}{(1+R(t))^{k}}  \Bigg[\int_{\R^3_+} \Bigg\vert \int_{\R^3} \mathbf{1}_{\mathcal{O}^t_{\eps}}(x,[\Gamma_\eps^{t,x}]^{-1}(w))  \vert w \vert^{k} f_\eps^0(\widetilde{\mathrm{X}}_{\eps}^{0;t}(x,w),w) \\
& \qquad \quad  \qquad \qquad \qquad \qquad \qquad \qquad \times \Big\vert w+e_3-(P u_\eps)(0,\widetilde{\mathrm{X}}_{\eps}^{0;t}(x,w))\Big\vert \, \mathrm{d}w \Bigg\vert^2 \, \mathrm{d}x \Bigg]^{\frac{1}{2}} \\[2mm]
&\quad + \frac{1}{(1+R(t))^{k}}\left[\int_{\R^3_+} \left( \int_{\R^3} \mathbf{1}_{\mathcal{O}^t_{\eps}}(x,[\Gamma_\eps^{t,x}]^{-1}(w)) \vert w \vert^{k}  f_\eps^0(\widetilde{\mathrm{X}}_{\eps}^{0;t}(x,w),w) \, \mathrm{d}w \right)^2 \, \mathrm{d}x \right]^{\frac{1}{2}}.
\end{align*}
For the first term, we apply Hölder inequality in velocity and write as in Lemma \ref{LM:G0}
\begin{align*}
&\int_{\R^3_+} \left( \int_{\R^3} \mathbf{1}_{\mathcal{O}^t_{\eps}}(x,[\Gamma_\eps^{t,x}]^{-1}(w))  \vert w \vert^{k} f_\eps^0(\widetilde{\mathrm{X}}_{\eps}^{0;t}(x,w),w)\Big\vert w+e_3-(P u_\eps)(0,\widetilde{\mathrm{X}}_{\eps}^{0;t}(x,w))\Big\vert \, \mathrm{d}w \right)^2 \, \mathrm{d}x \\
& \leq \Vert \vert v \vert^{k} f_\eps^0 \Vert_{\Ld^1(\R^3;\Ld^{\infty}(\R^3_+))}  \int_{\R^3_+ \times \R^3} \mathbf{1}_{\mathcal{O}^t_{\eps}}(x,[\Gamma_\eps^{t,x}]^{-1}(w)) \vert w \vert^{k} f_\eps^0(\widetilde{\mathrm{X}}_{\eps}^{0;t}(x,w),w) \Big\vert w+e_3-(P u_\eps)(0,\widetilde{\mathrm{X}}_{\eps}^{0;t}(x,w))\Big\vert^2  \, \mathrm{d}w   \, \mathrm{d}x.
\end{align*}
We then perform the change of variable in space $x \mapsto \widetilde{\mathrm{X}}_{\eps}^{0;t}(x,w)$ (see Lemma \ref{LM:ChgmtVAR-X}) and get for all $w \in \R^3$
\begin{multline*}
\int_{\R^3_+}  \left( \int_{\R^3} \mathbf{1}_{\mathcal{O}^t_{\eps}}(x,[\Gamma_\eps^{t,x}]^{-1}(w))  \vert w \vert^{k} f_\eps^0(\widetilde{\mathrm{X}}_{\eps}^{0;t}(x,w),w)\Big\vert w+e_3-(P u_\eps)(0,\widetilde{\mathrm{X}}_{\eps}^{0;t}(x,w))\Big\vert \, \mathrm{d}w \right)^2 \, \mathrm{d}x \\
 \leq \Vert \vert v \vert^{k} f_\eps^0 \Vert_{\Ld^1(\R^3;\Ld^{\infty}(\R^3_+))} \left[\Vert (1+\vert v \vert^{k+2}) f_\eps^0 \Vert_{\Ld^1(\R^3_+ \times \R^3)}+ \Vert \vert v \vert^{k} f_\eps^0 \Vert_{\Ld^1(\R^3;\Ld^{\infty}(\R^3_+))} \Vert u_\eps^0 \Vert_{\Ld^2(\R^3_+)}^2 \right].
\end{multline*}
as in Lemma \ref{LM:G0}. We have thus obtained the claimed estimate coming from the first term.

For the second term, we apply the generalized Minkowski inequality (see e.g. \cite{HLP}) and get
\begin{multline*}
\left[\int_{\R^3_+} \left( \int_{\R^3} \mathbf{1}_{\mathcal{O}^t_{\eps}}(x,[\Gamma_\eps^{t,x}]^{-1}(w)) \vert w \vert^{k}  f_\eps^0(\widetilde{\mathrm{X}}_{\eps}^{0;t}(x,w),w) \, \mathrm{d}w \right)^2 \, \mathrm{d}x \right]^{\frac{1}{2}} \\
\leq \int_{\R^3} \vert w \vert^{k}  \left( \int_{\R^3_+} \mathbf{1}_{\mathcal{O}^t_{\eps}}(x,[\Gamma_\eps^{t,x}]^{-1}(w))  f_\eps^0(\widetilde{\mathrm{X}}_{\eps}^{0;t}(x,w),w)^2 \, \mathrm{d}x \right)^{\frac{1}{2}} \, \mathrm{d}w \lesssim \int_{\R^3} \vert w \vert^{k}  \Vert f_\eps^0(\cdot, w) \Vert_{\Ld^2(\R^3_+)} \, \mathrm{d}w,
\end{multline*}
where we have performed the same procedure as above thanks to the change of variable in space $x \mapsto \widetilde{\mathrm{X}}_{\eps}^{0;t}(x,w)$. 

Adding the two previous contributions concludes the proof of the lemma, thanks to \eqref{gen:abs2}.
\end{proof}

\begin{remarque}\label{VAR-LM:F0-pointL2}There is a variant of the previous proof concerning the treatment of the first term and which leads to a slightly different conclusion.
We first write
\begin{multline*}
\int_{\R^3_+} \left( \int_{\R^3} \mathbf{1}_{\mathcal{O}^t_{\eps}}(x,[\Gamma_\eps^{t,x}]^{-1}(w))  \vert w \vert^{k} f_\eps^0(\widetilde{\mathrm{X}}_{\eps}^{0;t}(x,w),w)\Big\vert w+e_3-(P u_\eps)(0,\widetilde{\mathrm{X}}_{\eps}^{0;t}(x,w))\Big\vert \, \mathrm{d}w \right)^2 \, \mathrm{d}x \\
 \lesssim \int_{\R^3_+} \left( \int_{\R^3} \mathbf{1}_{\mathcal{O}^t_{\eps}}(x,[\Gamma_\eps^{t,x}]^{-1}(w))  \vert w \vert^{k} f_\eps^0(\widetilde{\mathrm{X}}_{\eps}^{0;t}(x,w),w)\vert w+e_3\vert \, \mathrm{d}w \right)^2 \, \mathrm{d}x \\
+\int_{\R^3_+} \left( \int_{\R^3} \mathbf{1}_{\mathcal{O}^t_{\eps}}(x,[\Gamma_\eps^{t,x}]^{-1}(w))  \vert w \vert^{k} f_\eps^0(\widetilde{\mathrm{X}}_{\eps}^{0;t}(x,w),w)\Big\vert (P u_\eps)(0,\widetilde{\mathrm{X}}_{\eps}^{0;t}(x,w))\Big\vert \, \mathrm{d}w \right)^2 \, \mathrm{d}x.
\end{multline*}
For the first of these two terms, we apply the generalized Minkowski inequality and obtain, using again the change of variable in space $x \mapsto \widetilde{\mathrm{X}}_{\eps}^{0;t}(x,w)$
\begin{align*}
&\left( \int_{\R^3_+} \left( \int_{\R^3} \mathbf{1}_{\mathcal{O}^t_{\eps}}(x,[\Gamma_\eps^{t,x}]^{-1}(w))  \vert w \vert^{k} f_\eps^0(\widetilde{\mathrm{X}}_{\eps}^{0;t}(x,w),w)\vert w+e_3\vert \, \mathrm{d}w \right)^2 \, \mathrm{d}x \right)^{\frac{1}{2}} \\
&\leq  \int_{\R^3} (1+ \vert w \vert)\vert w \vert^{k}  \left( \int_{\R^3_+} \mathbf{1}_{\mathcal{O}^t_{\eps}}(x,[\Gamma_\eps^{t,x}]^{-1}(w))  f_\eps^0(\widetilde{\mathrm{X}}_{\eps}^{0;t}(x,w),w)^2 \, \mathrm{d}x \right)^{\frac{1}{2}} \, \mathrm{d}w \\
&\lesssim \int_{\R^3} (1+ \vert w \vert)\vert w \vert^{k} \Vert f_\eps^0(\cdot, w) \Vert_{\Ld^2(\R^3_+)} \, \mathrm{d}w.
\end{align*}
For the second term, we proceed as in the original proof. We end up with the following conclusion:
\begin{align*}
\Vert F_\eps^{\natural,0}(t) \Vert_{\Ld^2(\R^3_+)}  &\lesssim \frac{e^{\frac{-t}{\eps}}}{(1+t)^{k}} \Big[ \Vert (1+\vert v \vert^{k+1})   f_\eps^0 \Vert_{\Ld^1(\R^3;\Ld^2(\R^3_+))}
+\Vert \vert v \vert^{k} f_\eps^0 \Vert_{\Ld^1(\R^3;\Ld^{\infty}(\R^3_+))}\Vert u_\eps^0 \Vert_{\Ld^2(\R^3_+)} \Big] \\
& \quad + \frac{1}{(1+t)^{k}} \Vert \vert v \vert^k   f_\eps^0 \Vert_{\Ld^1(\R^3;\Ld^2(\R^3_+))},
\end{align*}
and, in a similar way
\begin{align*}
\Vert F_\eps^{\flat,0}(t) \Vert_{\Ld^2(\R^3_+)}  &\lesssim \frac{e^{\frac{-t}{\eps}}}{(1+t)^{k}} \Big[ \Vert (1+\vert v \vert  ) \, x_3^k  \,  f_\eps^0 \Vert_{\Ld^1(\R^3;\Ld^2(\R^3_+))}
+\Vert x_3^{k} f_\eps^0 \Vert_{\Ld^1(\R^3;\Ld^{\infty}(\R^3_+))}\Vert u_\eps^0 \Vert_{\Ld^2(\R^3_+)} \Big] \\
& \quad + \frac{1}{(1+t)^{k}} \Vert x_3^k \,    f_\eps^0 \Vert_{\Ld^1(\R^3;\Ld^2(\R^3_+))}.
\end{align*}
\end{remarque}

\begin{lemme}\label{LM:F1-pointL2}
For all $t \in (T_0,t^{\star}_\eps)$ and any $k>0$, we have
\begin{align*}
\Vert  F_\eps^{\natural,1}(t)\Vert_{\Ld^2(\R^3_+)} &\lesssim  \frac{\eps^{\frac{1}{2}}}{(1+t)^k} \Vert \vert v \vert^k f_\eps^0 \Vert_{\Ld^1(\R^3; \Ld^{\infty}(\R^3_+))}  \Vert \partial_\tau u_\eps \Vert_{\Ld^2(0,t;\Ld^2(\R^3_+))}, \\
\Vert  F_\eps^{\flat,1}(t)\Vert_{\Ld^2(\R^3_+)} &\lesssim  \frac{\eps^{\frac{1}{2}}}{(1+t)^k} \Vert x_3^k f_\eps^0 \Vert_{\Ld^1(\R^3; \Ld^{\infty}(\R^3_+))}  \Vert \partial_\tau u_\eps \Vert_{\Ld^2(0,t;\Ld^2(\R^3_+))}.
\end{align*}
\end{lemme}
\begin{proof}
We focus on the treatment of the first estimate, the second one being similar. 
We first use the Hölder inequality in velocity and time in Lemma \ref{LM:G1} and get
\begin{align*}
 \int_{\R^3_+} \vert F_\eps^{\natural,1}(t,x) \vert^2 \, \mathrm{d}x 
&\leq \dfrac{1}{(1+\mathrm{R}(t))^{2k}} \int_{\R^3_+}  \left(\int_{\R^3} \int_0^t e^{\frac{\tau-t}{\eps}} \mathbf{1}_{\mathcal{O}^t_{\eps}}(x,[\Gamma_\eps^{t,x}]^{-1}(w)) \vert w \vert^k  f_\eps^0(\widetilde{\mathrm{X}}_{\eps}^{0;t}(x,w),w) \, \mathrm{d}\tau \, \mathrm{d}w \right) \\
& \quad \times \left( \int_{\R^3} \int_0^t e^{\frac{\tau-t}{\eps}} \mathbf{1}_{\mathcal{O}^t_{\eps}}(x,[\Gamma_\eps^{t,x}]^{-1}(w))  \vert w \vert^k  f_\eps^0(\widetilde{\mathrm{X}}_{\eps}^{0;t}(x,w),w)\Big\vert \partial_\tau [P u_\eps](\tau,\widetilde{\mathrm{X}}_\eps^{\tau;t}(x,w)) \Big\vert^2 \, \mathrm{d}\tau \, \mathrm{d}w \right)  \, \mathrm{d}x \\[2mm]
& \leq \dfrac{\eps \Vert \vert v \vert^k f_\eps^0 \Vert_{\Ld^1(\R^3; \Ld^{\infty}(\R^3_+))}}{(1+t)^{2k}}     \\
& \quad \times \int_{\R^3} \vert w \vert^k \int_0^t  e^{\frac{\tau-t}{\eps}} \Bigg(  \int_{\R^3_+}  \mathbf{1}_{\mathcal{O}^t_{\eps}}(x,[\Gamma_\eps^{t,x}]^{-1}(w))  f_\eps^0(\widetilde{\mathrm{X}}_{\eps}^{0;t}(x,w),w) \\
& \qquad \qquad  \qquad \qquad \qquad   \qquad \qquad \qquad \qquad \qquad \qquad \times \Big\vert \partial_\tau [P u_\eps](\tau,\widetilde{\mathrm{X}}_\eps^{\tau;t}(x,w)) \Big\vert^2  \, \mathrm{d}x \Bigg) \, \mathrm{d}\tau \, \mathrm{d}w,
\end{align*}
thanks to Fubini theorem and the same procedure as in the proof of Lemma \ref{LM:F0-pointL2}. By the change of variable $x'=\widetilde{\mathrm{X}}_\eps^{\tau;t}(x,w)$,
we obtain
\begin{align*}
 \int_{\R^3_+} \vert F_\eps^{\natural,1}(t,x) \vert^2 \, \mathrm{d}x & \lesssim \dfrac{\eps \Vert \vert w \vert^k f_\eps^0 \Vert_{\Ld^1(\R^3; \Ld^{\infty}(\R^3_+))}}{(1+\mathrm{R}(t))^{2k}}  \int_{\R^3} \int_0^t  e^{\frac{\tau-t}{\eps}} \vert w \vert^k \Vert f_\eps^0(\cdot,w) \Vert_{\Ld^{\infty}(\R^3_+)} \Vert \partial_\tau [P u_\eps](\tau) \Vert_{\Ld^2(\R^3)}^2 \, \mathrm{d}\tau \, \mathrm{d}w \\
 & \leq \dfrac{\eps \Vert \vert w \vert^k f_\eps^0 \Vert_{\Ld^1(\R^3; \Ld^{\infty}(\R^3_+))}^2}{(1+\mathrm{R}(t))^{2k}} \Vert P[\partial_t u_\eps ]\Vert_{\Ld^2((0,t) \times \R^3)}^{2},
\end{align*}
which leads to the conclusion thanks to the definition of the extension operator $P$.
\end{proof}


\begin{lemme}\label{LM:F2-pointL2}
For all $t \in (T_0,t^{\star}_\eps)$ and any $k_1, k_2 \geq 0$, we have 
\begin{align*}
\Vert  F_\eps^{\natural,2}(t) \Vert_{\Ld^2(\R^3_+)} & \lesssim \frac{1}{(1+t)^{k_1}} \left[ \eps^{\frac{3}{4}} \Vert (1+ \vert v \vert^{2}) \vert v \vert^{k_1} f_\eps^0 \Vert_{\Ld^1(\R^3; \Ld^{\infty}(\R^3_+))}^{\frac{1}{2}}   \Vert \vert v \vert^{k_1} f_\eps^0 \Vert_{\Ld^1(\R^3; \Ld^{\infty}(\R^3_+))}^{\frac{1}{2}} \mathcal{E}_\eps(0)^{\frac{1}{4}}  \Vert \mathrm{D}^2_x u_\eps\Vert_{\Ld^2((0,t) \times \R^3_+)}^{\frac{1}{2}} \right. \\
& \left. \qquad \qquad \quad + \, \eps^{\frac{3}{4}} \Vert  \vert v \vert^{k_1} f_\eps^0 \Vert_{\Ld^1(\R^3; \Ld^{\infty}(\R^3_+))}  \mathcal{E}_\eps(0)^{\frac{1}{4}}  \Vert \mathrm{D}^2_x u_\eps\Vert_{\Ld^2((0,t) \times \R^3_+)}^{\frac{1}{2}}\right] \\
& \quad + \frac{1}{(1+t)^{k_2-\frac{1}{4}}}\Vert \vert v \vert^{k_2} f_\eps^0 \Vert_{\Ld^1(\R^3;\Ld^{\infty}(\R^3_+))} \mathcal{E}_\eps(0)^{\frac{1}{2}}  \Vert u_{\eps} \Vert_{\Ld^{\infty}(0,t;\Ld^{6}(\R^3_+))}^{\frac{1}{2}}  \Vert \D^2_x u_{\eps} \Vert_{\Ld^{2}(0,t;\Ld^2(\R^3_+))}^{\frac{1}{2}},
\end{align*}
and
\begin{align*}
\Vert  F_\eps^{\flat,2}(t) \Vert_{\Ld^2(\R^3_+)} & \lesssim \frac{1}{(1+t)^{k_1}} \left[ \eps^{\frac{3}{4}} \Vert (1+ \vert v \vert^{2}) x_3^{k_1} f_\eps^0 \Vert_{\Ld^1(\R^3; \Ld^{\infty}(\R^3_+))}^{\frac{1}{2}}   \Vert x_3^{k_1} f_\eps^0 \Vert_{\Ld^1(\R^3; \Ld^{\infty}(\R^3_+))}^{\frac{1}{2}} \mathcal{E}_\eps(0)^{\frac{1}{4}}  \Vert \mathrm{D}^2_x u_\eps\Vert_{\Ld^2((0,t) \times \R^3_+)}^{\frac{1}{2}} \right. \\
& \left. \qquad \qquad \quad + \, \eps^{\frac{3}{4}} \Vert  x_3^{k_1} f_\eps^0 \Vert_{\Ld^1(\R^3; \Ld^{\infty}(\R^3_+))}  \mathcal{E}_\eps(0)^{\frac{1}{4}}  \Vert \mathrm{D}^2_x u_\eps\Vert_{\Ld^2((0,t) \times \R^3_+)}^{\frac{1}{2}}\right] \\
&\quad + \frac{1}{(1+t)^{k_2-\frac{1}{4}}}\Vert x_3^{k_2} f_\eps^0 \Vert_{\Ld^1(\R^3;\Ld^{\infty}(\R^3_+))} \mathcal{E}_\eps(0)^{\frac{1}{2}}  \Vert u_{\eps} \Vert_{\Ld^{\infty}(0,t;\Ld^{6}(\R^3_+))}^{\frac{1}{2}}  \Vert \D^2_x u_{\eps} \Vert_{\Ld^{2}(0,t;\Ld^2(\R^3_+))}^{\frac{1}{2}}.
\end{align*}
\end{lemme}
\begin{proof}
We focus on the first estimate as in the previous proofs.
Thanks to the indicator in velocity in the definition of $F_\eps^{\natural,2}$, we can write
\begin{align*}
\Vert F_\eps^{\natural,2}(t)\Vert_{\Ld^2(\R^3_+)}  \lesssim \frac{1}{(1+t)^k}\left[ \Vert \mathrm{I}^{\natural}(t) \Vert_{\Ld^2(\R^3_+)}+ \Vert \mathrm{II}^{\natural}(t)\Vert_{\Ld^2(\R^3_+)}+ \Vert\mathrm{III}^{\natural}(t)\Vert_{\Ld^2(\R^3_+)} \right],
\end{align*}
where
\begin{align*}
\mathrm{I}^{\natural}(t,x)&:=  \int_{\R^3}\mathbf{1}_{\mathcal{O}^t_{\eps}}(x,[\Gamma_\eps^{t,x}]^{-1}(w)) \vert w \vert^k \int_0^t  e^{\frac{-t}{\eps}} f_\eps^0(\widetilde{\mathrm{X}}_{\eps}^{0;t}(x,w),w) (1+ \vert w \vert) \Big\vert    \nabla_x  [P u_\eps](s,\widetilde{\mathrm{X}}_\eps^{s;t}(x,w)) \Big\vert \, \mathrm{d}s \, \mathrm{d}w, \\
 \mathrm{II}^{\natural}(t,x)&:=\int_{\R^3}\mathbf{1}_{\mathcal{O}^t_{\eps}}(x,[\Gamma_\eps^{t,x}]^{-1}(w)) \vert w \vert^k \int_0^t e^{\frac{s-t}{\eps}}  f_\eps^0(\widetilde{\mathrm{X}}_{\eps}^{0;t}(x,w),w) \Big\vert    \nabla_x  [P u_\eps](s,\widetilde{\mathrm{X}}_\eps^{s;t}(x,w)) \Big\vert \, \mathrm{d}s \,  \mathrm{d}w,
 \end{align*}
 and
 \begin{align*}
  \mathrm{III}^{\natural}(t,x)&:=\dfrac{1}{\eps} \int_{\R^3}\mathbf{1}_{\mathcal{O}^t_{\eps}}(x,[\Gamma_\eps^{t,x}]^{-1}(w)) \vert w \vert^k \int_0^t  \int_0^s e^{\frac{s-t}{\eps}} e^{\frac{\tau-s}{\eps}} f_\eps^0(\widetilde{\mathrm{X}}_{\eps}^{0;t}(x,w),w) \\
 & \qquad \qquad \qquad \qquad  \qquad \qquad \times    \left\vert (Pu_{\eps})(\tau,\widetilde{\mathrm{X}}_{\eps}^{\tau;t}(x,w)) \right\vert  \Big\vert \nabla_x[P u_\eps](s,\widetilde{\mathrm{X}}_\eps^{s;t}(x,w)) \Big\vert  \, \mathrm{d}\tau  \, \mathrm{d}s \, \mathrm{d}w.
\end{align*}
Here, we have used the formula \eqref{eq:CharacVGamma-1}.
The two first terms $\mathrm{I}^{\natural}(t)$ and $ \mathrm{II}^{\natural}(t)$ are handled in the same way in the proof of Lemma \ref{LM:G2}, taking into account the additional polynomial in velocity of degree $k$.  
By the energy inequality \eqref{ineq:energy2}, we obtain
\begin{align*}
\Vert \mathrm{I}^{\natural}(t) \Vert_{\Ld^2(\R^3_+)}  &  \lesssim \eps^{\frac{3}{4}} \Vert (1+ \vert v \vert^{2}) \vert v \vert^k f_\eps^0 \Vert_{\Ld^1(\R^3; \Ld^{\infty}(\R^3_+))}^{\frac{1}{2}}  \Vert \vert v \vert^k f_\eps^0 \Vert_{\Ld^1(\R^3; \Ld^{\infty}(\R^3_+))}^{\frac{1}{2}} \mathcal{E}_\eps(0)^{\frac{1}{4}} \Vert \mathrm{D}^2_x u_\eps\Vert_{\Ld^2((0,t) \times \R^3_+)}^{\frac{1}{2}}, \\
\Vert \mathrm{II}^{\natural}(t) \Vert_{\Ld^2(\R^3_+)}& \lesssim \eps^{\frac{3}{4}} \Vert  \vert v \vert^k f_\eps^0 \Vert_{\Ld^1(\R^3; \Ld^{\infty}(\R^3_+))}  \mathcal{E}_\eps(0)^{\frac{1}{4}}  \Vert \mathrm{D}^2_x u_\eps\Vert_{\Ld^2((0,t) \times \R^3_+)}^{\frac{1}{2}}.
\end{align*}

 For the last term $\mathrm{III}^{\natural}(t)$, we start writing 
 \begin{align*}
\vert \mathrm{III}^{\natural}(t,x) \vert^2 
& \leq \eps^{-1}\left( \int_0^t  \int_{\R^3} \mathbf{1}_{\mathcal{O}^t_{\eps}}(x,[\Gamma_\eps^{t,x}]^{-1}(w)) \vert w \vert^k f_\eps^0(\widetilde{\mathrm{X}}_{\eps}^{0;t}(x,w),w)  e^{\frac{s-t}{\eps}} \Big\vert \nabla_x[P u_\eps](s,\widetilde{\mathrm{X}}_\eps^{s;t}(x,w)) \Big\vert^{2} \, \mathrm{d}s \, \mathrm{d}w  \right) \\
& \times  \left(  \int_0^t \int_{\R^3} \mathbf{1}_{\mathcal{O}^t_{\eps}}(x,[\Gamma_\eps^{t,x}]^{-1}(w)) \vert w \vert^k f_\eps^0(\widetilde{\mathrm{X}}_{\eps}^{0;t}(x,w),w) e^{\frac{s-t}{\eps}} \left( \int_0^s  e^{\frac{\tau-s}{\eps}}   \left\vert (Pu_{\eps})(\tau,\widetilde{\mathrm{X}}_{\eps}^{\tau;t}(x,w)) \right\vert^2    \, \mathrm{d}\tau \right) \, \mathrm{d}s \, \mathrm{d}w  \right),
\end{align*}
where we have used the same manipulations than in Lemma \ref{LM:G2}. Since
\begin{align*}
(x, [\Gamma_\eps^{t,x}]^{-1}(w)) \in \mathcal{O}^t_{\eps} & \Longrightarrow \forall s \in [0,t], \, \widetilde{\mathrm{X}}_\eps(s;t,x,w) \in \R^3_+,
\end{align*}
we have
\begin{align*}
\int_{\R^3_+} \vert \mathrm{III}^{\natural}(t,x) \vert^2 \, \mathrm{d}x &\lesssim \eps^{-1} \Vert  \vert v \vert^k f_\eps^0 \Vert_{\Ld^1(\R^3;\Ld^{\infty}(\R^3_+))} \left(  \int_0^t e^{\frac{s-t}{\eps}} \left( \int_0^s  e^{\frac{\tau-s}{\eps}}   \Vert u_{\eps}(\tau) \Vert_{\Ld^{\infty}(\R^3_+)}^2   \, \mathrm{d}\tau \right) \, \mathrm{d}s   \right)  \\
& \times  \left( \int_0^t  \int_{\R^3_+ \times  \R^3} \mathbf{1}_{\mathcal{O}^t_{\eps}}(x,[\Gamma_\eps^{t,x}]^{-1}(w)) \vert w \vert^k f_\eps^0(\widetilde{\mathrm{X}}_{\eps}^{0;t}(x,w),w)  e^{\frac{s-t}{\eps}} \Big\vert \nabla_x[P u_\eps](s,\widetilde{\mathrm{X}}_\eps^{s;t}(x,w)) \Big\vert^{2}  \, \mathrm{d} x \, \mathrm{d}w \, \mathrm{d}s  \right).
\end{align*}
As above, we then use the change of variable in space $x'=\widetilde{\mathrm{X}}_\eps^{s;t}(x,w)$ given by Lemma \ref{LM:ChgmtVAR-X} and we obtain
\begin{align*}
\int_{\R^3_+} \vert \mathrm{III}^{\natural}(t,x) \vert^2 \, \mathrm{d}x & \lesssim \eps^{-1} \Vert \vert v \vert^k f_\eps^0 \Vert_{\Ld^1(\R^3;\Ld^{\infty}(\R^3_+))}^2 \left(  \int_0^t e^{\frac{s-t}{\eps}} \left( \int_0^s  e^{\frac{\tau-s}{\eps}}   \Vert u_{\eps}(\tau) \Vert_{\Ld^{\infty}(\R^3_+)}^2  \, \mathrm{d}\tau \right) \, \mathrm{d}s   \right)  \\
& \quad \times  \left( \int_0^t e^{\frac{s-t}{\eps}} \Vert \nabla_x u_\eps(s) \Vert^2_{\Ld^2(\R^3_+)} \, \mathrm{d}s  \right).
\end{align*}
Using
\begin{align*}
\Vert u_{\eps}(\tau) \Vert_{\Ld^{\infty}(\R^3_+)} \lesssim \Vert \D^2_x u_{\eps}(\tau) \Vert_{\Ld^{2}(\R^3_+)}^{\frac{1}{2}} \Vert u_{\eps}(\tau) \Vert_{\Ld^{6}(\R^3_+)}^{\frac{1}{2}},
\end{align*}
we end up with
\begin{align*}
\int_{\R^3_+} \vert \mathrm{III}^{\natural}(t,x) \vert^2 \, \mathrm{d}x & \lesssim \eps^{-1} \Vert \vert v \vert^k f_\eps^0 \Vert_{\Ld^1(\R^3;\Ld^{\infty}(\R^3_+))}^2 \left(  \int_0^t e^{\frac{s-t}{\eps}} \left( \int_0^s  \Vert \D^2_x u_{\eps}(\tau) \Vert_{\Ld^{2}(\R^3_+)} \Vert u_{\eps}(\tau) \Vert_{\Ld^{6}(\R^3_+)}   \, \mathrm{d}\tau \right) \, \mathrm{d}s   \right)  \\
& \quad \times  \left( \int_0^t e^{\frac{s-t}{\eps}} \Vert \nabla_x u_\eps(s) \Vert^2_{\Ld^2(\R^3_+)} \, \mathrm{d}s  \right) \\
& \lesssim \eps^{-1} \Vert  \vert v \vert^k f_\eps^0 \Vert_{\Ld^1(\R^3;\Ld^{\infty}(\R^3_+))}^2 \left(  \int_0^t e^{\frac{s-t}{\eps}} \Vert u_{\eps} \Vert_{\Ld^{\infty}(0,s;\Ld^{6}(\R^3_+))} \sqrt{s} \Vert \D^2_x u_{\eps} \Vert_{\Ld^{2}(0,s;\Ld^2(\R^3_+))} \, \mathrm{d}s   \right)  \\
& \quad \times  \left( \int_0^t \D_\eps(s) \, \mathrm{d}s  \right) \\
& \lesssim \eps^{-1} \Vert  \vert v \vert^k f_\eps^0 \Vert_{\Ld^1(\R^3;\Ld^{\infty}(\R^3_+))}^2 \mathcal{E}_\eps(0)  \Vert u_{\eps} \Vert_{\Ld^{\infty}(0,t;\Ld^{6}(\R^3_+))}  \Vert \D^2_x u_{\eps} \Vert_{\Ld^{2}(0,t;\Ld^2(\R^3_+))}\left(  \int_0^t e^{\frac{s-t}{\eps}} \sqrt{s} \, \mathrm{d}s   \right)  \\
&\lesssim  \Vert \vert v \vert^k f_\eps^0 \Vert_{\Ld^1(\R^3;\Ld^{\infty}(\R^3_+))}^2 \mathcal{E}_\eps(0)  \Vert u_{\eps} \Vert_{\Ld^{\infty}(0,t;\Ld^{6}(\R^3_+))}  \Vert \D^2_x u_{\eps} \Vert_{\Ld^{2}(0,t;\Ld^2(\R^3_+))} \sqrt{t},
\end{align*}
where we have used the energy inequality \eqref{ineq:energy2} (recall the Definition \eqref{def:Dissipation} of the dissipation). This entails
\begin{align*}
 \Vert \mathrm{III}^{\natural}(t) \Vert_{\Ld^2(\R^3_+)} \lesssim \Vert \vert v \vert^k f_\eps^0 \Vert_{\Ld^1(\R^3;\Ld^{\infty}(\R^3_+))} \mathcal{E}_\eps(0)^{\frac{1}{2}}  \Vert u_{\eps} \Vert_{\Ld^{\infty}(0,t;\Ld^{6}(\R^3_+))}^{\frac{1}{2}}  \Vert \D^2_x u_{\eps} \Vert_{\Ld^{2}(0,t;\Ld^2(\R^3_+))}^{\frac{1}{2}} (1+t)^{\frac{1}{4}}.
\end{align*}
We obtain the final result by combining all the terms together.
\end{proof}
\begin{corollaire}\label{coro:FpointL2}
For $\eps \in (0,1)$ and under Assumption \textbf{\ref{hypUnifBoundVNS}}, we have for any strong existence time $t \in (T_0,t^{\star}_\eps)$ 
\begin{align*}
\Vert F_\eps(t) \Vert_{\Ld^2(\R^3_+)} \leq \dfrac{M^{\varpi_2}}{(1+t)^{\frac{7}{4}}}, 
\end{align*}
for some universal constant $\varpi_2>0$, which is independent of $\eps$ and $T_0$.
\end{corollaire}
\begin{proof}
In what follows, the exponents $k$, $k_1$ and $k_2$ are generic and will be choosen in the end of the proof. In view of Lemma \ref{LM:F0-pointL2} and \ref{LM:F1-pointL2}, Assumption \textbf{\ref{hypUnifBoundVNS}} entails
\begin{align*}
\sum_{i=0}^1   \Vert F_\eps^{\natural,i}(t) \Vert_{\Ld^2(\R^3_+)}+ \sum_{i=0}^1  \Vert F_\eps^{\flat,i} (t)\Vert_{\Ld^2(\R^3_+)} &\lesssim \dfrac{M^{\nu_2}}{(1+t)^k}\left(1+ \Vert \partial_t u_\eps \Vert_{\Ld^2(0,t;\Ld^2(\R^3_+))} \right) \\
& \lesssim  \dfrac{M^{\nu_2}}{(1+t)^k},
\end{align*}
for some $\nu_2>0$, thanks to Corollary \ref{D_tu/D2:uL2}. Furthermore, using Lemma \ref{LM:F2-pointL2} and Assumption \textbf{\ref{hypUnifBoundVNS}}, there exists $\omega_2>0$ such that
\begin{align*}
\Vert F_\eps^{\natural,2}(t) \Vert_{\Ld^2(\R^3_+)}+  \Vert F_\eps^{\flat,2} (t)\Vert_{\Ld^2(\R^3_+)} & \lesssim \dfrac{M^{\omega_2}}{(1+t)^{k_1}} \left(1+ \Vert \mathrm{D}^2_x u_\eps \Vert_{\Ld^2(0,t;\Ld^2(\R^3_+))}^{\frac{1}{2}} \right) \\
& \quad + \dfrac{M^{\omega_2}}{(1+t)^{k_2-\frac{1}{4}}} \Vert u_{\eps} \Vert_{\Ld^{\infty}(0,t;\Ld^{6}(\R^3_+))}^{\frac{1}{2}}  \Vert \D^2_x u_{\eps} \Vert_{\Ld^{2}(0,t;\Ld^2(\R^3_+))}^{\frac{1}{2}} \\
& \lesssim \dfrac{M^{\omega_2}}{(1+t)^{k_1}}+ \dfrac{M^{\omega_2}}{(1+t)^{k_2-\frac{1}{4}}},
\end{align*}
by Corollary \ref{D_tu/D2:uL2}. Choosing $k=k_1=7/4$ and $k_2=2$ eventually yields the result, by taking the corresponding constant $M>1$ in Assumption \textbf{\ref{hypUnifBoundVNS}}.
\end{proof}
\begin{remarque}\label{Rmq:BrinkPointL2:noabsorb}
The estimates provided by Lemmas \ref{LM:F0-pointL2}--\ref{LM:F1-pointL2}--\ref{LM:F2-pointL2} with $k=k_1=k_2=0$ can be extended to the interval $(0,t_{\eps}^{\star})$ (considering $T_0=0$), because neither an exit geometric condition nor absorption are required. In view of Remark \ref{VAR-LM:F0-pointL2}, we can write for all $t \in (0,t_{\eps}^{\star})$
\begin{align*}
\Vert F_\eps(t) \Vert_{\Ld^2(\R^3_+)} &\lesssim   \Vert (1+\vert v \vert)   f_\eps^0 \Vert_{\Ld^1(\R^3;\Ld^2(\R^3_+))}
+\Vert f_\eps^0 \Vert_{\Ld^1(\R^3;\Ld^{\infty}(\R^3_+))}\Vert u_\eps^0 \Vert_{\Ld^2(\R^3_+)} +\Vert   f_\eps^0 \Vert_{\Ld^1(\R^3;\Ld^2(\R^3_+))} \\
&\quad +\eps^{\frac{1}{2}} \Vert  f_\eps^0 \Vert_{\Ld^1(\R^3; \Ld^{\infty}(\R^3_+))}  \Vert \partial_\tau u_\eps \Vert_{\Ld^2(0,t;\Ld^2(\R^3_+))} \\
& \quad + \left[ \eps^{\frac{3}{4}} \Vert (1+ \vert v \vert^{2}) f_\eps^0 \Vert_{\Ld^1(\R^3; \Ld^{\infty}(\R^3_+))}^{\frac{1}{2}}   \Vert  f_\eps^0 \Vert_{\Ld^1(\R^3; \Ld^{\infty}(\R^3_+))}^{\frac{1}{2}} \mathcal{E}_\eps(0)^{\frac{1}{4}}  \Vert \mathrm{D}^2_x u_\eps\Vert_{\Ld^2((0,t) \times \R^3_+)}^{\frac{1}{2}} \right. \\
& \left. \qquad \qquad \quad + \, \eps^{\frac{3}{4}} \Vert   f_\eps^0 \Vert_{\Ld^1(\R^3; \Ld^{\infty}(\R^3_+))}  \mathcal{E}_\eps(0)^{\frac{1}{4}}  \Vert \mathrm{D}^2_x u_\eps\Vert_{\Ld^2((0,t) \times \R^3_+)}^{\frac{1}{2}}\right] \\
& \quad + (1+t)^{\frac{1}{4}}\Vert  f_\eps^0 \Vert_{\Ld^1(\R^3;\Ld^{\infty}(\R^3_+))} \mathcal{E}_\eps(0)^{\frac{1}{2}}  \Vert u_{\eps} \Vert_{\Ld^{\infty}(0,t;\Ld^{6}(\R^3_+))}^{\frac{1}{2}}  \Vert \D^2_x u_{\eps} \Vert_{\Ld^{2}(0,t;\Ld^2(\R^3_+))}^{\frac{1}{2}}.
\end{align*}
\end{remarque}

\subsection{Polynomial decay estimates of the Brinkman force in $\Ld^p_t \Ld^p_x$ }\label{Subsect:F-LpLp}
We aim at deriving uniform in time estimates in $\Ld^p_t \Ld^p_x$ for some weighted in time version of the Brinkman force (namely, of the form $(1+t)^k F_\eps$). Thanks to Corollary \ref{coro:FpointL2}, we are now in position to use the polynomial decay in time of $u_\eps$ provided by Theorem \ref{cond:decay}. We state a series of lemmas based on the decomposition of Lemma \ref{Split-F2}. 

The strategy of proofs is mainly inspired by the one performed in Subsection \ref{Subsect:FpointL2}, with an additional integration in time here. For the sake of readibility, we state the result and the  proof of Lemmas \ref{LM:F0a}--\ref{LM:F1a}--\ref{LM:F2a} is postponed to Appendix \ref{AnnexeRemainsProofLemmas}.
\begin{lemme}\label{LM:F0a}
Let $T \in (T_0,t^{\star}_\eps)$ and $r \in [2,\infty)$. For any $k,k_1,k_2 \geq 0$ satisfying $k_1>k$ and $k_2>k+\frac{1}{r}$, we have
\begin{align*}
\Vert(1+t)^k F_\eps^{\natural,0}\Vert_{\Ld^r((T_0 ,T) \times \R^3_+)} &\lesssim  \eps^{\frac{1}{r}}\Vert \vert v \vert^{k_1} f_\eps^0 \Vert_{\Ld^1(\R^3;\Ld^{\infty}(\R^3_+))}^{\frac{r-1}{r}}\Big[ \Vert (1+\vert v \vert^{k_1+r}) f_\eps^0 \Vert_{\Ld^1(\R^3_+ \times \R^3)}^{\frac{1}{r}}\\
& \qquad \qquad  \qquad \qquad   \qquad \qquad + \Vert \vert v \vert^{k_1} f_\eps^0 \Vert_{\Ld^1(\R^3;\Ld^{\infty}(\R^3_+))}^{\frac{1}{r}} \Vert u_\eps^0 \Vert_{\Ld^r(\R^3_+)}\Big]  \\[2mm]
&  \quad + \Vert \vert v \vert^{k_2}  \Vert f_\eps^0 \Vert_{\Ld^1(\R^3;\Ld^r(\R^3_+))}^{\frac{1}{r}},
\end{align*}
and
\begin{align*}
\Vert (1+t)^k  F_\eps^{\flat,0}\Vert_{\Ld^r((T_0,T) \times \R^3_+)} &\lesssim \eps^{\frac{1}{r}} \Vert  x_3^{k_1} \,  f_\eps^0 \Vert_{\Ld^1(\R^3;\Ld^{\infty}(\R^3_+))}^{\frac{r-1}{r}} \Big[\Vert(1+ \vert v \vert^{r})  x_3^{k_1}\,  f_\eps^0 \Vert_{\Ld^1(\R^3_+ \times \R^3)}^{\frac{1}{r}} \\
& \qquad \qquad  \qquad \qquad   \qquad \qquad + \Vert   x_3^{k_1} \,  f_\eps^0 \Vert_{\Ld^1(\R^3;\Ld^{\infty}(\R^3_+))}^{\frac{1}{r}} \Vert u_\eps^0 \Vert_{\Ld^r(\R^3_+)}\Big]\\
& \quad + \Vert x_3^{k_2} \, f_\eps^0 \Vert_{\Ld^1(\R^3;\Ld^r(\R^3_+))}^{\frac{1}{r}}.
\end{align*}
Furthermore, if $k=k_1=k_2=0$, we can replace $T_0$ by $0$ in the previous result.
\end{lemme}

\begin{lemme}\label{LM:F1a}
Let $T \in (T_0,t^{\star}_\eps)$ and $r \in [2,\infty)$. For any $k \geq0$, we have 
\begin{align*}
  \Vert (1+t)^k F_\eps^{\natural,1}\Vert_{\Ld^r(T_0,T;\Ld^r(\R^3_+))} &\lesssim \eps \Vert \vert v \vert^k f_\eps^0 \Vert_{\Ld^1(\R^3; \Ld^{\infty}(\R^3_+))} \Vert \partial_t u_\eps \Vert_{\Ld^r(0,T;\Ld^r(\R^3_+))}, \\[2mm]
\Vert (1+t)^k  F_\eps^{\flat,1}\Vert_{\Ld^r(T_0,T;\Ld^r(\R^3_+))} &\lesssim \eps \Vert x_3^k \, f_\eps^0 \Vert_{\Ld^1(\R^3; \Ld^{\infty}(\R^3_+))} \Vert \partial_t  u_\eps \Vert_{\Ld^r(0,T;\Ld^r(\R^3_+))}.
\end{align*}
Furthermore, if $k=0$, we can replace $T_0$ by $0$ in the previous result.
\end{lemme}
Recall that the estimate obtained in Corollary \ref{coro:FpointL2} holds true on $(T_0,t_\eps^{\star})$. Until the end of the section, we will refer to a generic nonnegative continuous and nondecreasing function $\Psi$, which is independent of $\eps$ (and related to Theorem \ref{cond:decay}).
\begin{lemme}\label{LM:F2a}
Let $T \in (T_0,t^{\star}_\eps)$ and $r \in (3,+\infty)$.
If we assume that
\begin{align}\label{assumption:DECAYSOURCE}
\forall s \in [0,T], \ \  \Vert F_\eps (s)\Vert_{\Ld^2(\R^3_+)} \lesssim \dfrac{1}{(1+s)^{7/4}},
\end{align}
then for any $k\geq0$, 
\begin{align*}
\Vert (1+t)^k F_\eps^{\natural,2} \Vert_{\Ld^r((T_0,T) \times \R^3_+)} &  \lesssim \eps \Vert (1+ \vert v \vert^{\frac{r}{r-1}}) \vert v \vert^k  f_\eps^0 \Vert_{\Ld^1(\R^3; \Ld^{\infty}(\R^3_+))}^{\frac{r-1}{(r(1-\alpha_r)}}  \Vert \vert v \vert^k  f_\eps^0 \Vert_{\Ld^1(\R^3; \Ld^{\infty}(\R^3_+))}^{\frac{1}{r(1-\alpha_r)}}\\
&\qquad \qquad \qquad \qquad \qquad \qquad \qquad \times  \Psi \left (\Vert  u_{\eps}^0 \Vert_{\Ld^{1}\cap \Ld^2(\R^3_+)}^2+M \right)^{\frac{1}{2}}  \\
& \quad + \eps \Vert \vert v \vert^k  f_\eps^0 \Vert_{\Ld^1(\R^3; \Ld^{\infty}(\R^3_+))}^{\frac{1}{1-\alpha_r}}  \Psi \left (\Vert  u_{\eps}^0 \Vert_{\Ld^{1}\cap \Ld^2(\R^3_+)}^2+M \right)^{\frac{1}{2}}  \\
& \quad + \eps (\Upsilon_\eps^0)^{\frac{1}{1-\beta_r}} \Vert \vert v \vert^k  f_\eps^0\Vert^{\frac{1}{1-\beta_r}}_{\Ld^1(\R^3; \Ld^{\infty}(\R^3_+))} \Psi \left (\Vert  u_{\eps}^0 \Vert_{\Ld^{1}\cap \Ld^2(\R^3_+)}^2+M \right)^{\frac{1}{2}}  \\
& \quad +\eps \Vert \mathrm{D}^2_x u_\eps \Vert_{\Ld^r((0,T) \times \R^3_+)}.
\end{align*}
for some $\alpha_r, \beta_r \in (0,1)$. Furthermore, if $k=0$, we can replace $T_0$ by $0$ in the previous result.
\end{lemme}

Dealing with the term $F_\eps^{\flat,2}$ can be achieved in the same way. For the sake of conciseness, we do not detail the proof and directly state the result.
\begin{lemme}\label{LM:F2b}
Let $T \in (T_0,t^{\star}_\eps)$ and $r \in (3,+\infty)$.
If we assume that
\begin{align*}
\forall s \in [0,T], \ \  \Vert F_\eps (s)\Vert_{\Ld^2(\R^3_+)} \lesssim \dfrac{1}{(1+s)^{7/4}},
\end{align*}
then for any $k\geq0$, 
\begin{align*}
\Vert (1+t)^k  F_\eps^{\flat,2}\Vert_{\Ld^r((T_0,T) \times \R^3_+)} &  \lesssim \eps \Vert (1+ \vert v \vert^{\frac{r}{r-1}}) x_3^k \,  f_\eps^0 \Vert_{\Ld^1(\R^3; \Ld^{\infty}(\R^3_+))}^{\frac{r-1}{(r(1-\alpha_r)}}  \Vert x_3^k \, f_\eps^0 \Vert_{\Ld^1(\R^3; \Ld^{\infty}(\R^3_+))}^{\frac{1}{r(1-\alpha_r)}} \\
&\qquad \qquad \qquad \qquad \qquad \qquad \qquad \times  \Psi \left (\Vert  u_{\eps}^0 \Vert_{\Ld^{1}\cap\Ld^2(\R^3_+)}^2+M \right)^{\frac{1}{2}}  \\
& \quad + \eps \Vert x_3^k \,  f_\eps^0 \Vert_{\Ld^1(\R^3; \Ld^{\infty}(\R^3_+))}^{\frac{1}{1-\alpha_r}}  \Psi \left (\Vert  u_{\eps}^0 \Vert_{\Ld^{1}\cap\Ld^2(\R^3_+)}^2+M\right)^{\frac{1}{2}}  \\
& \quad + \eps (\Upsilon_\eps^0)^{\frac{1}{1-\beta_r}} \Vert x_3^k \,  f_\eps^0\Vert^{\frac{1}{1-\beta_r}}_{\Ld^1(\R^3; \Ld^{\infty}(\R^3_+))} \Psi \left (\Vert  u_{\eps}^0 \Vert_{\Ld^{1}\cap\Ld^2(\R^3_+)}^2+M \right)^{\frac{1}{2}}  \\
& \quad +\eps \Vert \mathrm{D}^2_x u_\eps \Vert_{\Ld^r((0,T) \times \R^3_+)}.
\end{align*}
for some $\alpha_r, \beta_r \in (0,1)$. Furthermore, if $k=0$, we can replace $T_0$ by $0$ in the previous result.
\end{lemme}

\bigskip

We now conclude this section by collecting all the previous estimates. Note that we obviously have
\begin{align*}
\Vert \partial_t  u_\eps \Vert_{\Ld^r(0,T;\Ld^r(\R^3_+))} +\Vert \mathrm{D}^2_x u_\eps \Vert_{\Ld^r((0,T) \times \R^3_+)} \leq \Vert (1+t)^k\partial_t  u_\eps \Vert_{\Ld^r(0,T;\Ld^r(\R^3_+))} +\Vert (1+t)^k\mathrm{D}^2_x u_\eps \Vert_{\Ld^r((0,T) \times \R^3_+)}.
\end{align*}
By writing
\begin{align*}
\Vert (1+t)^k F_\eps \Vert_{\Ld^r((T_0,T) \times \R^3_+)} &\leq \sum_{i=0}^2 \Vert (1+t)^k  F_\eps^{\natural,i}\Vert_{\Ld^r((T_0,T) \times \R^3_+)} + \sum_{i=0}^2 \Vert (1+t)^k  F_\eps^{\flat,i}\Vert_{\Ld^r((T_0,T) \times \R^3_+)}, 
\end{align*}
we can infer the following result.
\begin{corollaire}\label{Coro:BrinkLr-w}
Let $T \in (T_0,t^{\star}_\eps)$ and $r \in (3,+\infty)$. Let $k \geq 0$ be fixed. There exists $\ell_{k,r}=\ell>0$ such that if we assume
\begin{align*}
\Vert  (1+\vert v \vert^{\ell})(1+x_3^{\ell}) \, f_\eps^0 \Vert_{\Ld^1(\R^3; \Ld^{\infty}\cap \Ld^1(\R^3_+))} &\leq M,\\[2mm]
\Vert u_\eps^0 \Vert_{\Ld^r(\R^3_+)} &\leq M,
\end{align*}
for some $M>1$, then the following holds.
Under the assumption 
\begin{align}\label{assumption:DECAYSOURCEbis}
\forall s \in [0,T], \ \  \Vert F_\eps (s)\Vert_{\Ld^2(\R^3_+)} \lesssim \dfrac{1}{(1+s)^{7/4}},
\end{align}
we have
\begin{align*}
\Vert (1+t)^k F_\eps \Vert_{\Ld^r((T_0,T) \times \R^3_+)} &\lesssim  \eps^{\frac{1}{r}} M^{\omega_r} + M  +\eps \Big[1+(\Upsilon_\eps^0)^{\mu_r} \Big]M^{\omega_r} \Psi \left (\Vert  u_{\eps}^0 \Vert_{\Ld^{1}\cap\Ld^2(\R^3_+)}^2+M \right)^{\frac{1}{2}} \\
& \quad +\eps M \Vert (1+t)^k\partial_t  u_\eps \Vert_{\Ld^r((0,T) \times \R^3_+)}  \\
& \quad +\eps \Vert (1+t)^k\mathrm{D}^2_x u_\eps \Vert_{\Ld^r((0,T) \times \R^3_+)},
\end{align*}
for some $\omega_r, \mu_r>0$. Furthermore, if $k=0$, we can replace $T_0$ by $0$ in the previous result.
\end{corollaire}

\bigskip

As a consequence, we are now in position to obtain higher order integrability results for $u_\eps$, which are uniform in $\eps$. Indeed, by the maximal $\Ld^p_t \Ld^p_x$ parabolic regularity theory for the Stokes system with $p \in (1,\infty)$ (see in the Appendix \ref{AnnexeMaxregStokes}), we know that
\begin{multline*}
\Vert \partial_t u_\eps \Vert_{\Ld^p(0,T;\Ld^p(\R^3_+))} + \Vert \mathrm{D}^2_x u_\eps \Vert_{\Ld^p(0,T;\Ld^p(\R^3_+))} \\  \lesssim \Vert F_\eps \Vert_{\Ld^p(0,T;\Ld^p(\R^3_+))} + \Vert (u_\eps \cdot \nabla_x) u_\eps \Vert_{\Ld^p(0,T;\Ld^p(\R^3_+))} + \Vert u_\eps^0 \Vert_{\mathrm{D}_p^{1-\frac{1}{p},p}(\R^3_+))}.
\end{multline*}
Hence, assuming that the pointwise decay of the Brinkman force \eqref{assumption:DECAYSOURCEbis} is satisfied on $(0,T)$ with $T \in (0,t^{\star}_\eps)$, we can infer from Corollary \ref{Coro:BrinkLr-w} with $k=0$ and Corollary \ref{coro:estimatesSTRONG} that for $p$ close enough to $3$, we have
\begin{align*}
\Vert \partial_t u_\eps \Vert_{\Ld^p(0,T;\Ld^p(\R^3_+))} + \Vert \mathrm{D}^2_x u_\eps \Vert_{\Ld^p(0,T;\Ld^p(\R^3_+))}  &\lesssim  \eps^{\frac{1}{p}} M^{\omega_p} + M  +\eps \Big[1+(\Upsilon_\eps^0)^{\mu_p} \Big]M^{\omega_p} \Psi \left (\Vert  u_{\eps}^0 \Vert_{\Ld^{1}\cap\Ld^2(\R^3_+)}^2+M \right)^{\frac{1}{2}} \\
&\quad +\eps M \Vert \partial_t  u_\eps \Vert_{\Ld^p((0,T) \times \R^3_+)}  \\
& \quad+\eps \Vert \mathrm{D}^2_x u_\eps \Vert_{\Ld^p((0,T) \times \R^3_+)} \\
& \quad + (\Upsilon_\eps^0)^{\widetilde{\mu}_p} \mathcal{E}_\eps(0)^{\mu_p} \Vert \mathrm{D}^2_x u_\eps \Vert_{\Ld^p((0,T) \times \R^3_+)} \\
& \quad + \Vert u_\eps^0 \Vert_{\mathrm{D}_p^{1-\frac{1}{p},p}(\R^3_+))}.
\end{align*}
Taking $\eps$ and $\mathcal{E}_\eps(0)$ small enough according to Assumption \ref{hypSmallDataTech} and the uniform bound from Assumption \textbf{\ref{hypUnifBoundVNS}}, we can infer the following result.
\begin{corollaire}\label{BrinkDt2uLp:unif}
There exists $\eps_0>0$ and $p_0>3$ such that for all $\eps \in (0, \eps_0)$ and $p \in (3,p_0)$, the following holds. Let $T \in (0,t^{\star}_\eps)$ and assume that
\begin{align}\label{assumption:DECAYSOURCEbisbis}
\forall s \in [0,T], \ \  \Vert F_\eps (s)\Vert_{\Ld^2(\R^3_+)} \lesssim \dfrac{1}{(1+s)^{7/4}}.
\end{align}
Then 
\begin{align*}
\Vert \partial_t u_\eps \Vert_{\Ld^p(0,T;\Ld^p(\R^3_+))} + \Vert \mathrm{D}^2_x u_\eps \Vert_{\Ld^p(0,T;\Ld^p(\R^3_+))} \lesssim M^{\omega_p},
\end{align*}
for some $\omega_p>0$.
\end{corollaire}

\begin{remarque}\label{Coro:BrinkLr-wLOC}
Under the same assumptions than Corollary \ref{BrinkDt2uLp:unif}, the very same kind of manipulations lead to the following local and weighted control in time. There exists $\eps_0>0$ and $p_0>3$ such that for all $\eps \in (0, \eps_0)$ and $p \in (3,p_0)$, we have for all $T \in (0,t^{\star}_\eps)$
\begin{align*}
\Vert (1+t)^k F_\eps \Vert_{\Ld^p((0,T) \times \R^3_+)} &\lesssim (1+T)^k \Big( \eps^{\frac{1}{p}} M^{\omega_p} +\eps M^{\omega_p}+ M  +\eps \Big[1+(\Upsilon_\eps^0)^{\mu_p} \Big]M^{\omega_p} \Psi \left (\Vert  u_{\eps}^0 \Vert_{\Ld^{1}\cap\Ld^2(\R^3_+)}^2+M \right)^{\frac{1}{2}}  \Big).
\end{align*}
\end{remarque}

\section{Bootstrap and convergence towards the Boussinesq-Navier-Stokes system}\label{Section:Boostrap}
The main goal of this section is to complete the bootstrap argument which is needed in our study of the convergence to the Boussinesq-Navier-Stokes sytem. To do so, we will use the precise decay estimates obtained in Section \ref{Section:estimateBrinkman}, proving that such controls can be propagated until any time along the evolution. Recall that this mainly amounts to prove that for $\eps$ small enough, we have $t_\eps^{\star}=+\infty$. Bearing on the uniform estimates on $u_\eps$ and $\rho_\eps$ derived from this strategy, we will be able to prove Theorems \ref{thCV}--\ref{th-rate:cvgence}.  
\subsection{Initial horizon for the bootstrap procedure}
We first set up the bootstrap procedure, as explained in Subsection \ref{subsection:Strat}. As the polynomial estimates of Subsection \ref{Subsect:F-LpLp} are all based on the absorption effect at the boundary (through the propagation of the exit geometric condition), one must ensure that we are indeed allowed to use such an effect (see the discussion made just before Subsection \ref{Subsect:FpointL2}).

We rely on the family of estimates of the previous section which hold without assuming any absorption at the boundary. The smallness condition contained in Assumption \textbf{\ref{hypSmallData}} allows to prove that, for $\eps$ small enough, the time $t_\eps^{\star}$ is bounded from below by some uniform time after which one can use the exit geometric condition. We first consider the following definition.
\begin{definition}
We set 
\begin{align*}
T_{\mathrm{abs}}&:=t_{\frac{1}{2}}^g(1,1),
\end{align*}
where $t_{\frac{1}{2}}^g$ is defined in \eqref{def:t_0(R)} in Subsection \ref{Subsection:EGCabs}.
\end{definition}
%
%
Roughly speaking, the following lemma is the counterpart of Corollary \ref{coro:FpointL2}.
\begin{lemme}\label{LM:pointwiseMin:initial}
There exists $\eps_0>0$ such that for all $\eps \in (0, \eps_0)$, we have
\begin{align*}
\forall t \in (0,\min (T_{\mathrm{abs}}+10\alpha,t_\eps^{\star})), \ \ \Vert F_\eps(t) \Vert_{\Ld^2(\R^3_+)} \leq  \frac{M^{\varpi_2}}{(1+T_{\mathrm{abs}}+10\alpha)^{\frac{7}{4}}}.
\end{align*}
where $\varpi_2>0$ is the universal constant involved in Corollary \ref{coro:FpointL2} and where $\alpha$ has been fixed in Definition \eqref{def:tstar}.
\end{lemme}
\begin{proof}
According to Remark \ref{Rmq:BrinkPointL2:noabsorb},  there exists $\omega>0$ such that for all $t \in (0,\min (T_{\mathrm{abs}}+10\alpha,t_\eps^{\star}))$
\begin{align*}
 \Vert F_\eps(t) \Vert_{\Ld^2(\R^3_+)} &\lesssim   \Vert (1+\vert v \vert)   f_\eps^0 \Vert_{\Ld^1(\R^3;\Ld^2(\R^3_+))}
+\Vert f_\eps^0 \Vert_{\Ld^1(\R^3;\Ld^{\infty}(\R^3_+))}\mathcal{E}_\eps(0)^{\frac{1}{2}} +\Vert   f_\eps^0 \Vert_{\Ld^1(\R^3;\Ld^2(\R^3_+))} \\
& \quad +\eps^{\frac{1}{2}} \Vert  f_\eps^0 \Vert_{\Ld^1(\R^3; \Ld^{\infty}(\R^3_+))}  \Vert \partial_\tau u_\eps \Vert_{\Ld^2(0,t;\Ld^2(\R^3_+))} \\
& \quad + \left[ \eps^{\frac{3}{4}} \Vert (1+ \vert v \vert^{2}) f_\eps^0 \Vert_{\Ld^1(\R^3; \Ld^{\infty}(\R^3_+))}^{\frac{1}{2}}   \Vert  f_\eps^0 \Vert_{\Ld^1(\R^3; \Ld^{\infty}(\R^3_+))}^{\frac{1}{2}} \mathcal{E}_\eps(0)^{\frac{1}{4}}  \Vert \mathrm{D}^2_x u_\eps\Vert_{\Ld^2((0,t) \times \R^3_+)}^{\frac{1}{2}} \right. \\
& \left. \qquad \qquad \quad + \, \eps^{\frac{3}{4}} \Vert   f_\eps^0 \Vert_{\Ld^1(\R^3; \Ld^{\infty}(\R^3_+))}  \mathcal{E}_\eps(0)^{\frac{1}{4}}  \Vert \mathrm{D}^2_x u_\eps\Vert_{\Ld^2((0,t) \times \R^3_+)}^{\frac{1}{2}}\right] \\
& \quad + (1+t)^{\frac{1}{4}}\Vert  f_\eps^0 \Vert_{\Ld^1(\R^3;\Ld^{\infty}(\R^3_+))} \mathcal{E}_\eps(0)^{\frac{1}{2}}  \Vert u_{\eps} \Vert_{\Ld^{\infty}(0,t;\Ld^{6}(\R^3_+))}^{\frac{1}{2}}  \Vert \D^2_x u_{\eps} \Vert_{\Ld^{2}(0,t;\Ld^2(\R^3_+))}^{\frac{1}{2}}.
\end{align*}
Since $t<t_{\eps}^{\star}$, we can use the unifom bound \eqref{ineq:VNSreg} given by Proposition \ref{propdatasmall:VNSreg} and by Assumption \textbf{\ref{hypUnifBoundVNS}} to get
\begin{align*}
\Vert F_\eps(t) \Vert_{\Ld^2(\R^3_+)} &\lesssim \Vert (1+\vert v \vert)   f_\eps^0 \Vert_{\Ld^1(\R^3;\Ld^2(\R^3_+))} +M \mathcal{E}_\eps(0)^{\frac{1}{2}} +(\eps^{\frac{3}{4}} +\eps^{\frac{1}{2}}) M^{\omega} + (1+t)^{\frac{1}{4}}  \mathcal{E}_\eps(0)^{\frac{1}{2}}  M^{\omega} \\
&\lesssim \Vert (1+\vert v \vert)   f_\eps^0 \Vert_{\Ld^1(\R^3;\Ld^2(\R^3_+))}+M \mathcal{E}_\eps(0)^{\frac{1}{2}} +(\eps^{\frac{3}{4}} +\eps^{\frac{1}{2}}) M^{\omega} + (1+T_{\mathrm{abs}}+10\alpha)^{\frac{1}{4}}  \mathcal{E}_\eps(0)^{\frac{1}{2}}  M^{\omega}.
\end{align*}
Since $T_{\mathrm{abs}}+10\alpha$ is fixed and does not depend on $\eps$, owing to Assumption \ref{hypSmallDataTech}, we can first reduce $\eps>0$ and then $\Vert (1+\vert v \vert)   f_\eps^0 \Vert_{\Ld^1(\R^3;\Ld^2(\R^3_+))}$ and $\mathcal{E}_\eps(0)^{\frac{1}{2}}$ so that 
\begin{align*}
\Vert (1+\vert v \vert)   f_\eps^0 \Vert_{\Ld^1(\R^3;\Ld^2(\R^3_+))} +M \mathcal{E}_\eps(0)^{\frac{1}{2}} +(\eps^{\frac{3}{4}} +\eps^{\frac{1}{2}}) M^{\omega} + (1+T_{\mathrm{abs}}+10\alpha)^{\frac{1}{4}}  \mathcal{E}_\eps(0)^{\frac{1}{2}}  M^{\omega}<\frac{M^{\varpi_2}}{(1+T_{\mathrm{abs}}+10\alpha)^{\frac{7}{4}}}.
\end{align*}
This concludes the proof of the lemma.
\end{proof}
Thanks to the pointwise decay of the Brinkman force provided by Lemma \ref{LM:pointwiseMin:initial}, we are allowed to use the conditional decay of the kinetic energy stated in Theorem \ref{cond:decay}, but only on the interval $(0,\min (T_{\mathrm{abs}}+10\alpha,t_\eps^{\star}))$ for the moment. In addition, this means that this can be used on this interval in the estimates of Subsection \ref{Subsect:F-LpLp}, with the exponent $k=0$.

\begin{lemme}\label{BrinkL2:unif}
There exists $\eps_0$ such that for all $\eps \in (0, \eps_0)$, and $T \in (0,\min (T_{\mathrm{abs}}+10\alpha,t_\eps^{\star}))$, we have
\begin{align*}
 \Vert F_\eps \Vert_{\Ld^2(0,T;\Ld^2(\R^3_+))}& \lesssim \eps^{\frac{1}{2}} M^{\mu_2}+\eps M \Vert \partial_t u_\eps \Vert_{\Ld^2((0,T) \times \R^3_+)} + \eps  \Vert \D^2_x u_\eps \Vert_{\Ld^2((0,T) \times \R^3_+)}\\
  & \quad +\eps \Vert \mathrm{D}^2_x u_\eps  \Vert_{\Ld^p(0,T;\Ld^p(\R^3_+))}^{2 \beta_p} (1+T)^{\zeta_p} + \Vert  f_\eps^0 \Vert_{\Ld^1(\R^3;\Ld^2(\R^3_+))}^{\frac{1}{2}},
\end{align*}
for any $p>3$ and some $\beta_p \in (0,1)$ and $\zeta_p>0$.
\end{lemme}
\begin{proof}
Note that we do not use the absorption at the boundary in the proof. 
We use the splitting of Lemma \ref{Split-F1} by writing
\begin{align*}
F_\eps= F_\eps^0+ F_\eps^1+ F_\eps^2,
\end{align*}
and we observe that the estimates given by Lemma \ref{LM:F0a} and Lemma \ref{LM:F1a} for $k=0$ are unchanged and applied for $r=2$, that is
\begin{align*}
\Vert F_\eps^0 \Vert_{\Ld^2((0,T) \times \R^3_+)} &\lesssim \eps^{\frac{1}{2}} M^{\mu_2} + \Vert  f_\eps^0 \Vert_{\Ld^1(\R^3;\Ld^2(\R^3_+))}^{\frac{1}{2}}, \\
\Vert F_\eps^1 \Vert_{\Ld^2((0,T) \times \R^3_+)} &\lesssim \eps M \Vert \partial_t u_\eps \Vert_{\Ld^2((0,T) \times \R^3_+)}.
\end{align*}
Compared to the proof of Corollary \ref{LM:F2a} for $p>3$, we are not allowed to use the Gagliardo-Nirenberg-Sobolev inequality to treat $\Vert \nabla_x u_\eps \Vert_{\Ld^{\infty}(\R^3_+)}$ involved in $F_\eps^2$ in the same way. We still have
\begin{align*}
\Vert F_\eps^2 \Vert_{\Ld^2((0,T) \times \R^3_+)} \lesssim \eps M^{\widetilde{\mu}_2} +  \eps\Vert \mathrm{D}^2_x u_\eps \Vert_{\Ld^2((0,T) \times \R^3_+)} + \Vert  \mathrm{III} \Vert_{\Ld^2((0,T) \times \R^3_+)},
\end{align*}
where
\begin{align*}
\mathrm{III}(t,x)&=\dfrac{1}{\eps} \int_{\R^3}\mathbf{1}_{\mathcal{O}^t_{\eps}}(x,[\Gamma_\eps^{t,x}]^{-1}(w)) \int_0^t  \int_0^s e^{\frac{s-t}{\eps}} e^{\frac{\tau-s}{\eps}} f_\eps^0(\widetilde{\mathrm{X}}_{\eps}^{0;t}(x,w),w) \\
& \qquad \qquad \qquad \qquad \qquad \qquad \qquad \qquad \times   \left\vert (Pu_{\eps})(\tau,\widetilde{\mathrm{X}}_{\eps}^{\tau;t}(x,w)) \right\vert  \Big\vert \nabla_x[P u_\eps](s,\widetilde{\mathrm{X}}_\eps^{s;t}(x,w)) \Big\vert  \, \mathrm{d}\tau  \, \mathrm{d}s \, \mathrm{d}w.
\end{align*}
Similar computations as those of the proof of Corollary \ref{LM:F2a} give
\begin{align*}
\int_0^T \int_{\R^3_+} \vert \mathrm{III}(t,x) \vert^2 \, \mathrm{d}x \, \mathrm{d}t &\lesssim \eps^{-1} \Vert f_\eps^0\Vert_{\Ld^1(\R^3; \Ld^{\infty}(\R^3_+))} \left[ \int_0^T \left( \int_0^t  e^{\frac{s-t}{\eps}}  \Vert \nabla_x u_\eps(s) \Vert_{\Ld^{\infty}(\R^3_+)}^{2} \, \mathrm{d}s \right) \, \mathrm{d}t \right] \\
& \quad \times \underset{t \in (0,T)}{\sup} \Bigg\{ \int_0^te^{\frac{s-t}{\eps}}  \int_0^s  e^{\frac{\tau-s}{\eps}}  \\
& \qquad \times \int_{\R^3_+ \times \R^3} \mathbf{1}_{\mathcal{O}^t_{\eps}}(x,[\Gamma_\eps^{t,x}]^{-1}(w)) f_\eps^0(\widetilde{\mathrm{X}}_{\eps}^{0;t}(x,w),w)      \left\vert (Pu_{\eps})(\tau,\widetilde{\mathrm{X}}_{\eps}^{\tau;t}(x,w)) \right\vert^2 \, \mathrm{d}x \, \mathrm{d}w  \, \mathrm{d}\tau  \, \mathrm{d}s  \Bigg\} .
\end{align*}
Again, for the term between braces we have
\begin{align*}
& \int_0^te^{\frac{s-t}{\eps}}  \int_0^s  e^{\frac{\tau-s}{\eps}}  \int_{\R^3_+ \times \R^3} \mathbf{1}_{\mathcal{O}^t_{\eps}}(x,[\Gamma_\eps^{t,x}]^{-1}(w)) f_\eps^0(\widetilde{\mathrm{X}}_{\eps}^{0;t}(x,w),w)      \left\vert (Pu_{\eps})(\tau,\widetilde{\mathrm{X}}_{\eps}^{\tau;t}(x,w)) \right\vert^2 \, \mathrm{d}x \, \mathrm{d}w  \, \mathrm{d}\tau \, \mathrm{d}s \\
 & \lesssim \eps^2 \Vert f_\eps^0 \Vert_{\Ld^1(\R^3 ; \Ld^{\infty}(\R^3_+)) } (\Upsilon_\eps^0)^2 \\
 & \lesssim \eps^2 M.
\end{align*}
The term between brackets is actually the only one requiring a slightly different treatment: using the Gagliardo-Nirenberg-Sobolev inequality (see Theorem \ref{gagliardo-nirenberg} in the Appendix), we can write that for $p>3$
\begin{align*}
\Vert \nabla_x u_\eps(s) \Vert_{\Ld^{\infty}(\R^3_+)} \lesssim \Vert \mathrm{D}^2_x u_\eps (s) \Vert^{\beta_p}_{\Ld^p(\R^3_+)} \Vert u_\eps (s) \Vert^{1-\beta_p}_{\Ld^2(\R^3_+)}, \ \  \beta_p=\frac{5p}{10p-6},
\end{align*}
therefore, applying Hölder inequality as well as the combination of Lemma \ref{LM:pointwiseMin:initial} and Theorem \ref{cond:decay}, we get
\begin{align*}
\int_0^T \left( \int_0^t  e^{\frac{s-t}{\eps}}  \Vert \nabla_x u_\eps(s) \Vert_{\Ld^{\infty}(\R^3_+)}^{2} \, \mathrm{d}s \right) \, \mathrm{d}t &\lesssim \int_0^T \left( \int_0^t  e^{\frac{s-t}{\eps}}  \Vert \mathrm{D}^2_x u_\eps (s) \Vert^{2\beta_p}_{\Ld^p(\R^3_+)} \Vert u_\eps (s) \Vert^{2(1-\beta_p)}_{\Ld^2(\R^3_+)} \, \mathrm{d}s \right) \, \mathrm{d}t \\
& \lesssim \int_0^T \left( \int_0^t  e^{\frac{s-t}{\eps}}  \Vert \mathrm{D}^2_x u_\eps (s) \Vert^{2\beta_p}_{\Ld^p(\R^3_+)} \dfrac{1}{(1+s)^{2(1-\beta_p) \frac{3}{4}}} \, \mathrm{d}s \right) \, \mathrm{d}t \\
& = \int_0^T \left( \int_s^T  e^{\frac{s-t}{\eps}}  \Vert \mathrm{D}^2_x u_\eps (s) \Vert^{2\beta_p}_{\Ld^p(\R^3_+)} \dfrac{1}{(1+s)^{2(1-\beta_p) \frac{3}{4}}} \, \mathrm{d}t \right) \, \mathrm{d}s \\ 
& \leq \eps \Vert \mathrm{D}^2_x u_\eps  \Vert_{\Ld^p(0,T;\Ld^p(\R^3_+))}^{2 \beta_p}   \left(  \int_0^T \dfrac{1}{(1+s)^{\frac{p(1-\beta_p)}{p-2 \beta_p} \frac{6}{4}}} \right)^{1-\frac{2 \beta_p}{p}}.
\end{align*}
Note that if $p \geq  3$ then
\begin{align*}
\dfrac{p}{2\beta_p}=\frac{7p-6}{10} \geq 1, \ \ \frac{p(1-\beta_p)}{p-2 \beta_p} \frac{6}{4}=\frac{2p^2-6p}{7p^2-16p} \frac{6}{4} \in \left(0, \frac{1}{2}\right).
\end{align*}
We do not get a uniform bound in time for the last integral: we can just write
\begin{align*}
\int_0^T \left( \int_0^t  e^{\frac{s-t}{\eps}}  \Vert \nabla_x u_\eps(s) \Vert_{\Ld^{\infty}(\R^3_+)}^{2} \, \mathrm{d}s \right) \, \mathrm{d}t &\lesssim  \eps \Vert \mathrm{D}^2_x u_\eps  \Vert_{\Ld^p(0,T;\Ld^p(\R^3_+))}^{2 \beta_p} (1+T)^{\zeta_p},
\end{align*}
for some $\zeta_p \in \left(0, 1\right)$. At the end of the day, we obtain the conclusion by gathering all the previous estimates together.
\end{proof}
Recall that $\alpha \in (0,1)$ has been be fixed once and for all before Definition \ref{def:tstar}. In the following result, we prove that the time $t_\eps^{\star}$ is bounded from below by some time independent of $\eps$, after which all the estimates based on the absorption effect will be available.
\begin{corollaire}\label{Coro:FindTM}
Under Assumption \textbf{\ref{hypSmallData}}, there exists $\eps_0=\eps_0(\alpha)>0$ such that for all $\eps \in (0, \eps_0)$, we have
\begin{align*}
T_{\mathrm{abs}}+10\alpha< t^{\star}_\eps.
\end{align*}
Furthermore, we have
\begin{align}
\label{estimBrinkL12L2:BEFORET_M} \int_0^{T_{\mathrm{abs}}+10\alpha} \Vert F_\eps(s) \Vert_{\Ld^2(\R^3_+)}^2 \, \mathrm{d}s +\int_0^{T_{\mathrm{abs}}+10\alpha} \Vert F_\eps(s) \Vert_{\Ld^2(\R^3_+)} \, \mathrm{d}s &<\frac{\mathrm{C}_{\star}}{4}, \\[2mm]
\label{estimU_eps:BEFORET_M}\int_0^{T_{\mathrm{abs}}+10\alpha} \Vert u_\eps(s) \Vert_{\W^{1, \infty}(\R^3_+)} \, \mathrm{d}s &< \dfrac{\delta^{\star}}{4}, \\[2mm]
\label{Brink:point:LOC} \forall t \in (0,T_{\mathrm{abs}}+10\alpha), \ \ \Vert F_\eps(t) \Vert_{\Ld^2(\R^3_+)} &\leq \frac{M^{\varpi_2}}{(1+T_{\mathrm{abs}}+10\alpha)^{\frac{7}{4}}},
\end{align}
where $\varpi_2>0$ is the universal constant given in Corollary \ref{coro:FpointL2}, where $\mathrm{C}_{\star}$ is the universal constant of Proposition \ref{propdatasmall:VNSreg} and where $\delta^{\star}$ has been introduced in the Definition \eqref{def:tstar} of $t_\eps^{\star}$. 

In particular, $T_{\mathrm{abs}}+10\alpha$ is a strong existence time.
\end{corollaire}
\begin{proof}
By combining Lemma \ref{BrinkL2:unif} and Corollary \ref{D_tu/D2:uL2}, we know that for $\eps>0$ small enough, we have for all $T \in (0,\min (T_{\mathrm{abs}}+10\alpha,t_\eps^{\star}))$
\begin{align*}
\Vert F_\eps \Vert_{\Ld^1 \cap \Ld^2(0,T;\Ld^2(\R^3_+))}& \lesssim (1+ \sqrt{T})\Vert F_\eps \Vert_{ \Ld^2(0,T;\Ld^2(\R^3_+))} \\
 & \lesssim (1+ \sqrt{T})\Big(\eps^{\frac{1}{2}} M^{\mu_2}+\eps M \Vert \partial_t u_\eps \Vert_{\Ld^2((0,T) \times \R^3_+)} + \eps  \Vert \D^2_x u_\eps \Vert_{\Ld^2((0,T) \times \R^3_+)}\\
  & \quad \quad \quad  \quad \quad \quad \quad +\eps \Vert \mathrm{D}^2_x u_\eps  \Vert_{\Ld^p(0,T;\Ld^p(\R^3_+))}^{2 \beta_p} (1+T)^{\zeta_p} + \Vert  f_\eps^0 \Vert_{\Ld^1(\R^3;\Ld^2(\R^3_+))}^{\frac{1}{2}} \Big),
\end{align*}
where $p \rightarrow 3^+$. We then use the uniform bounds of Corollary \ref{D_tu/D2:uL2} and of Corollary \ref{BrinkDt2uLp:unif}. 
Note that this last corollary actually requires the decay of the Brinkman force on $[0,T]$, provided by Lemma \ref{LM:pointwiseMin:initial}. 
This entails, according to Assumption \textbf{\ref{hypUnifBoundVNS}} and for $\eps$ small enough
\begin{align}\label{estimLoc1}
\Vert F_\eps \Vert_{\Ld^1 \cap \Ld^2(0,T;\Ld^2(\R^3_+))}& \lesssim (1+ \sqrt{T})\Big( \eps^{\frac{1}{2}} M^{\widetilde{\mu}_2}+ \eps M^{\widetilde{\mu}_p}(1+T)^{\zeta_p}+ \Vert  f_\eps^0 \Vert_{\Ld^1(\R^3;\Ld^2(\R^3_+))}^{\frac{1}{2}} \Big).
\end{align}

Furthermore, owing to the Gagliardo-Nirenberg-Sobolev inequality (see Theorem \ref{gagliardo-nirenberg} in the Appendix), we can write for all $T \in (0,\min (T_{\mathrm{abs}}+10\alpha,t_\eps^{\star}))$
\begin{align}\label{estimLoc2}
\begin{split}
\Vert  u_\eps \Vert_{\Ld^1(0,T;\mathrm{W}^{\infty}(\R^3_+))} &\lesssim \int_0^T \Vert \mathrm{D}^2_x u_\eps (s) \Vert^{\alpha_p}_{\Ld^p(\R^3_+)} \Vert u_\eps (s) \Vert^{1-\alpha_p}_{\Ld^2(\R^3_+)} \, \mathrm{d}s+\int_0^T \Vert \mathrm{D}^2_x u_\eps (s) \Vert^{\beta_p}_{\Ld^p(\R^3_+)} \Vert u_\eps (s) \Vert^{1-\beta_p}_{\Ld^2(\R^3_+)} \, \mathrm{d}s \\
& \lesssim T^{1-\alpha_p} {\mathcal{E}_\eps(0)}^{1-\alpha_p} \Vert \D^2_x u_\eps \Vert_{\Ld^p((0,T) \times \R^3_+)}^{p \alpha_p} + T^{1-\beta_p} {\mathcal{E}_\eps(0)}^{1-\beta_p} \Vert \D^2_x u_\eps \Vert_{\Ld^p((0,T) \times \R^3_+)}^{p \beta_p} \\
&\lesssim T^{1-\alpha_p} {\mathcal{E}_\eps(0)}^{1-\alpha_p} M^{\omega_p} + T^{1-\beta_p} {\mathcal{E}_\eps(0)}^{1-\beta_p}  M^{\tilde{\omega}_p},
\end{split}
\end{align}
where $p \rightarrow 3^+$ and for some $(\alpha_p, \beta_p) \in (0,1)^2$, thanks to Corollary \ref{BrinkDt2uLp:unif} and Assumption \textbf{\ref{hypUnifBoundVNS}}.

We then proceed as follows.
%
First, invoking Assumption \ref{hypSmallDataTech}, there exists $\eps_0 >0$ such that if we choose $\eps$, $\mathcal{E}_\eps(0)$ and $ \Vert  f_\eps^0 \Vert_{\Ld^1(\R^3;\Ld^2(\R^3_+))}$ small enough, then for all $\eps \in (0, \eps_0)$
\begin{align*}
M^{\omega_p } {\mathcal{E}_\eps(0)}^{1-\alpha_p}+ M^{\tilde{\omega}_p} {\mathcal{E}_\eps(0)}^{1-\beta_p} &<\frac{\delta^{\star}}{8\max \lbrace (T_{\mathrm{abs}}+10\alpha)^{1-\alpha_p}, (T_{\mathrm{abs}}+10\alpha)^{1-\beta_p} \rbrace}, \\
\eps^{\frac{1}{2}} M^{\widetilde{\mu}_2}+ \Vert  f_\eps^0 \Vert_{\Ld^1(\R^3;\Ld^2(\R^3_+))}^{\frac{1}{2}} &<\dfrac{1}{2 \left(1+ \sqrt{T_{\mathrm{abs}}+10\alpha}\right)}\left(\frac{\mathrm{C}_{\star}}{4} \right)^{1/2}, \\
\eps M^{\widetilde{\mu}_p} &< \dfrac{1}{2 (1+T_{\mathrm{abs}}+10\alpha)^{\zeta_p}\left(1+ \sqrt{T_{\mathrm{abs}}+10\alpha}\right)}\left(\frac{\mathrm{C}_{\star}}{4} \right)^{1/2}.
\end{align*}
Suppose now that there exists $\eps \in (0, \eps_0)$ such that $t^{\star}_\eps \leq T_{\mathrm{abs}}+10\alpha$ (in particular, $t^{\star}_\eps$ is finite and is a strong existence time). In view of Assumption \ref{hypSmallDataSTRONG}, the two previous inequalities combined with the continuity of $s \mapsto \Vert F_\eps \Vert_{\Ld^2(0,s;\Ld^2(\R^3_+))} $ and $s \mapsto \Vert  u_\eps \Vert_{\Ld^1(0,s;\mathrm{W}^{\infty}(\R^3_+))}$ at $t^{\star}_\eps$ (according to the integrability properties \eqref{integL2:Brink}, \eqref{integu:L1infty} and \eqref{nablauInL1Linfty}) lead to a contradiction with the definition of $t^{\star}_\eps$. 

This means that for all $\eps \in (0, \eps_0)$, we have $t^{\star}_\eps > T_{\mathrm{abs}}+10\alpha$ and this bound from below is independent of $\eps$ and can also be expressed as $$\min (T_{\mathrm{abs}}+10\alpha,t_\eps^{\star})= T_{\mathrm{abs}}+10\alpha,$$ for all $\eps \in (0, \eps_0)$. The previous choice also yields \eqref{estimBrinkL12L2:BEFORET_M} and \eqref{estimU_eps:BEFORET_M}.
The pointwise local estimate \eqref{Brink:point:LOC} for the Brinkman force is now a direct consequence of Lemma \ref{LM:pointwiseMin:initial}. This eventually concludes the proof.
\end{proof}



\textbf{Starting the absorption.}

Recall again that $\alpha \in (0,1)$ has been fixed just before Definition \ref{def:tstar} and is independent of $\eps$. We then consider the associated $\eps_0=\eps_0(\alpha)$ given by Corollary \ref{Coro:FindTM}. We know that 
\begin{align*}
\forall \eps \in (0, \eps_0), \ \ T_{\mathrm{abs}}+10 \alpha<t_\eps^{\star}.
\end{align*}
For any $\eps \in (0, \eps_0)$ and $t \in (T_{\mathrm{abs}}+10 \alpha,t_\eps^{\star})$, we observe that
\begin{align*}
t_{\frac{1}{2}}^g(1,1)=T_{\mathrm{abs}}<T_{\mathrm{abs}}+ \alpha<T_{\mathrm{abs}}+9\alpha<t-\alpha.
\end{align*}
According to Lemma \ref{EGCt:reverse}, the trivial velocity field satisfies $\mathrm{EGC}_{1/2}^{1+\ell_{1/2}^1(t-\alpha),1+r_{1/2}^1(t-\alpha)}(t-\alpha)$. Thus, by Remark \ref{RMK:EGCepsdiminue}, we obtain the fact that the trivial velocity field satisfies $\mathrm{EGC}_{\eps}^{1+\ell_{1/2}^1(t-\alpha),1+r_{1/2}^1(t-\alpha)}(t-\alpha)$ provided that $\eps<\frac{1}{4}$. Thus, if $\eps \in (0, \min(\frac{1}{4}, \eps_0))$ and $t \in (T_{\mathrm{abs}}+10 \alpha,t_\eps^{\star})$ then we know that
\begin{align*}
\int_0^t \Vert u_\eps(s) \Vert_{\Ld^{\infty}(\R^3_+)} \, \mathrm{d}s <\delta^{\star},
\end{align*}
therefore Lemma \ref{LM:perturbEGC} entails that the vector field $u_\eps$ satisfies $\mathrm{EGC}_{\eps}^{1+\ell_{1/2}^1(t-\alpha),1+r_{1/2}^1(t-\alpha)}(t)$. In addition, again according to Lemma \ref{EGCt:reverse}, we know that for all $t> T_{\mathrm{abs}}+10 \alpha$
\begin{align*}
\dfrac{1}{1+\ell_{1/2}^1(t-\alpha)} \lesssim \dfrac{1}{1+t-\alpha}\lesssim \dfrac{1}{1+t}, \ \ \dfrac{1}{1+r_{1/2}^1(t-\alpha)} \lesssim \dfrac{1}{1+t-\alpha}\lesssim \dfrac{1}{1+t},
\end{align*}
where $\lesssim$ refers to a constant independent of $\eps$ and $t$.

\begin{definition}\label{def:T_0}
We set
\begin{align}
T_0:=T_{\mathrm{abs}}+10 \alpha.
\end{align}
\end{definition}
We can sum up our results as follows. We have $T_0<t_\eps^{\star}$ for $\eps$ small enough and the following lemma holds.
\begin{lemme}\label{LM:résumé}
For all $\eps \in(0, \min(\frac{1}{4}, \eps_0))$ and for all $t \in (T_0,t_\eps^{\star})$, the vector field $u_\eps$ satisfies $\mathrm{EGC}_{\eps}^{1+\ell(t),1+r(t)}(t)$ for some continuous and positive functions $\ell$ and $r$ independent of $\eps$ and satisfying
\begin{align*}
\forall t \in (T_0,t_\eps^{\star}), \ \ \dfrac{1}{1+\ell(t)} \leq \dfrac{C}{1+t}, \ \  \dfrac{1}{1+r(t)} \leq \dfrac{C}{1+t},
\end{align*}
for some constant $C>0$ independent of $\eps$ and $t$. Furthermore, according to \eqref{Brink:point:LOC} and Corollary \ref{coro:FpointL2}, we have for all $T \in (T_0,t_\eps^{\star})$
\begin{align*}
\forall t \in [0,T], \ \ \Vert F_\eps(t) \Vert_{\Ld^2(\R^3_+)} &\leq \frac{M^{\varpi_2}}{(1+t)^{\frac{7}{4}}}.
\end{align*}
\end{lemme}
Roughly speaking, all the estimates of Subsection \ref{Subsect:F-LpLp} which involve the absorption effect (namely, with $k>0$) are now admissible.

\subsection{Weighted in time estimates}\label{Subsec:Weighted}
To obtain the fact that $t_\eps^{\star}=+\infty$, our final technical argument is based on an interpolation procedure of the form 
\begin{align}\label{strat:interpo}
 \Vert \nabla_x u_\eps(s) \Vert_{\Ld^{\infty}(\R^3_+)} \lesssim \Vert \D^2_x u_\eps(s) \Vert_{\Ld^{p}(\R^3_+)}^{\beta_p} \Vert u_\eps(s) \Vert_{\Ld^{2}(\R^3_+)}^{1-\beta_p},
\end{align}
for $p>3$ and $\beta_p \in (0,1)$. According to the exponent involved in the polynomial decay of $u_\eps$ provided by Theorem \ref{cond:decay}, we can't directly recover integrability result for large times. We are thus looking for some polynomial weighted in time versions of \eqref{strat:interpo}. While dealing with the higher order terms by maximal parabolic estimates, we ultimately rely on the polynomial decay estimates of the Brinkman force.

\medskip

Let us explain how one can obtain polynomial in time decay estimates for derivatives of $u_\eps$. As explained in Subsection \ref{subsection:Strat}, our approach is based on the maximal parabolic regularity theory. Setting
\begin{align*}
\U_\eps(t,x):=(1+t)^\gamma u_\eps(t,x), \ \ \gamma \geq 0,
\end{align*}
we observe that the function $\U_\eps$ satisfies the following Stokes system on $(0,+\infty) \times \R^3_+$:
\begin{equation*}
    \left\{
\begin{aligned}
\partial_t \U_\eps- \Delta_x \U_\eps + \nabla_x p_\eps&=S(u_\eps,f_\eps),\\
\mathrm{div}_x \, \U_\eps &=0, \\
{\U_\eps}_{\mid x_3=0}&=0, \\
\U_\eps(0, \cdot)&=u_\eps^0,
\end{aligned}
\right.
\end{equation*}
where
\begin{align*}
S(u_\eps,f_\eps):=(1+t)^\gamma(j_\eps-\rho_\eps u_\eps) -(1+t)^\gamma (u_\eps \cdot \nabla_x) u_\eps +\gamma(1+t)^{\gamma-1}u_\eps.
\end{align*}
Applying the Leray projection $\P$, the previous system also reads as
\begin{equation*}
    \left\{
\begin{aligned}
\partial_t \U_\eps+A_p \U_\eps &=\P S(u_\eps,f_\eps),\\
{\U_\eps}_{\mid x_3=0}&=0, \\
\U_\eps(0, \cdot)&=u_\eps^0,
\end{aligned}
\right.
\end{equation*}
where $A_p$ refers to the Stokes operator on $\Ld^p(\R^3_+)$. The maximal $\Ld^p_t \Ld^p_x$ parabolic regularity theory for the Stokes system with $p \in (1,\infty)$ (see in the Appendix \ref{AnnexeMaxregStokes}) leads to the following fact: for all $T>0$, we have
\begin{align*}
\Vert \partial_t \U_\eps \Vert_{\Ld^p(0,T;\Ld^p(\R^3_+))} + \Vert \mathrm{D}^2_x \U_\eps \Vert_{\Ld^p(0,T;\Ld^p(\R^3_+))} \lesssim \Vert \P S(u_\eps,f_\eps) \Vert_{\Ld^p(0,T;\Ld^p(\R^3_+))} + \Vert u_\eps^0 \Vert_{\mathrm{D}_p^{1-\frac{1}{p},p}(\R^3_+)},
\end{align*}
where $\lesssim$ is independent of $T$ and $\eps$. We thus obtain for all $T>0$
\begin{multline*}
\Vert (1+t)^{\gamma} \partial_t u_\eps \Vert_{\Ld^p(0,T;\Ld^p(\R^3_+))} + \Vert (1+t)^{\gamma}  \mathrm{D}^2_x u_\eps \Vert_{\Ld^p(0,T;\Ld^p(\R^3_+))} \\ 
\lesssim \Vert \P S(u_\eps,f_\eps) \Vert_{\Ld^p(0,T;\Ld^p(\R^3_+))}+ \gamma \Vert(1+t)^{\gamma-1}u_\eps \Vert_{\Ld^p(0,T;\Ld^p(\R^3_+))} + \Vert u_\eps^0 \Vert_{\mathrm{D}_p^{1-\frac{1}{p},p}(\R^3_+)},
\end{multline*}
and then, by continuity of the Leray projection on $\Ld^p(\R^3_+)$
\begin{align}\label{ineq:MaxReg:weightVNS}
\begin{split}
&\Vert (1+t)^{\gamma} \partial_t u_\eps   \Vert_{\Ld^p(0,T;\Ld^p(\R^3_+))} + \Vert (1+t)^{\gamma}  \mathrm{D}^2_x u_\eps \Vert_{\Ld^p(0,T;\Ld^p(\R^3_+))} \\
&\lesssim   \Vert (1+t)^\gamma F_\eps \Vert_{\Ld^p(0,T;\Ld^p(\R^3_+))} + \Vert (1+t)^\gamma (u_\eps \cdot \nabla_x) u_\eps \Vert_{\Ld^p(0,T;\Ld^p(\R^3_+))} +   \gamma \Vert(1+t)^{\gamma-1}u_\eps \Vert_{\Ld^p(0,T;\Ld^p(\R^3_+))}  \\
& \quad + \Vert u_\eps^0 \Vert_{\mathrm{D}_p^{1-\frac{1}{p},p}(\R^3_+)}.
\end{split}
\end{align}
The guiding line is therefore to get some estimates on the three first terms of the r.h.s in the previous inequality.

\medskip

We fix $T \in (T_0,t_\eps^{\star})$. Thanks to Lemma \ref{LM:résumé}, we are allowed to use the decay estimate \eqref{eq:decaythmcond} of Theorem \ref{cond:decay} on the interval $[0,T]$. This estimate explicitly involves some quantity $\Psi\left(\Vert u_\eps^0 \Vert_{\Ld^1 \cap \Ld^2(\R^3_+)}^2 +M^{\varpi_2} \right)$, where $\Psi$ is a nonnegative and nondecreasing continuous function. In view of Assumption \textbf{\ref{hypUnifBoundVNS}}, we shall get rid of this dependency on the initial data in the estimates.

\medskip

We first state two results dealing with two terms of the estimate \eqref{ineq:MaxReg:weightVNS}. The proof of \cite[Lemma 5.12 and Lemma 5.13]{HK} apply \textit{mutatis mutandis} to the half-space case and we refer to this paper for more details.
\begin{corollaire}\label{coro:source1}
Let $p>3$. There exists $\sigma>0$ such that for all $\gamma \in (0, \frac{17}{8}-\frac{7}{4p})$ and $\eps>0$, we have
\begin{align*}
\Vert (1+t)^{\gamma-1}u_\eps \Vert_{\Ld^p(0,T;\Ld^p(\R^3_+))} \lesssim M^{\sigma}.
\end{align*}
\end{corollaire}

\begin{corollaire}\label{coro:source2}
There exist $\varsigma>0$ and $\mu>0$ such that for all $p \in (3,3+\varsigma)$ and $\eps>0$, we have
\begin{align*}
\Vert (1+t)^{\gamma} (u_\eps \cdot \nabla_x)u_\eps \Vert_{\Ld^p(0,T;\Ld^p(\R^3_+))} \lesssim \mathcal{E}_\eps(0)^{\mu} \Vert (1+t)^{\gamma} \D^2_x u_\eps \Vert_{\Ld^p(0,T;\Ld^p(\R^3_+))}.
\end{align*}
\end{corollaire}
In order to deal with the term involving the Brinkman force in the estimate \eqref{ineq:MaxReg:weightVNS}, we shall rely on Corollary \ref{Coro:BrinkLr-w}. We are thus in position to state the following result.
\begin{lemme}\label{LM:lastboundD2uLp}
There exists $\varsigma>0$ and $\eps_0>0$ such that for all $p \in (3,3+\varsigma)$ and $\eps \in (0, \eps_0)$, the following holds. For all $\gamma \in (0, \frac{17}{8}-\frac{7}{4p})$, we have
\begin{align*}
\Vert (1+t)^{\gamma} \partial_t u_\eps   \Vert_{\Ld^p(0,T;\Ld^p(\R^3_+))} +\Vert (1+t)^{\gamma} \D^2_x u_\eps \Vert_{\Ld^p(0,T;\Ld^p(\R^3_+))} \lesssim M^{\widetilde{\omega}_p},
\end{align*}
for some $\widetilde{\omega}_p>0$.
\end{lemme}
\begin{proof}
According to the weighted maximal parabolic estimate \eqref{ineq:MaxReg:weightVNS}, we have 
\begin{align*}
&\Vert (1+t)^{\gamma} \partial_t u_\eps   \Vert_{\Ld^p(0,T;\Ld^p(\R^3_+))} + \Vert (1+t)^{\gamma}  \mathrm{D}^2_x u_\eps \Vert_{\Ld^p(0,T;\Ld^p(\R^3_+))} \\
&\lesssim   \Vert (1+t)^\gamma F_\eps \Vert_{\Ld^p(0,T;\Ld^p(\R^3_+))} + \Vert (1+t)^\gamma (u_\eps \cdot \nabla_x) u_\eps \Vert_{\Ld^p(0,T;\Ld^p(\R^3_+))} +   \gamma \Vert(1+t)^{\gamma-1}u_\eps \Vert_{\Ld^p(0,T;\Ld^p(\R^3_+))}  \\
& \quad + \Vert u_\eps^0 \Vert_{\mathrm{D}_p^{1-\frac{1}{p},p}(\R^3_+)}.
\end{align*}
For the first term in the r.h.s, we can invoke Corollary \ref{Coro:BrinkLr-w} and Remark \ref{Coro:BrinkLr-wLOC} with the corresponding exponent $p$, splitting the quantity $\Vert (1+t)^\gamma F_\eps \Vert_{\Ld^p(0,T;\Ld^p(\R^3_+))}$ in two parts between the intervals $(0,T_0)$ and $(T_0,T)$. The second and third term of the previous r.h.s are handled thanks to Corollary \ref{coro:source1} and Corollary \ref{coro:source2}. We get
\begin{align*}
\Vert (1+t)^{\gamma} \partial_t u_\eps   \Vert_{\Ld^p(0,T;\Ld^p(\R^3_+))} & + \Vert (1+t)^{\gamma}  \mathrm{D}^2_x u_\eps \Vert_{\Ld^p(0,T;\Ld^p(\R^3_+))} \\
&\lesssim    \eps^{\frac{1}{p}} M^{\omega_p} + M  +\eps \Big[1+(\Upsilon_\eps^0)^{\mu_p} \Big]M^{\omega_p} \Psi \left (\Vert  u_{\eps}^0 \Vert_{\Ld^{1}\cap\Ld^2(\R^3_+)}^2+M\right)^{\frac{1}{2}} \\
& \quad + \eps M \Vert (1+t)^k\partial_t  u_\eps \Vert_{\Ld^p(0,T;\Ld^p(\R^3_+))}  +\eps \Vert (1+t)^k\mathrm{D}^2_x u_\eps \Vert_{\Ld^p((0,T) \times \R^3_+)} \\
& \quad +\Vert (1+t)^\gamma F_\eps \Vert_{\Ld^p(0,T_0;\Ld^p(\R^3_+)) }\\
&\quad + \mathcal{E}_\eps(0)^{\mu} \Vert (1+t)^{\gamma} \D^2_x u_\eps \Vert_{\Ld^p(0,T;\Ld^p(\R^3_+))} +  \gamma M^{\sigma} \\
&  \quad +  \Vert u_\eps^0 \Vert_{\mathrm{D}_p^{1-\frac{1}{p},p}(\R^3_+)}.
\end{align*}
Taking for example $\eps \in (0, \frac{1}{2})$ such that 
\begin{align*}
\eps M <\frac{1}{2}, \ \  \mathcal{E}_\eps(0)^{\mu}<\frac{1}{4},
\end{align*}
thanks to Assumption \ref{hypSmallDataTech}, the previous estimate yields a bound of the type
\begin{align*}
&\Vert (1+t)^{\gamma} \partial_t u_\eps   \Vert_{\Ld^p(0,T;\Ld^p(\R^3_+))} + \Vert (1+t)^{\gamma}  \mathrm{D}^2_x u_\eps \Vert_{\Ld^p(0,T;\Ld^p(\R^3_+))} \\
&\lesssim     M^{\omega_p} + M  + \Big[1+(\Upsilon_\eps^0)^{\mu_p} \Big]M^{\omega_p} \Psi \left (\Vert  u_{\eps}^0 \Vert_{\Ld^{1}\cap\Ld^2(\R^3_+)}^2+M \right)^{\frac{1}{2}} +  \gamma M^{\sigma}+  \Vert u_\eps^0 \Vert_{\mathrm{D}_p^{1-\frac{1}{p},p}(\R^3_+)} \\
& \quad  +\Vert (1+t)^\gamma F_\eps \Vert_{\Ld^p(0,T_0;\Ld^p(\R^3_+))}  \\
& \lesssim M^{\omega_p} + M  + \Big[1+(\Upsilon_\eps^0)^{\mu_p} \Big]M^{\omega_p} \Psi \left (\Vert  u_{\eps}^0 \Vert_{\Ld^{1}\cap\Ld^2(\R^3_+)}^2+M \right)^{\frac{1}{2}} +  \gamma M^{\sigma}+  \Vert u_\eps^0 \Vert_{\mathrm{D}_p^{1-\frac{1}{p},p}(\R^3_+)} \\
& \quad + (1+T_0)^{\gamma} \Big( \eps^{\frac{1}{p}} M^{\omega_r} +\eps M^{\omega_p}+ M  +\eps \Big[1+(\Upsilon_\eps^0)^{\mu_r} \Big]M^{\omega_r} \Psi \left (\Vert  u_{\eps}^0 \Vert_{\Ld^{1}\cap\Ld^2(\R^3_+)}^2+M \right)^{\frac{1}{2}}  \Big),
\end{align*}
where we have used Remark \ref{Coro:BrinkLr-wLOC}. Again by Assumption \textbf{\ref{hypUnifBoundVNS}}, this reads
\begin{align*}
\Vert (1+t)^{\gamma} \partial_t u_\eps   \Vert_{\Ld^p(0,T;\Ld^p(\R^3_+))} + \Vert (1+t)^{\gamma}  \mathrm{D}^2_x u_\eps \Vert_{\Ld^p(0,T;\Ld^p(\R^3_+))} \lesssim  M^{\widetilde{\omega}_p},
\end{align*} 
where $\widetilde{\omega}_p>0$. This yields the claimed estimate.
\end{proof}

\subsection{Conclusion of the bootstrap argument}
Recall the Definition \eqref{def:tstar} of the time $t_{\eps}^{\star}$. We eventually reach the following result.
\begin{proposition}\label{Prop:CclBootstrap}
There exists $\eps_0 >0$ such that for all $\eps \in (0, \eps_0)$ and for all $T \in (T_{\mathrm{abs}}+10\alpha, t_{\eps}^{\star})$
 \begin{align*}
  \int_{T_{\mathrm{abs}}+10\alpha}^T \Vert u_\eps(s) \Vert_{\W^{1, \infty}(\R^3_+)} \, \mathrm{d}s &\lesssim \mathcal{E}_\eps(0)^{\theta}.
 \end{align*}
 for some $\theta>0$, and where $\lesssim$ only depends on the constant $M$.
\end{proposition}
\begin{proof}
By the Gagliardo-Nirenberg-Sobolev inequality (see Theorem \ref{gagliardo-nirenberg} in the Appendix), we have for all $s \in (T_{\mathrm{abs}}+10\alpha, t_{\eps}^{\star})$
\begin{align*}
\Vert u_\eps(s) \Vert_{\W^{1, \infty}(\R^3_+)} \lesssim \Vert \mathrm{D}^2_x u_\eps (s) \Vert^{\alpha_p}_{\Ld^p(\R^3_+)} \Vert u_\eps (s) \Vert^{1-\alpha_p}_{\Ld^2(\R^3_+)} + \Vert \mathrm{D}^2_x u_\eps (s) \Vert^{\beta_p}_{\Ld^p(\R^3_+)} \Vert u_\eps (s) \Vert^{1-\beta_p}_{\Ld^2(\R^3_+)},
\end{align*}
 where $$p>3, \ \ \alpha_p= \dfrac{3p}{7p-6}, \ \ \beta_p=\dfrac{5p}{7p-6}.$$
 Hence, for all $\gamma>0$, we have by the Hölder inequality in time
 \begin{align*}
 \int_{T_{\mathrm{abs}}+10\alpha}^T \Vert u_\eps(s) \Vert_{\W^{1, \infty}(\R^3_+)} \, \mathrm{d}s &\lesssim  \int_{T_{\mathrm{abs}}+10\alpha}^T  \Vert \mathrm{D}^2_x u_\eps (s) \Vert^{\alpha_p}_{\Ld^p(\R^3_+)} \Vert u_\eps (s) \Vert^{1-\alpha_p}_{\Ld^2(\R^3_+)} \, \mathrm{d}s \\
 & \quad + \int_{T_{\mathrm{abs}}+10\alpha}^T  \Vert \mathrm{D}^2_x u_\eps (s) \Vert^{\beta_p}_{\Ld^p(\R^3_+)} \Vert u_\eps (s) \Vert^{1-\beta_p}_{\Ld^2(\R^3_+)} \, \mathrm{d}s \\
& \leq  \int_{0}^T \left[ \Vert (1+s)^{\gamma} \mathrm{D}^2_x u_\eps (s) \Vert_{\Ld^p(\R^3_+)} \right]^{\alpha_p} (1+s)^{-\gamma \alpha_p}\Vert u_\eps (s) \Vert^{1-\alpha_p}_{\Ld^2(\R^3_+)} \, \mathrm{d}s \\
& \quad +\int_{0}^T  \left[ \Vert (1+s)^{\gamma} \mathrm{D}^2_x u_\eps (s)\Vert_{\Ld^p(\R^3_+)} \right]^{\beta_p} (1+s)^{-\gamma \beta_p} \Vert u_\eps (s) \Vert ^{1-\beta_p}_{\Ld^2(\R^3_+)} \, \mathrm{d}s \\
&\leq \Vert (1+t)^{\gamma} \D^2_x u_\eps \Vert_{\Ld^p(0,T;\Ld^p(\R^3_+))}^{\alpha_p} \left( \int_{0}^T (1+s)^{-\gamma \frac{p \alpha_p}{p-\alpha_p}} \mathcal{E}_\eps(s)^{\frac{1-\alpha_p}{2}\frac{p}{p-\alpha_p}} \, \mathrm{d}s \right)^{\frac{p-\alpha_p}{p}} \\
& \quad + \Vert (1+t)^{\gamma} \D^2_x u_\eps \Vert_{\Ld^p(0,T;\Ld^p(\R^3_+))}^{\beta_p} \left( \int_{0}^T (1+s)^{-\gamma \frac{p \beta_p}{p-\beta_p}}  \mathcal{E}_\eps(s)^{\frac{1-\beta_p}{2}\frac{p}{p-\beta_p}} \, \mathrm{d}s \right)^{\frac{p-\beta_p}{p}},
 \end{align*}
 therefore, by the energy-dissipation inequality \eqref{ineq:energy2}
 \begin{align*}
 \int_{T_{\mathrm{abs}}+10\alpha}^T \Vert u_\eps(s) \Vert_{\W^{1, \infty}(\R^3_+)} \, \mathrm{d}s &\lesssim    \mathcal{E}_\eps(0)^{\frac{1-\alpha_p}{2}} \Vert (1+t)^{\gamma} \D^2_x u_\eps \Vert_{\Ld^p(0,T;\Ld^p(\R^3_+))}^{\alpha_p}  \left(\int_{0}^T (1+s)^{-\gamma \frac{p \alpha_p}{p-\alpha_p}} \, \mathrm{d}s \right)^{\frac{p-\alpha_p}{p}} \\
& \quad +  \mathcal{E}_\eps(0)^{\frac{1-\beta_p}{2} } \Vert (1+t)^{\gamma} \D^2_x u_\eps \Vert_{\Ld^p(0,T;\Ld^p(\R^3_+))}^{\beta_p} \left(  \int_{0}^T (1+s)^{-\gamma \frac{p \beta_p}{p-\beta_p}} \, \mathrm{d}s \right)^{\frac{p-\beta_p}{p}}.
 \end{align*}
 Setting $\gamma_p:=\frac{17}{8}-\frac{7}{4p}$, we observe that for $p>3$, we have
 \begin{align*}
\min\left(\frac{p \beta_p}{p-\beta_p},\frac{p \alpha_p}{p-\alpha_p} \right)> \frac{1}{\gamma_p}, 
 \end{align*}
 therefore we can pick some $\gamma \in (0, \gamma_p)$ (which depends on $p$) such that 
 \begin{align*}
 \gamma \frac{p \beta_p}{p-\beta_p}>1, \ \ \gamma \frac{p \alpha_p}{p-\alpha_p}>1.
 \end{align*}
The two previous integrals in time are thus bounded uniformly in $T$. Owing to the uniform bound of Lemma \ref{LM:lastboundD2uLp}, we get
 \begin{align*}
  \int_{T_{\mathrm{abs}}+10\alpha}^T \Vert u_\eps(s) \Vert_{\W^{1, \infty}(\R^3_+)} \, \mathrm{d}s &\lesssim \mathcal{E}_\eps(0)^{\frac{1-\alpha_p}{2}}+ \mathcal{E}_\eps(0)^{\frac{1-\beta_p}{2}},
 \end{align*}
 where $\lesssim$ only depends on the constant $M$. Taking the maximum of the two last quantities yields the result.
\end{proof}

\begin{proposition}\label{t_star=infty}
There exists $\eps_0 >0$ such that if $\eps \in (0, \eps_0)$ then $t_{\eps}^{\star}=+\infty$.
\end{proposition}
\begin{proof}
According to the estimate \eqref{estimU_eps:BEFORET_M} from Corollary \ref{Coro:FindTM} and to Proposition \ref{Prop:CclBootstrap}, we have for all $T \in (T_{\mathrm{abs}}+10\alpha, t_\eps^{\star})$
 \begin{align*}
  \int_{0}^T \Vert u_\eps(s) \Vert_{\W^{1, \infty}(\R^3_+)} \, \mathrm{d}s &< \frac{\delta^{\star}}{4}+ \mathrm{A}\mathcal{E}_\eps(0)^{\theta},
 \end{align*}
for some $\theta>0$ and $\mathrm{A}>0$, provided that $\eps>0$ is small enough. Thus, according to the smallness condition of Assumption \ref{hypSmallDataTech}, we can ensure that for all $T \in (T_{\mathrm{abs}}+10\alpha, t_\eps^{\star})$
 \begin{align}\label{To:Bootstrap1}
  \int_{0}^T \Vert u_\eps(s) \Vert_{\W^{1, \infty}(\R^3_+)} \, \mathrm{d}s &< \frac{\delta^{\star}}{2}.
 \end{align}
 On the other hand, according to the estimate \eqref{estimBrinkL12L2:BEFORET_M} from Corollary \ref{Coro:FindTM}, we know that 
 \begin{align*}
 \int_0^{T_{\mathrm{abs}}+10\alpha} \Vert F_\eps(s) \Vert_{\Ld^2(\R^3_+)}^2 \, \mathrm{d}s +\int_0^{T_{\mathrm{abs}}+10\alpha} \Vert F_\eps(s) \Vert_{\Ld^2(\R^3_+)} \, \mathrm{d}s &<\frac{\mathrm{C}_{\star}}{4}.
 \end{align*}
Using the inequality
\begin{align*}
\Vert F_\eps(s) \Vert_{\Ld^2(\R^3_+)}^2 \leq \Vert \rho_\eps (s)\Vert_{\Ld^{\infty}(\R^3_+)} \D_\eps(s),
\end{align*}
and the energy dissipation inequality \eqref{ineq:energy2}, we also have for all $T \in (T_{\mathrm{abs}}+10\alpha, t_\eps^{\star})$
\begin{align*}
  \int_{T_{\mathrm{abs}}+10\alpha}^T \Vert F_\eps(s) \Vert_{\Ld^2(\R^3_+)}^2 \, \mathrm{d}s &\leq \Vert \rho_\eps \Vert_{\Ld^{\infty}(0,T;\Ld^{\infty}(\R^3_+))} \mathcal{E}_\eps(0) \leq \Vert f_\eps^0 \Vert_{\Ld^1(\R^3; \Ld^{\infty}(\R^3_+))} \mathcal{E}_\eps(0) \leq M\mathcal{E}_\eps(0),
\end{align*}
 thanks to Corollary \ref{coro:boundrho}. Using Assumption \ref{hypSmallDataTech}, one can ensure 
 \begin{align*}
  \int_{T_{\mathrm{abs}}+10\alpha}^T \Vert F_\eps(s) \Vert_{\Ld^2(\R^3_+)}^2 \, \mathrm{d}s < \frac{\mathrm{C}_{\star}}{16}.
\end{align*}
By the Cauchy-Schwarz inequality, we also have
 \begin{align*}
   \int_{T_{\mathrm{abs}}+10\alpha}^T \Vert F_\eps(s) \Vert_{\Ld^2(\R^3_+)} \, \mathrm{d}s \lesssim \left(\int_{T_{\mathrm{abs}}+10\alpha}^T (1+t)^{\gamma} \Vert F_\eps(s) \Vert_{\Ld^2(\R^3_+)}^2 \, \mathrm{d}s \right)^{\frac{1}{2}},
 \end{align*}
for any $\gamma>1$ (take for instance $\gamma=1^+$ to be optimal). We then invoke the pointwise estimates obtained in Subsection \ref{Subsect:FpointL2}: via Lemmas \ref{LM:F0-pointL2}--\ref{LM:F1-pointL2}--\ref{LM:F2-pointL2} and the uniform bound of Assumption \textbf{\ref{hypUnifBoundVNS}}, we have for all $t \in (T_{\mathrm{abs}}+10\alpha,T)$
 \begin{align*}
 \Vert F_\eps(s) \Vert_{\Ld^2(\R^3_+)}^2& \lesssim  \frac{e^{\frac{-2t}{\eps}}}{(1+t)^{k_1}} M^{\iota_2}+\frac{1}{(1+t)^{2k_2}} \left[\Vert \vert v \vert^{2k_2}   f_\eps^0 \Vert_{\Ld^1(\R^3;\Ld^2(\R^3_+))}^2+\Vert  x_3^{2k_2}   f_\eps^0 \Vert_{\Ld^1(\R^3;\Ld^2(\R^3_+))}^2 \right]\\
 & \quad +\frac{\eps}{(1+t)^{k_3}}M^{\iota_2} + \frac{ \eps^{\frac{3}{2}}}{(1+t)^{k_4}} M^{\iota_2}+\frac{\mathcal{E}_\eps(0)}{(1+t)^{k_5-\frac{1}{2}}} M^{\iota_2},
 \end{align*}
for some $\iota_2>0$ and where the exponents $k_i$ ($i=1, \cdots,5$) are choosen as follows, according to Assumption \textbf{\ref{hypUnifBoundVNS}}: we take $k_1=\gamma$, $2k_2>1+\gamma$, $k_3>1+\gamma$, $k_4 >1+\gamma$ and $k_5>\frac{3}{2}+\gamma$. We end up with
 \begin{align*}
   \int_{T_{\mathrm{abs}}+10\alpha}^T \Vert F_\eps(s) \Vert_{\Ld^2(\R^3_+)} \, \mathrm{d}s \lesssim \eps^{\frac{1}{2}} + \eps^{\frac{3}{4}} + \mathcal{E}_\eps(0)^{\frac{1}{2}} + \Vert \vert v \vert^{2k_2}   f_\eps^0 \Vert_{\Ld^1(\R^3;\Ld^2(\R^3_+))}+\Vert  x_3^{2k_2}   f_\eps^0 \Vert_{\Ld^1(\R^3;\Ld^2(\R^3_+))}.
 \end{align*}
 According to Assumption \ref{hypSmallDataTech}, we can ensure that the last quantity is choosen so that
 \begin{align*}
    \int_{T_{\mathrm{abs}}+10\alpha}^T \Vert F_\eps(s) \Vert_{\Ld^2(\R^3_+)}< \frac{\mathrm{C}_{\star}}{16}.
 \end{align*}
Gathering the previous estimates together, we get for all $T \in (0, t_\eps^{\star})$
\begin{align}\label{To:Bootstrap2}
\Vert u_\eps^0 \Vert_{\H^1(\R^3_+)}^2+\int_0^{T} \Vert F_\eps(s) \Vert_{\Ld^2(\R^3_+)}^2 \, \mathrm{d}s +\int_0^{T} \Vert F_\eps(s) \Vert_{\Ld^2(\R^3_+)} \, \mathrm{d}s &<\frac{7\mathrm{C}_{\star}}{8},
\end{align}
thanks to Assumption \ref{hypSmallDataSTRONG}. The estimates \eqref{To:Bootstrap1} and \eqref{To:Bootstrap2} therefore hold for $\eps>0$ small enough, say $\eps \in (0, \eps^{\star})$.

Now assume that there exists $\eps \in  (0, \eps^{\star})$ such that $t_\eps^{\star}<\infty$. Invoking the continuity of $s \mapsto \Vert F_\eps \Vert_{\Ld^2\cap \Ld^1(0,s; \Ld^2(\R^3_+))}$ (given by \eqref{integL2:Brink}), one observes that the estimate \eqref{To:Bootstrap2} entails there exists a strong existence time strictly greater than $t_\eps^{\star}$. An additional continuity argument (owing to the integrability results \eqref{integu:L1infty} and \eqref{nablauInL1Linfty}) combined with the estimate \eqref{To:Bootstrap1} shows that there exists a strong existence time $T^{\eps}>t_\eps^{\star}$ which satisfies
\begin{align*}
  \int_{0}^{T^{\eps}} \Vert u_\eps(s) \Vert_{\W^{1, \infty}(\R^3_+)} \, \mathrm{d}s &< \frac{3\delta^{\star}}{4}.
\end{align*}
This is a contradiction with the definition of $t_\eps^{\star}$ and this finally achieves the proof of the proposition.

\end{proof}

We now turn to the proof of Theorem \ref{thCV}, where we obtain the weak convergence of $(\rho_\eps, u_\eps)$ towards a solution $(\rho,u)$ of the Boussinesq-Navier-Stokes system \eqref{BNS}, assuming the convergence of the initial condition $(\rho_\eps^0, u_\eps^0)$.

As explained in the introduction, our proof is based on all the uniform (in $\eps$) estimates that we have obtained as a byproduct of the previous boostrap strategy and which now hold on any interval of time since $t_\eps^{\star}=+\infty$. More precisely, we work in the framework underlined by the conditional Proposition \ref{IF:Propo}.
\begin{proof}[Proof of Theorem \ref{thCV} ]
Let $T>0$ be any fixed time. Let us show that the assumptions of Proposition \ref{IF:Propo} are satisfied for $\eps$ small enough.
\begin{itemize}
\item \eqref{C1} is satisfied thanks to Proposition \ref{t_star=infty} and the very definition of $t_\eps^{\star}$.
\item \eqref{C2} is satisfied in view of Corollary \ref{D_tu/D2:uL2} and because any time is a strong existence time.
\item \eqref{C3} is satisfied in view of Assumption \textbf{\ref{hypUnifBoundVNS}} on the initial data.
\item \eqref{C4} can be obtained by using the fact that
\begin{align*}
\forall s \geq 0, \ \ \Vert  u_\eps(s) \Vert_{\Ld^2(\R^3_+)} \lesssim \frac{\Psi \left (\Vert  u_{\eps}^0 \Vert_{\Ld^{1}\cap\Ld^2(\R^3_+)}^2+M \right)}{(1+s)^{\frac{3}{4}}}.
\end{align*}
Indeed, performing the same computations as in the end of Lemma \ref{BrinkL2:unif}, one can use the previous decay in time to prove that
\begin{align*}
\int_0^T   \Vert \nabla_x u_\eps(s) \Vert_{\Ld^{\infty}(\R^3_+)}^{2} \, \mathrm{d}s  \lesssim  \Vert \mathrm{D}^2_x u_\eps  \Vert_{\Ld^p(0,T;\Ld^p(\R^3_+))}^{2 \beta_p} (1+T)^{\zeta_p},
\end{align*}
for any $p>3$ and some $(\beta_p, \zeta_p) \in \left(0, 1\right) \times (0, + \infty)$. Owing to Lemma \ref{LM:lastboundD2uLp} for $p \rightarrow 3^+$ then ensures \eqref{C4}.
\end{itemize}
Applying the conditional Proposition \ref{IF:Propo} on the interval $[0,T]$ yields the claimed convergence when $\eps \rightarrow 0$, up to an extraction. A final diagonal extraction along increasing interval of times shows that we can find a common extraction which is valid for all times.

According to the estimate \eqref{ineq:MaxReg:weightVNS} of Proposition \ref{propdatasmall:VNSreg} which holds for all times, and to the very definition of $t_\eps^{\star}=+\infty$, another weak compactness argument shows that the accumulation point $u$ we have obtained before satisfies
\begin{align*}
u \in \Ld^{\infty}(\R^+;\H^1(\R^3_+))\cap \Ld^2(\R^+;\H^2(\R^3_+))\cap \Ld^1(\R^+;\W^{1,\infty}(\R^3_+)).
\end{align*}
This eventually concludes the proof.
\end{proof}
\begin{remarque}\label{rmq:rateCVGENCE}
Following precisely the exponent involved in the proof of \ref{BrinkL2:unif} and which depend on $p \rightarrow 3^+$ (especially for the treatment of the term $\mathrm{III}$), one can prove that there exists $\mu_2>0$ such that for all $\eps>0$ small enough, we have
\begin{align*}
\Vert j_\eps-\rho_\eps u_\eps +\rho_\eps e_3 \Vert_{\Ld^2((0,T) \times \R^3_+)} &\lesssim M^{\mu_2} \left(\eps^{\frac{1}{2}}+ \eps+ \eps (1+T)^{\frac{1}{5}} \right), 
\end{align*}
provided that Assumptions~\textbf{\ref{hypGeneral}}--\textbf{\ref{hypUnifBoundVNS}}--\textbf{\ref{hypSmallData}} hold.
\end{remarque}

\subsection{Rates of strong convergence}

We finally deal with the proof of Theorem \ref{th-rate:cvgence}: namely, we are looking for some rate of convergence in order to quantify the result of Theorem \ref{thCV}, assuming strong convergence of the initial data. To do so, we use a proof based on energy estimates.
\begin{proof}[Proof of Theorem \ref{th-rate:cvgence}]
Let $(\rho,u)$ be any global strong solution of the Boussinesq-Navier-Stokes system \eqref{BNS} (namely, $(\rho,u) \in \Ld^{\infty}_{\mathrm{loc}}(\R^+;\Ld^{\infty}(\R^3_+ \times \R^3)) \times \Ld^{\infty}_{\mathrm{loc}}(\R^+;\H^1_{\mathrm{div}}(\R^3_+))\cap \Ld^2_{\mathrm{loc}}(\R^+;\H^2(\R^3_+))\cap \Ld^1(\R^+;\W^{1,\infty}(\R^3_+))$): note that such a solution indeed exists thanks to the previous subsection. We set for any $\eps>0$
\begin{align*}
w_\eps:= u_\eps-u, \ \ \theta_\eps:= \rho_\eps-\rho.
\end{align*}
Let $T>0$. We observe that $(w_\eps, \theta_\eps)$ satisfies (at least in the sense of distributions)
\begin{equation}\label{syst:rate:rho}
    \left\{
\begin{aligned}
\partial_t \theta_\eps+ \mathrm{div}_x \, [\theta_\eps (u_\eps-e_3) ]&=-\mathrm{div}_x \left[ F_\eps+\rho_\eps e_3+\rho(u_\eps-u) \right], \\
{\theta_\eps}_{\mid t=0}&= \rho_\eps^0-\rho^0,
\end{aligned}
\right.
\end{equation}
and
\begin{equation}\label{syst:rate:w}
    \left\{
\begin{aligned}
\partial_t w_\eps + (u \cdot \nabla_x)w_\eps -\Delta_x w_\eps + \nabla_x(p_\eps-p)&=F_\eps+\rho_\eps e_3 -  \theta_\eps e_3 - (w_\eps \cdot \nabla_x)u_\eps, \\
\mathrm{div}_x \, w_\eps &=0, \\
{w_\eps}_{\mid t=0}&= u_\eps^0-u^0.
\end{aligned}
\right.
\end{equation}
Since $u(t)$ is divergence free and all the solutions are strong for all times, a classical energy estimate in $\Ld^2_x$ for the system \eqref{syst:rate:w} leads to
\begin{align*}
\frac{1}{2}\Vert w_\eps(t) \Vert_{\Ld^2(\R^3_+)}^2 + \int_0^t \Vert \nabla_x w_\eps(s) \Vert_{\Ld^2(\R^3_+)}^2 \, \mathrm{d}s  &\lesssim \frac{1}{2}\Vert w_\eps^0 \Vert_{\Ld^2(\R^3_+)}^2 \\
& \quad + \int_0^t \Vert w_\eps(s) \Vert_{\Ld^2(\R^3_+)}^{\frac{1}{2}} \Vert \nabla_x w_\eps(s) \Vert_{\Ld^2(\R^3_+)}^{\frac{3}{2}} \Vert \nabla_x u_\eps(s) \Vert_{\Ld^2(\R^3_+)} \, \mathrm{d}s \\
& \quad + \int_0^t \int_{\R^3_+} \left\vert \left[ F_\eps+\rho_\eps e_3 +  \theta_\eps e_3 \right](s,x) \right\vert \vert  w_\eps(s,x) \vert  \, \mathrm{d}x \, \mathrm{d}s,
\end{align*}
for any $t \in [0,T]$. Using Proposition \ref{propdatasmall:VNSreg} and Young inequality, we can write
\begin{multline*}
\int_0^t \Vert w_\eps(s) \Vert_{\Ld^2(\R^3_+)}^{\frac{1}{2}} \Vert \nabla_x w_\eps(t) \Vert_{\Ld^2(\R^3_+)}^{\frac{3}{2}} \Vert \nabla_x u_\eps(s) \Vert_{\Ld^2(\R^3_+)} \, \mathrm{d}s \\ \lesssim \frac{c^{-4}}{4} \int_0^t \Vert w_\eps(s) \Vert_{\Ld^2(\R^3_+)}^{2} \, \mathrm{d}s + \frac{3c^{\frac{4}{3}}}{4}  \int_0^t \Vert \nabla_x w_\eps(s) \Vert_{\Ld^2(\R^3_+)}^{2} \, \mathrm{d}s,
\end{multline*}
for any $c>0$. Choosing $c>0$ small enough, we can absorb the last term and get for all $t \in [0,T]$
\begin{align*}
&\frac{1}{2}\Vert w_\eps(t) \Vert_{\Ld^2(\R^3_+)}^2  + \frac{1}{2}\int_0^t \Vert \nabla_x w_\eps(s) \Vert_{\Ld^2(\R^3_+)}^{2} \, \mathrm{d}s  \\
&\lesssim  \Vert w_\eps^0 \Vert_{\Ld^2(\R^3_+)}^2 + \int_0^t \Vert w_\eps(s) \Vert_{\Ld^2(\R^3_+)}^{2} \, \mathrm{d}s+  \int_0^t \int_{\R^3_+} \left\vert \left[ F_\eps+\rho_\eps e_3 -  \theta_\eps e_3 \right](s,x) \right\vert \vert  w_\eps(s,x) \vert  \, \mathrm{d}x \, \mathrm{d}s  \\
& \leq \Vert w_\eps^0 \Vert_{\Ld^2(\R^3_+)}^2+ \int_0^t \Vert w_\eps(s) \Vert_{\Ld^2(\R^3_+)}^{2} \, \mathrm{d}s \\
& \quad + \int_0^t \Vert  F_\eps(s)+\rho_\eps(s) e_3  \Vert_{\Ld^2(\R^3_+)}  \Vert  w_\eps(s) \Vert_{\Ld^2(\R^3_+)} \, \mathrm{d}s+\int_0^t \Vert   \theta_\eps(s) \Vert_{\H^{-1}(\R^3_+)}  \Vert  w_\eps(s) \Vert_{\H^1_0(\R^3_+)} \, \mathrm{d}s \\
& \leq \Vert w_\eps^0 \Vert_{\Ld^2(\R^3_+)}^2+ \frac{3}{2}\int_0^t \Vert w_\eps(s) \Vert_{\Ld^2(\R^3_+)}^{2} \, \mathrm{d}s+  \int_0^t \Vert  F_\eps(s)+\rho_\eps(s) e_3  \Vert_{\Ld^2(\R^3_+)}^2\, \mathrm{d}s  \\
&  \quad \qquad \qquad  \qquad  \qquad  \qquad  \qquad  \qquad \qquad  \qquad   +\frac{c^{-1}}{2} \int_0^t \Vert   \theta_\eps(s) \Vert_{\H^{-1}(\R^3_+)}^2 \, \mathrm{d}s +\frac{c}{2}\int_0^t   \Vert  w_\eps(s) \Vert_{\H^1_0(\R^3_+)}^2 \, \mathrm{d}s,
\end{align*}
for any $c>0$, as above. Therefore, choosing again $c$ small enough, we can absorb the gradient part of $\Vert  w_\eps(s) \Vert_{\H^1_0(\R^3_+)}^2$ in the l.h.s, and this yields for all $t \in [0,T]$
\begin{align}\label{ineqRate:preGronwall}
\begin{split}
&\Vert w_\eps(t) \Vert_{\Ld^2(\R^3_+)}^2  \\
&\lesssim \Vert w_\eps^0 \Vert_{\Ld^2(\R^3_+)}^2 + \Vert F_\eps+\rho_\eps e_3 \Vert_{\Ld^2(0,T;\Ld^2(\R^3_+))}^2+  \int_0^t \Vert   \theta_\eps(s) \Vert_{\H^{-1}(\R^3_+)}^2 \, \mathrm{d}s+  \int_0^t  \Vert  w_\eps(s) \Vert_{\Ld^2(\R^3_+)}^2 \, \mathrm{d}s.
\end{split}
\end{align}
Hence, we need to derive an estimate on $\theta_\eps$ in $\Ld^2_t \H^{-1}_x$. This is the purpose of the following lemma.

\begin{lemme}
For all $s \in [0,T]$, we have 
\begin{align}\label{estimateH-1}
\Vert  \theta_\eps(s) \Vert_{\H^{-1}(\R^3_+)}^2 \lesssim \Vert  \theta_\eps^0 \Vert_{\H^{-1}(\R^3_+)}^2 + T \Vert F_\eps+\rho_\eps e_3 \Vert_{\Ld^2(0,T;\Ld^2(\R^3_+))}^2 +  T\int_0^s  \Vert  w_\eps(\tau) \Vert_{\Ld^2(\R^3_+)}^2 \, \mathrm{d}\tau.
\end{align}
\end{lemme}
\begin{proof}Let us define the characteristic curves associated to the continuity equation \eqref{syst:rate:rho}: for $t \in \R^+$ and $x \in \R^3$, we consider $s \mapsto \mathfrak{X}_\eps(s;t,x)$ the solution to
\begin{equation}\label{EDO-charac:RATE}
\left\{
      \begin{aligned}
        \dot{\mathfrak{X}}_{\varepsilon}(s;t,x) &=(Pu_{\eps})(s,\mathfrak{X}_\eps(s;t,x))-e_3,\\[2mm]
	\mathfrak{X}_{\varepsilon}(t;t,x)&=x,\\
      \end{aligned}
    \right.
\end{equation}
where the dot means derivative along the first variable. Note that $\left[(Pu_{\eps})(s,x)-e_3 \right]{\cdot} n(x)=1>0$ for all $x \in \lbrace x_3=0 \rbrace$ therefore 
\begin{align*}
y \in \R^3_+ &\Longrightarrow \forall \, 0 \leq s \leq t , \ \ \mathfrak{X}(s;t,y) \in \R^3_+, \\
y \in  \lbrace x_3=0 \rbrace &\Longrightarrow  \forall \, t > 0, \ \ \mathfrak{X}(t;0,y) \notin \R^3_+.
\end{align*}
%
In addition, since $\mathrm{div}_x[(Pu_{\eps})(s)-e_3]=0$, we have $\det \, \D_x \mathfrak{X}(s;t,x)=1$ and
\begin{align}\label{Dx-convrateH-1}
\Vert \D_x \mathfrak{X} (s;t,\cdot) \Vert_{\Ld^{\infty}(\R^3)} \leq e^{\Vert \nabla_x (Pu_\eps- e_3) \Vert_{\Ld^1(0,T;\Ld^{\infty}(\R^3))}}  \lesssim 1,
\end{align}
since $t_\eps^{\star}=+\infty$, at least for $\eps$ small enough. Our proof is thus based on the following identity: for any test function $\varphi \in \mathscr{C}_{c}^{\infty}(\R^3_+)$ and any $t \geq 0$, we have
\begin{align}\label{weak:form:RATE:transport}
\int_{\R^3_+} \theta_\eps(t,x) \varphi(x) \, \mathrm{d}x=\int_{\R^3_+} \theta_\eps^0(\mathfrak{X}_\eps(0;t,x) ) \varphi(x) \, \mathrm{d}x+ \int_0^t \int_{\R^3_+}    H_{\eps}(s,x) \cdot \nabla_x [ \varphi(\mathfrak{X}_{\eps}(t;s,x))] \, \mathrm{d}x \, \mathrm{d}s,
\end{align}
where 
\begin{align*}
H_{\eps}:= F_\eps+\rho_\eps e_3+\rho(u_\eps-u).
\end{align*}
Let us fix $\varphi \in \mathscr{C}_{c}^{\infty}(\R^3_+)$ such that $\Vert \varphi \Vert_{\H^1(\R^3_+)} \leq 1$. We shall estimate the two terms in the r.h.s of \eqref{weak:form:RATE:transport}.

$\bullet$ For the first term, we use the natural change of variable $x \mapsto \mathfrak{X}_\eps(0;t,x)$ and get 
\begin{align*}
\int_{\R^3_+} \theta_\eps^0(\mathfrak{X}_\eps(0;t,x) ) \varphi(x) \, \mathrm{d}x=& \int_{\mathfrak{X}_\eps(0;t,\R^3_+)} \theta_\eps^0(y) \varphi(\mathfrak{X}_\eps(t;0,y)) \, \mathrm{d}y 
\leq  \int_{\R^3_+} \theta_\eps^0(y) \varphi(\mathfrak{X}_\eps(t;0,y)) \, \mathrm{d}y.
\end{align*}
By the previous remark, we observe that $y \mapsto \varphi(\mathfrak{X}_\eps(t;0,y))$ vanishes on $\lbrace x_3=0 \rbrace$. It also belongs to $\H^1(\R^3_+)$. Indeed, since $\varphi
$ is compactly supported in $\R^3_+$, we have
\begin{align*}
  \int_{\R^3_+}  \vert \varphi(\mathfrak{X}_\eps(t;0,y)) \vert^2 \, \mathrm{d}y =\int_{\mathfrak{X}_\eps(t;0,\R^3_+)}  \vert \varphi(x) \vert^2 \, \mathrm{d}x = \int_{\mathfrak{X}_\eps(t;0,\R^3_+) \cap \R^3_+}  \vert \varphi(x) \vert^2 \, \mathrm{d}x \leq \Vert \varphi \Vert^2_{\Ld^2(\R^3_+)}.
\end{align*}
Furthermore, by \eqref{Dx-convrateH-1}, we have
\begin{align*}
\Vert \nabla_x [ \varphi(\mathfrak{X}_\eps(t;0))] \Vert_{\Ld^2(\R^3_+)}&=\Vert \mathrm{D}_x \mathfrak{X}_\eps(t;0) \nabla_x \varphi(\mathfrak{X}_\eps(t;0))\Vert_{\Ld^2(\R^3_+)} \\
&\leq \Vert \mathrm{D}_x \mathfrak{X}_\eps(t;0) \Vert_{\Ld^{\infty}(\R^3_+)} \Vert \nabla_x \varphi(\mathfrak{X}_\eps(t;0)) \Vert_{\Ld^2(\R^3_+)} \\
& \lesssim \Vert \nabla_x \varphi \Vert_{\Ld^2(\R^3_+)},
\end{align*}
as in the previous computation. This yields
\begin{align*}
\int_{\R^3_+} \theta_\eps^0(\mathfrak{X}_\eps(0;t,x) ) \varphi(x) \, \mathrm{d}x \leq \Vert \theta_\eps^0 \Vert_{\H^{-1}(\R^3_+)}.
\end{align*}


$\bullet$ For the second term, we write
\begin{align*}
\int_0^t \int_{\R^3_+}    H_{\eps}(s,x) \cdot \nabla_x [ \varphi(\mathfrak{X}_{\eps}(t;s,x))] \, \mathrm{d}x \, \mathrm{d}s &\leq \int_0^t \Vert  H_{\eps}(s)\Vert_{\Ld^2(\R^3_+)}  \Vert \nabla_x [\varphi(\mathfrak{X}_{\eps}(t;s))] \Vert_{\Ld^2(\R^3_+)} \, \mathrm{d}s \\
& \lesssim \Vert \nabla_x \varphi \Vert_{\Ld^2(\R^3_+)}\int_0^t \Vert  H_{\eps}(s)\Vert_{\Ld^2(\R^3_+)}  \, \mathrm{d}s,
\end{align*}
as above.
By Cauchy-Schwarz inequality, we get
\begin{align*}
\int_0^t \int_{\R^3_+}    H_{\eps}(s,x) \cdot \nabla_x [ \varphi(\mathfrak{X}_{\eps}(t;s,x))] \, \mathrm{d}x \, \mathrm{d}s &\leq\sqrt{T} \Vert  F_\eps+\rho_\eps e_3 \Vert_{\Ld^2(0,T;\Ld^2(\R^3_+))}+ \sqrt{T} \Vert  \rho(u_\eps-u) \Vert_{\Ld^2(0,t;\Ld^2(\R^3_+))} \\
&\leq\sqrt{T} \Vert  F_\eps+\rho_\eps e_3 \Vert_{\Ld^2(0,T;\Ld^2(\R^3_+))}+ \sqrt{T} \Vert  w_\eps \Vert_{\Ld^2(0,t;\Ld^2(\R^3_+))},
\end{align*}
since \begin{align*}
\Vert \rho \Vert_{\Ld^{\infty}(0,T;\Ld^{\infty}(\R^3_+))} \leq \underset{\eps \rightarrow 0}{ \liminf} \,\Vert \rho_\eps \Vert_{\Ld^{\infty}(0,T;\Ld^{\infty}(\R^3_+))} \lesssim 1,
\end{align*}
by lower semicontinuity of the weak$^{\ast}$ convergence and the uniform bound given in Corollary \eqref{coro:boundrho} which holds for all times. 

We obtain the conclusion of the lemma thanks to the identity \eqref{weak:form:RATE:transport}.
\end{proof}

We can now conclude the proof of Theorem \ref{th-rate:cvgence}. We add the estimate \eqref{ineqRate:preGronwall} to the estimate \eqref{estimateH-1} and obtain for all $t \in [0,T]$
\begin{align*}
\Vert w_\eps(t) \Vert_{\Ld^2(\R^3_+)}^2 +\Vert  \theta_\eps(t) \Vert_{\H^{-1}(\R^3_+)}^2  &\lesssim \Vert w_\eps^0 \Vert_{\Ld^2(\R^3_+)}^2+\Vert  \theta_\eps^0 \Vert_{\H^{-1}(\R^3_+)}^2  +(1+T) \Vert F_\eps+\rho_\eps e_3 \Vert_{\Ld^2(0,T;\Ld^2(\R^3_+))}^2\\
& \quad +   (1+T)\int_0^t \left(  \Vert  w_\eps(s) \Vert_{\Ld^2(\R^3_+)}^2+\Vert  \theta_\eps(s) \Vert_{\H^{-1}(\R^3_+)}^2   \right) \, \mathrm{d}s.
\end{align*}
Grönwall's lemma gives for all $t \in [0,T]$\begin{align*}
\Vert w_\eps(t) \Vert_{\Ld^2(\R^3_+)}^2 +\Vert  \theta_\eps(t) \Vert_{\H^{-1}(\R^3_+)}^2  &\lesssim e^{C_M(1+T)} \left( \Vert w_\eps^0 \Vert_{\Ld^2(\R^3_+)}^2+\Vert  \theta_\eps^0 \Vert_{\H^{-1}(\R^3_+)}^2+(1+T) \Vert F_\eps+\rho_\eps e_3 \Vert_{\Ld^2(0,T;\Ld^2(\R^3_+))}^2 \right).
\end{align*}
Invoking Remark \ref{rmq:rateCVGENCE}, we end up with the result claimed in Theorem \ref{th-rate:cvgence}.
\end{proof}

\appendix

\section{Gagliardo-Nirenberg-Sobolev inequality on the half-space}
\begin{theoreme}\label{gagliardo-nirenberg}
Let $(p,q,r) \in [1, + \infty]^3$ and $m \in \N$. Suppose $j \in \N$ and $\alpha \in [0,1]$ satisfy the relations
\begin{align*}
&\dfrac{1}{p}=\dfrac{j}{3}+\left( \dfrac{1}{r}-\dfrac{m}{3} \right)\alpha+\dfrac{1-\alpha}{q},\\
&\dfrac{j}{m} \leq \alpha \leq 1,
\end{align*}
with the exception $\alpha<1$ if $m-j-d/r \in \N$. 

Then for all $g \in \Ld^q(\R^3_+)$, if $\mathrm{D}^m g \in \Ld^r(\R^3_+)$, we have $\mathrm{D}^j g \in \Ld^p(\R^3_+)$ with the estimate 
$$ \Vert \mathrm{D}^j g  \Vert_{\Ld^p(\R^3_+)} \lesssim \Vert \mathrm{D}^m g \Vert_{\Ld^r(\R^3_+)} ^{\alpha} \Vert g \Vert_{\Ld^q(\R^3_+)}^{1-\alpha},$$
where $\lesssim$ refers to a universal constant.
\end{theoreme}

\section{Maximal $\Ld^s_t \Ld^q_x$ regularity for the Stokes system on the half-space}\label{AnnexeMaxregStokes}

%

The main result is the following and can be found with further references in \cite{Giga}. We refer to \eqref{domaineDqs} for the definition of the space $\mathrm{D}_q^{1-\frac{1}{s},s}$.

\begin{theoreme}
Consider $0 < T \leq \infty$ and $1 <q,s <\infty$. Then, for every $u_0 \in \mathrm{D}_q^{1-\frac{1}{s},s}(\R^3_+)$ which is divergence free and $f \in \Ld^s(0,T;\Ld^q_{\mathrm{div}}(\R^3_+))$, there exists a unique solution $u$ of the Stokes system
\begin{align*}
\left\{
      \begin{aligned}
      \partial_t u +A_q u &=f, \\ 
u_{\mid x_3=0 }&=0,  \\
u(0,x)&=u_0(x),
      \end{aligned}
    \right.
\end{align*}
satisfying 
\begin{align*}
& u \in \Ld^s(0,T';D(A_q)) \ \text{for all finite } T' \leq T,
\end{align*}
and
\begin{align*}
\Vert \partial_t u \Vert_{\Ld^s(0,T;\Ld^q(\R^3_+))}  + \Vert \mathrm{D}^2_x u\Vert_{\Ld^s(0,T;\Ld^q(\R^3_+))}\leq C\left(\Vert u_0 \Vert_{\mathrm{D}_q^{1-\frac{1}{s},s}(\R^3_+)} + \Vert f \Vert_{\Ld^s(0,T;\Ld^q(\R^3_+))} \right),
\end{align*}
where $C=C(q,s)>0$.

Furthermore, if $u_0  \in \W^{1,q}_0(\R^3_+)\cap \Ld^q_{\mathrm{div}}(\R^3_+)$ and if $s \in (1,2]$, the statement holds and we can replace $\Vert u_0 \Vert_{\mathrm{D}_q^{1-1/s,s}(\R^3_+)}$ by $\Vert  u_0 \Vert_{\W^{1,q}_0(\R^3_+)}$ on the right-hand side of the previous inequality.
\end{theoreme}

\section{Proof of Lemmas \ref{LM:F0a}--\ref{LM:F1a}--\ref{LM:F2a}}\label{AnnexeRemainsProofLemmas}
This appendix is devoted to the proof of Lemmas \ref{LM:F0a}--\ref{LM:F1a}--\ref{LM:F2a}, which provide decay estimates of the Brinkman force in $\Ld^p_t \Ld^p_x$.
Recall that we work under the polynomial decay in time of $u_\eps$ given by Theorem \ref{cond:decay}.

\begin{proof}[Proof of Lemma \ref{LM:F0a}]
We only focus on the first estimate of Lemma \ref{LM:F0a}. By the triangular inequality, we have for all $t \in (T_0,T)$
\begin{align*}
\Vert F_\eps^{\natural,0}(t)\Vert_{\Ld^2(\R^3_+)} 
&\leq \frac{e^{-\frac{t}{\eps}}}{(1+R(t))^{k}}  \Bigg[\int_{\R^3_+} \Bigg\vert \int_{\R^3} \mathbf{1}_{\mathcal{O}^t_{\eps}}(x,[\Gamma_\eps^{t,x}]^{-1}(w))  \vert w \vert^{k_1} f_\eps^0(\widetilde{\mathrm{X}}_{\eps}^{0;t}(x,w),w) \\
& \qquad \quad  \qquad \qquad \qquad \qquad \qquad \qquad \times \Big\vert w+e_3-(P u_\eps)(0,\widetilde{\mathrm{X}}_{\eps}^{0;t}(x,w))\Big\vert \, \mathrm{d}w \Bigg\vert^r \, \mathrm{d}x \Bigg]^{\frac{1}{r}} \\[2mm]
& \quad + \frac{1}{(1+R(t))^{k_2}}\left[\int_{\R^3_+} \left( \int_{\R^3} \mathbf{1}_{\mathcal{O}^t_{\eps}}(x,[\Gamma_\eps^{t,x}]^{-1}(w)) \vert w \vert^{k_2}  f_\eps^0(\widetilde{\mathrm{X}}_{\eps}^{0;t}(x,w),w) \, \mathrm{d}w \right)^r \, \mathrm{d}x \right]^{\frac{1}{r}}.
\end{align*}
We apply Hölder inequality in velocity for the first term, as in the proof of Lemma \ref{LM:F0-pointL2}, and we obtain
\begin{align*}
&\int_{\R^3_+} \left( \int_{\R^3} \mathbf{1}_{\mathcal{O}^t_{\eps}}(x,[\Gamma_\eps^{t,x}]^{-1}(w))  \vert w \vert^{k_1} f_\eps^0(\widetilde{\mathrm{X}}_{\eps}^{0;t}(x,w),w)\Big\vert w+e_3-(P u_\eps)(0,\widetilde{\mathrm{X}}_{\eps}^{0;t}(x,w))\Big\vert \, \mathrm{d}w \right)^r \, \mathrm{d}x \\
& \leq  \int_{\R^3_+}  \left( \int_{\R^3} \mathbf{1}_{\mathcal{O}^t_{\eps}}(x,[\Gamma_\eps^{t,x}]^{-1}(w)) \vert w \vert^{k_1} f_\eps^0(\widetilde{\mathrm{X}}_{\eps}^{0;t}(x,w),w) \, \mathrm{d}w  \right)^{r-1}  \\
& \qquad \qquad \times \left( \int_{\R^3} \mathbf{1}_{\mathcal{O}^t_{\eps}}(x,[\Gamma_\eps^{t,x}]^{-1}(w)) \vert w \vert^{k_1} f_\eps^0(\widetilde{\mathrm{X}}_{\eps}^{0;t}(x,w),w) \Big\vert w+e_3-(P u_\eps)(0,\widetilde{\mathrm{X}}_{\eps}^{0;t}(x,w))\Big\vert^r  \, \mathrm{d}w  \right)  \, \mathrm{d}x \\
& \leq \Vert \vert v \vert^{k_1} f_\eps^0 \Vert_{\Ld^1(\R^3;\Ld^{\infty}(\R^3_+))}^{r-1}\left[\Vert (1+\vert v \vert^{k_1+r}) f_\eps^0 \Vert_{\Ld^1(\R^3_+ \times \R^3)}+ \Vert \vert v \vert^{k_1} f_\eps^0 \Vert_{\Ld^1(\R^3;\Ld^{\infty}(\R^3_+))} \Vert u_\eps^0 \Vert_{\Ld^r(\R^3_+)}^r \right],
\end{align*}
where we have followed the rest of the proof of Lemma \ref{LM:F0-pointL2}. For the second term, we apply the generalized Minkowski inequality (see e.g. \cite{HLP}) and get
\begin{align*}
&\left[\int_{\R^3_+} \left( \int_{\R^3} \mathbf{1}_{\mathcal{O}^t_{\eps}}(x,[\Gamma_\eps^{t,x}]^{-1}(w)) \vert w \vert^{k_2}  f_\eps^0(\widetilde{\mathrm{X}}_{\eps}^{0;t}(x,w),w) \, \mathrm{d}w \right)^r \, \mathrm{d}x \right]^{\frac{1}{r}} \\
&\leq \int_{\R^3} \vert w \vert^{k_2}  \left( \int_{\R^3_+} \mathbf{1}_{\mathcal{O}^t_{\eps}}(x,[\Gamma_\eps^{t,x}]^{-1}(w))  f_\eps^0(\widetilde{\mathrm{X}}_{\eps}^{0;t}(x,w),w)^r \, \mathrm{d}x \right)^{\frac{1}{r}} \, \mathrm{d}w \lesssim \int_{\R^3} \vert w \vert^{k_2}  \Vert f_\eps^0(\cdot, w) \Vert_{\Ld^r(\R^3_+)} \, \mathrm{d}w.
\end{align*}
All in all, we obtain
\begin{align*}
\int_{T_0}^T \Vert (1+t)^k F_\eps^{\natural,0}(t)\Vert_{\Ld^r(\R^3_+)}^r \, \mathrm{d}t &\lesssim \int_0^T \frac{e^{-r\frac{t}{\eps}}}{(1+t)^{r(k_1-k)}} \Vert \vert v \vert^{k_1} f_\eps^0 \Vert_{\Ld^1(\R^3;\Ld^{\infty}(\R^3_+))}^{r-1}\Big[\Vert (1+\vert v \vert^{k_1+r}) f_\eps^0 \Vert_{\Ld^1(\R^3_+ \times \R^3)} \\
& \qquad \qquad \qquad \qquad  \qquad \qquad + \Vert \vert v \vert^{k_1} f_\eps^0 \Vert_{\Ld^1(\R^3;\Ld^{\infty}(\R^3_+))} \Vert u_\eps^0 \Vert_{\Ld^p(\R^3_+)}^r \Big]  \, \mathrm{d}t  \\
& \quad + \int_0^T \frac{1}{(1+t)^{r(k_2-k)}} \Vert \vert v \vert^{k_2}   f_\eps^0 \Vert_{\Ld^1(\R^3;\Ld^r(\R^3_+))}  \, \mathrm{d}t \\
&\lesssim \eps \Vert \vert v \vert^{k_1} f_\eps^0 \Vert_{\Ld^1(\R^3;\Ld^{\infty}(\R^3_+))}^{r-1}\Big[\Vert (1+\vert v \vert^{k_1+r}) f_\eps^0 \Vert_{\Ld^1(\R^3_+ \times \R^3)} \\
& \qquad \qquad \qquad \qquad  \qquad \qquad + \Vert \vert v \vert^{k_1} f_\eps^0 \Vert_{\Ld^1(\R^3;\Ld^{\infty}(\R^3_+))} \Vert u_\eps^0 \Vert_{\Ld^p(\R^3_+)}^r \Big] \\
& \quad + \Vert \vert v \vert^{k_2}   f_\eps^0 \Vert_{\Ld^1(\R^3;\Ld^r(\R^3_+))},
\end{align*}
where the last inequality $\lesssim$ is independent of $T$, thanks to the choice of $k_1$ and $k_2$. This concludes the proof of the lemma.
\end{proof}

\begin{proof}[Proof of Lemma \ref{LM:F1a}]
As in the proof of Lemma \ref{LM:F0a}, we only write the proof for the first estimate. Following the proof of Lemma \ref{LM:F1-pointL2}, we get for any $t \in (T_0,T)$
\begin{align*}
 \int_{\R^3_+} \vert F_\eps^{\natural,1}(t,x) \vert^r \, \mathrm{d}x 
& \lesssim \frac{\eps^{r-1} \Vert \vert w \vert^k f_\eps^0 \Vert_{\Ld^1(\R^3; \Ld^{\infty}(\R^3_+))}^{r}}{(1+t)^{kr}}  \int_0^t  e^{\frac{\tau-t}{\eps}} \Vert \partial_\tau  u_\eps(\tau) \Vert_{\Ld^r(\R^3)}^r \, \mathrm{d}\tau.
\end{align*}
We then use Fubini theorem to get
\begin{align*}
 \Vert (1+t)^k F_\eps^{\natural,1}(t)\Vert_{\Ld^r(T_0,T;\Ld^r(\R^3_+))}^r & \lesssim \eps^{r-1} \Vert \vert v \vert^k f_\eps^0 \Vert_{\Ld^1(\R^3; \Ld^{\infty}(\R^3_+))}^{r} \int_{T_0}^T \int_0^t  e^{\frac{\tau-t}{\eps}} \Vert \partial_\tau  u_\eps(\tau) \Vert_{\Ld^r(\R^3)}^r \, \mathrm{d}\tau \, \mathrm{d}t \\
 &\lesssim \eps^{r-1} \Vert \vert v \vert^k f_\eps^0 \Vert_{\Ld^1(\R^3; \Ld^{\infty}(\R^3_+))}^{r}\int_0^T \Vert \partial_\tau  u_\eps(\tau) \Vert_{\Ld^r(\R^3)}^r  \left( \int_\tau^T e^{\frac{\tau-t}{\eps}} \, \mathrm{d}t \right) \, \mathrm{d}\tau \\
 & \lesssim \eps^{r} \Vert \vert v \vert^k f_\eps^0 \Vert_{\Ld^1(\R^3; \Ld^{\infty}(\R^3_+))}^{r} \Vert \partial_t  u_\eps \Vert_{\Ld^r(0,T;\Ld^r(\R^3_+))}^r,
\end{align*}
and this concludes the proof.
\end{proof}

\begin{proof}[Proof of Lemma \ref{LM:F2a}]
First, we derive estimates which hold without using \eqref{assumption:DECAYSOURCE}.  We have for all $t \in (T_0,T)$
\begin{align*}
\Vert (1+t)^k F_\eps^{\natural,2}(t)\Vert_{\Ld^r(\R^3_+)}  \lesssim \Vert \mathrm{I}^{\natural}(t) \Vert_{\Ld^r(\R^3_+)}+ \Vert \mathrm{II}^{\natural}(t)\Vert_{\Ld^r(\R^3_+)}+ \Vert\mathrm{III}^{\natural}(t)\Vert_{\Ld^r(\R^3_+)},
\end{align*}
where we refer to the proof of Lemma \ref{LM:F2-pointL2} for the definition of the previous terms.
We have
\begin{align*}
\vert \mathrm{I}^{\natural}(t,x) \vert^r &\lesssim \eps^{r-1} \Vert (1+ \vert v \vert^{\frac{r}{r-1}}) \vert v \vert^k f_\eps^0 \Vert_{\Ld^1(\R^3; \Ld^{\infty}(\R^3_+))}^{r-1} \\ 
& \qquad \qquad   \times \int_0^t e^{\frac{r(s-t)}{\eps}}  \int_{\R^3} \mathbf{1}_{\mathcal{O}^t_{\eps}}(x,[\Gamma_\eps^{t,x}]^{-1}(w)) \vert w \vert^kf_\eps^0(\widetilde{\mathrm{X}}_{\eps}^{0;t}(x,w),w) \Big\vert    \nabla_x  [P u_\eps](s,\widetilde{\mathrm{X}}_\eps^{s;t}(x,w)) \Big\vert^r \, \mathrm{d}w \, \mathrm{d}s \\
\end{align*}
thanks to Jensen inequality, so that
\begin{align*}
\int_{T_0}^T \int_{\R^3_+} \vert \mathrm{I}^{\natural}(t,x) \vert^r \, \mathrm{d}x \,\mathrm{d}t & \lesssim  \eps^{r-1} \Vert (1+ \vert v \vert^{\frac{r}{r-1}}) \vert v \vert^k f_\eps^0 \Vert_{\Ld^1(\R^3; \Ld^{\infty}(\R^3_+))}^{r-1}  \Vert \vert v \vert^k f_\eps^0 \Vert_{\Ld^1(\R^3; \Ld^{\infty}(\R^3_+))}   \\
& \qquad \qquad \qquad \qquad \qquad \qquad  \times \int_{T_0}^T \int_0^t    e^{\frac{r(s-t)}{\eps}} \Vert   \nabla_x  [Pu_\eps](s) \Vert_{\Ld^r(\R^3)}^{r} \, \mathrm{d}s \, \mathrm{d}t \\
& \lesssim  \eps^{r} \Vert (1+ \vert v \vert^{\frac{r}{r-1}}) \vert v \vert^k f_\eps^0 \Vert_{\Ld^1(\R^3; \Ld^{\infty}(\R^3_+))}^{r-1}  \Vert \vert v \vert^k f_\eps^0 \Vert_{\Ld^1(\R^3; \Ld^{\infty}(\R^3_+))}   \\
& \qquad \qquad \qquad \qquad \qquad \qquad  \times \int_{0}^T  \Vert   \nabla_x  u_\eps(s) \Vert_{\Ld^r(\R^3_+)}^{r} \, \mathrm{d}s.
\end{align*}
For $\mathrm{II}^{\natural}$, we obtain in the same fashion
\begin{align*}
\vert \mathrm{II}^{\natural}(t,x) \vert^r
& \leq \Vert  \vert v \vert ^k f_\eps^0 \Vert_{\Ld^1(\R^3; \Ld^{\infty}(\R^3_+))}^{r-1} \\
& \qquad  \times \Bigg[\int_0^t e^{\frac{r(s-t)}{\eps}} \left( \int_{\R^3}\mathbf{1}_{\mathcal{O}^t_{\eps}}(x,[\Gamma_\eps^{t,x}]^{-1}(w))  \vert w \vert ^k f_\eps^0(\widetilde{\mathrm{X}}_{\eps}^{0;t}(x,w),w) \Big\vert    \nabla_x  [P u_\eps](s,\widetilde{\mathrm{X}}_\eps^{s;t}(x,w)) \Big\vert^r \, \mathrm{d}w \right)^{\frac{1}{r}} \, \mathrm{d}s  \Bigg]^r \\
&\lesssim \eps^{r-1} \Vert   f_\eps^0 \Vert_{\Ld^1(\R^3; \Ld^{\infty}(\R^3_+))}^{r-1} \\
& \qquad  \times \int_0^t e^{\frac{r(s-t)}{\eps}} \left( \int_{\R^3}\mathbf{1}_{\mathcal{O}^t_{\eps}}(x,[\Gamma_\eps^{t,x}]^{-1}(w))  \vert w \vert ^k f_\eps^0(\widetilde{\mathrm{X}}_{\eps}^{0;t}(x,w),w) \Big\vert    \nabla_x  [P u_\eps](s,\widetilde{\mathrm{X}}_\eps^{s;t}(x,w)) \Big\vert^r \, \mathrm{d}w \right) \, \mathrm{d}s,
\end{align*}
therefore
\begin{align*}
\int_{T_0}^T \int_{\R^3_+} \vert \mathrm{II}^{\natural}(t,x) \vert^r \, \mathrm{d}x \, \mathrm{d}t & \lesssim  \eps^{r-1} \Vert   f_\eps^0 \Vert_{\Ld^1(\R^3; \Ld^{\infty}(\R^3_+))}^{r-1} \int_{T_0}^T \int_0^t    e^{\frac{r(s-t)}{\eps}} \Vert   \nabla_x  [Pu_\eps](s)  \Vert_{\Ld^r(\R^3)}^{r} \, \mathrm{d}s \, \mathrm{d}t \\
&  \lesssim  \eps^{r} \Vert   f_\eps^0 \Vert_{\Ld^1(\R^3; \Ld^{\infty}(\R^3_+))}^{r-1} \int_{0}^T   \Vert   \nabla_x  u_\eps(s)  \Vert_{\Ld^r(\R^3_+)}^{r} \, \mathrm{d}s.
\end{align*}
For the third term $\mathrm{III}^{\natural}$, we have as in the proof of Lemma \ref{LM:F2-pointL2}
\begin{align*}
\int_{T_0}^T \int_{\R^3_+} \vert \mathrm{III}^{\natural}(t,x) \vert^r \, \mathrm{d}x \, \mathrm{d}t &\lesssim \eps^{-1} \Vert \vert v \vert^k f_\eps^0\Vert^{r-1}_{\Ld^1(\R^3; \Ld^{\infty}(\R^3_+))} \left[ \int_0^T \left( \int_0^t  e^{\frac{r(s-t)}{2(r-1)\eps}}  \Vert \nabla_x u_\eps(s) \Vert_{\Ld^{\infty}(\R^3_+)}^{\frac{r}{r-1}} \, \mathrm{d}s \right)^{r-1} \, \mathrm{d}t \right] \\
&  \times \underset{t \in (0,T)}{\sup} \Bigg\{ \int_0^te^{\frac{r(s-t)}{2\eps}}  \int_0^s  e^{\frac{r(\tau-s)}{2\eps}}  \\
&  \times \int_{\R^3_+ \times \R^3} \mathbf{1}_{\mathcal{O}^t_{\eps}}(x,[\Gamma_\eps^{t,x}]^{-1}(w)) \vert w \vert^k f_\eps^0(\widetilde{\mathrm{X}}_{\eps}^{0;t}(x,w),w)      \left\vert (Pu_{\eps})(\tau,\widetilde{\mathrm{X}}_{\eps}^{\tau;t}(x,w)) \right\vert^r \, \mathrm{d}x \, \mathrm{d}w  \, \mathrm{d}\tau  \, \mathrm{d}s  \Bigg\} .
\end{align*}
For the term between braces, we perform the change of variable $x'=\widetilde{\mathrm{X}}_{\eps}^{\tau;t}(x,w)$ and we get
\begin{align*}
& \int_0^te^{\frac{r(s-t)}{2\eps}}  \int_0^s  e^{\frac{r(\tau-s)}{2\eps}}  \int_{\R^3_+ \times \R^3} \mathbf{1}_{\mathcal{O}^t_{\eps}}(x,[\Gamma_\eps^{t,x}]^{-1}(w)) \vert w \vert^k  f_\eps^0(\widetilde{\mathrm{X}}_{\eps}^{0;t}(x,w),w)      \left\vert (Pu_{\eps})(\tau,\widetilde{\mathrm{X}}_{\eps}^{\tau;t}(x,w)) \right\vert^r \, \mathrm{d}x \, \mathrm{d}w  \, \mathrm{d}\tau \, \mathrm{d}s \\
 & \lesssim \Vert \vert w \vert^k f_\eps^0 \Vert_{\Ld^1(\R^3 ; \Ld^{\infty}(\R^3_+)) }\int_0^te^{\frac{r(s-t)}{2\eps}}  \int_0^s  e^{\frac{r(\tau-s)}{2\eps}} \Vert u_\eps(\tau) \Vert_{\Ld^p(\R^3_+)}^r \, \mathrm{d}\tau \, \mathrm{d}s \\
 & \lesssim \eps^2 \Vert \vert v \vert^kf_\eps^0 \Vert_{\Ld^1(\R^3 ; \Ld^{\infty}(\R^3_+)) } (\Upsilon_\eps^0)^r,
\end{align*}
where we have used the fact that $T$ is a strong existence time so that $\Vert u_\eps \Vert_{\Ld^{\infty}(0,T;\Ld^p(\R^3_+))} \lesssim \Upsilon_\eps^0$ for all $ p \in [2,6]$ (see Corollary \eqref{coro:estimatesSTRONG}). For the term in brackets, we write
\begin{align*}
& \int_0^T  \left( \int_0^t  e^{\frac{r(s-t)}{2(r-1)\eps}}  \Vert \nabla_x u_\eps(s) \Vert_{\Ld^{\infty}(\R^3_+)}^{\frac{r}{r-1}} \ \mathrm{d}s \right)^{r-1} \, \mathrm{d}t \\
&=\int_0^T \left( \int_0^t    \frac{2(r-1)\eps}{r}(1-e^{\frac{-rt}{2(r-1)\eps}})\Vert \nabla_x u_\eps(s) \Vert_{\Ld^{\infty}(\R^3_+)}^{\frac{r}{r-1}} \frac{e^{\frac{r(s-t)}{2(r-1)\eps}} \, \mathrm{d}s}{\frac{2(r-1)\eps}{r}(1-e^{\frac{-rt}{2(r-1)\eps}})}  \right)^{r-1} \, \mathrm{d}t \\
&\lesssim\eps^{r-2} \int_0^T  \int_0^t   \Vert \nabla_x u_\eps(s) \Vert_{\Ld^{\infty}(\R^3_+)}^{r} e^{\frac{r(s-t)}{2\eps}} \, \mathrm{d}s  \, \mathrm{d}t = \eps^{r-2} \int_0^T \Vert \nabla_x u_\eps(s) \Vert_{\Ld^{\infty}(\R^3_+)}^{r}  \left(  \int_s^T    e^{\frac{r(s-t)}{2\eps}}  \, \mathrm{d}t \right) \, \mathrm{d}s,
\end{align*}
where we have used the Jensen inequality in the probability space
\begin{align*}
\left( (0,t), \  \frac{e^{\frac{r(s-t)}{2(r-1)\eps}} \, \mathrm{d}s}{\frac{2(r-1)\eps}{r}(1-e^{\frac{-rt}{2(r-1)\eps}})} \right).
\end{align*}
This yields
\begin{align*}
\int_{T_0}^T \int_{\R^3_+} \vert \mathrm{III}^{\natural}(t,x) \vert^r \, \mathrm{d}x \, \mathrm{d}t &\lesssim \eps^{r} \Vert \vert v \vert^k  f_\eps^0 \Vert_{\Ld^1(\R^3 ; \Ld^{\infty}(\R^3_+)) }^r (\Upsilon_\eps^0)^r \int_0^T \Vert \nabla_x u_\eps(s) \Vert_{\Ld^{\infty}(\R^3_+)}^{r}   \, \mathrm{d}s.
\end{align*}

The proof then goes by taking advantage of the hypothesis \eqref{assumption:DECAYSOURCE}. We shall make a constant use of the Gagliardo-Nirenberg-Sobolev inequality (see Theorem \ref{gagliardo-nirenberg} in the Appendix)
\begin{align*}
\Vert \nabla_x u_\eps (s)\Vert_{\Ld^r(\R^3_+)} &\lesssim\Vert \mathrm{D}^2_x u_\eps (s)\Vert_{\Ld^r(\R^3_+)}^{\alpha_r} \Vert  u_\eps(s) \Vert_{\Ld^2(\R^3_+)}^{1-\alpha_r}, \ \  r \geq 2, \ \  \alpha_r \in (0,1), \\
\Vert \nabla_x u_\eps (s)\Vert_{\Ld^{\infty}(\R^3_+)} &\lesssim \Vert \mathrm{D}^2_x u_\eps (s)\Vert_{\Ld^r(\R^3_+)}^{\beta_r} \Vert  u_\eps(s) \Vert_{\Ld^2(\R^3_+)}^{1-\beta_r}, \ \ r>3,  \ \ \beta_r \in(0,1).
\end{align*}
Thanks to \eqref{assumption:DECAYSOURCE}, we can use the conditional Theorem \ref{cond:decay} on the interval $[0,T]$, which means that for all $s \in [0,T]$
\begin{align*}
 \ \ \Vert  u_\eps(s) \Vert_{\Ld^2(\R^3_+)} \lesssim \frac{\Psi \left (\Vert  u_{\eps}^0 \Vert_{\Ld^{1}\cap\Ld^2(\R^3_+)}^2+1\right)}{(1+s)^{\frac{3}{4}}}.
\end{align*}
We then get
\begin{align*}
\int_{T_0}^T \int_{\R^3_+} \vert \mathrm{I}^{\natural}(t,x) \vert^r \, \mathrm{d}x \,\mathrm{d}t & \lesssim \eps^{r} \Vert (1+ \vert v \vert^{\frac{r}{r-1}}) \vert v \vert^k f_\eps^0 \Vert_{\Ld^1(\R^3; \Ld^{\infty}(\R^3_+))}^{r-1}  \Vert \vert v \vert^k f_\eps^0 \Vert_{\Ld^1(\R^3; \Ld^{\infty}(\R^3_+))}   \\
& \qquad \qquad \qquad \qquad \qquad \qquad  \times \int_{0}^T \Vert \mathrm{D}^2_x u_\eps (s)\Vert_{\Ld^r(\R^3_+)}^{r\alpha_r} \Vert  u_\eps(s) \Vert_{\Ld^2(\R^3_+)}^{r(1-\alpha_r)} \, \mathrm{d}s \\
& \lesssim  \eps^{r} \Vert (1+ \vert v \vert^{\frac{r}{r-1}}) \vert v \vert^k f_\eps^0 \Vert_{\Ld^1(\R^3; \Ld^{\infty}(\R^3_+))}^{r-1}  \Vert \vert v \vert^k f_\eps^0 \Vert_{\Ld^1(\R^3; \Ld^{\infty}(\R^3_+))}     \\
& \qquad    \times \Psi \left (\Vert  u_{\eps}^0 \Vert_{\Ld^{1}\cap\Ld^2(\R^3_+)}^2+M \right)^{\frac{r(1-\alpha_r)}{2}} \int_{0}^T   \Vert   \mathrm{D}^2_x u_\eps(s) \Vert_{\Ld^r(\R^3_+)}^{r\alpha_r} \frac{1}{(1+s)^{r(1-\alpha_r)\frac{3}{4}}} \, \mathrm{d}s  \\
& \lesssim \eps^{r} \Vert (1+ \vert v \vert^{\frac{r}{r-1}}) \vert v \vert^k f_\eps^0 \Vert_{\Ld^1(\R^3; \Ld^{\infty}(\R^3_+))}^{r-1}  \Vert \vert v \vert^k f_\eps^0 \Vert_{\Ld^1(\R^3; \Ld^{\infty}(\R^3_+))}    \\
& \qquad   \times \Psi \left (\Vert  u_{\eps}^0 \Vert_{\Ld^{1}\cap\Ld^2(\R^3_+)}^2+M \right)^{\frac{r(1-\alpha_r)}{2}}\Vert \mathrm{D}^2_x u_\eps\Vert_{\Ld^r((0,T) \times \R^3_+)}^{r \alpha_r} \left( \int_0^T   \frac{1}{(1+s)^{r\frac{3}{4}}}  \, \mathrm{d}s \right)^{1-\alpha_r},
\end{align*}
thanks to Hölder inequality. Since $r \geq 3 \geq \frac{4}{3}^+$, we have $r\frac{3}{4}>1$ and the last integral in time is uniformly bounded in $T$ therefore, by Young inequality, we have
\begin{align*}
 \int_{T_0}^T \int_{\R^3_+} \vert \mathrm{I}^{\natural}(t,x) \vert^r \, \mathrm{d}x \,\mathrm{d}t  &\lesssim \eps^{r} \Vert (1+ \vert v \vert^{\frac{r}{r-1}}) \vert v \vert^k f_\eps^0 \Vert_{\Ld^1(\R^3; \Ld^{\infty}(\R^3_+))}^{\frac{r-1}{1-\alpha_r}}  \Vert \vert v \vert^k f_\eps^0 \Vert_{\Ld^1(\R^3; \Ld^{\infty}(\R^3_+))}^{\frac{1}{1-\alpha_r}} \Psi \left (\Vert  u_{\eps}^0 \Vert_{\Ld^{1}\cap\Ld^2(\R^3_+)}^2+M \right)^{\frac{r}{2}}   \\
& \quad + \eps^r \Vert \mathrm{D}^2_x u_\eps\Vert_{\Ld^r((0,T) \times \R^3_+)}^{r}.
\end{align*}
The same procedure essentially leads to
\begin{align*}
\int_{T_0}^T \int_{\R^3_+} \vert \mathrm{II}^{\natural}(t,x) \vert^r  \, \mathrm{d}x \, \mathrm{d}t & \lesssim \eps^{r} \Vert \vert v \vert^k f_\eps^0 \Vert_{\Ld^1(\R^3; \Ld^{\infty}(\R^3_+))}^{\frac{r}{1-\alpha_r}}  \Psi \left (\Vert  u_{\eps}^0 \Vert_{\Ld^{1}\cap\Ld^2(\R^3_+)}^2+M \right)^{\frac{r}{2}} + \eps^r \Vert \mathrm{D}^2_x u_\eps\Vert_{\Ld^r((0,T) \times \R^3_+)}^{r}.
\end{align*}
For the last term, we have
\begin{align*}
\int_{T_0}^T \int_{\R^3_+} \vert \mathrm{III}^{\natural}(t,x) \vert^r  \, \mathrm{d}x \, \mathrm{d}t 
&\lesssim \eps^{r} \Vert \vert v \vert^k  f_\eps^0 \Vert_{\Ld^1(\R^3 ; \Ld^{\infty}(\R^3_+)) }^r (\Upsilon_\eps^0)^r \int_0^T \Vert \mathrm{D}^2_x u_\eps (s)\Vert_{\Ld^r(\R^3_+)}^{r\beta_r} \Vert  u_\eps(s) \Vert_{\Ld^2(\R^3_+)}^{r(1-\beta_r)}   \, \mathrm{d}s \\
&\lesssim \eps^{r} \Vert \vert v \vert^k  f_\eps^0 \Vert_{\Ld^1(\R^3 ; \Ld^{\infty}(\R^3_+)) }^r (\Upsilon_\eps^0)^r \\
& \quad    \times \Psi \left (\Vert  u_{\eps}^0 \Vert_{\Ld^{1}\cap\Ld^2(\R^3_+)}^2+M \right)^{\frac{r(1-\beta_r)}{2}} \int_{0}^T   \Vert   \mathrm{D}^2_x u_\eps(s) \Vert_{\Ld^r(\R^3_+)}^{r\beta_r} \frac{1}{(1+s)^{r(1-\beta_r)\frac{3}{4}}} \, \mathrm{d}s  \\
&\lesssim \eps^{r} \Vert \vert v \vert^k  f_\eps^0 \Vert_{\Ld^1(\R^3 ; \Ld^{\infty}(\R^3_+)) }^r (\Upsilon_\eps^0)^r \\
&\quad   \times \Psi \left (\Vert  u_{\eps}^0 \Vert_{\Ld^{1}\cap\Ld^2(\R^3_+)}^2+M \right)^{\frac{r(1-\beta_r)}{2}}\Vert \mathrm{D}^2_x u_\eps\Vert_{\Ld^r((0,T) \times \R^3_+)}^{r \beta_r} \left( \int_0^T   \frac{1}{(1+s)^{r\frac{3}{4}}}  \, \mathrm{d}s \right)^{1-\beta_r}, 
\end{align*}
as in the previous estimates. We finally obtain
\begin{align*}
\int_{T_0}^T \int_{\R^3_+} \vert \mathrm{III}^{\natural}(t,x) \vert^r  \, \mathrm{d}x \, \mathrm{d}t   \lesssim \eps^r(\Upsilon_\eps^0)^{\frac{p}{1-\beta_r}} \Vert \vert v \vert f_\eps^0\Vert^{\frac{r}{1-\beta_r}}_{\Ld^1(\R^3; \Ld^{\infty}(\R^3_+))} \Psi \left (\Vert  u_{\eps}^0 \Vert_{\Ld^{1}\cap\Ld^2(\R^3_+)}^2+M \right)^{\frac{r}{2}}  +\eps^r  \Vert \mathrm{D}^2_x u_\eps \Vert_{\Ld^r((0,T) \times \R^3_+)}^{r}.
\end{align*}
We deduce the result by gathering all the terms together.
\end{proof}

\section*{Acknowledgements}
I am very grateful to my PhD advisors Daniel Han-Kwan and Ayman Moussa for bringing this problem to my attention, as well as for their constant support, advice and guidance on my work. I would also like to thank Richard Höfer for interesting discussions about the Vlasov-Navier-Stokes system and in particular for suggesting me the introduction of the potential energy. 

\bibliographystyle{abbrv}


\bibliography{biblio}

\begin{thebibliography}{10}

\bibitem{AH}
H.~Abidi and T.~Hmidi.
\newblock On the global well-posedness for {B}oussinesq system.
\newblock {\em {J. Differ. Equations}}, 233(1):199--220, 2007.

\bibitem{BD}
C.~{Bardos} and P.~{Degond}.
\newblock {Global existence for the {V}lasov-{P}oisson equation in 3 space
  variables with small initial data}.
\newblock {\em {Ann. Inst. Henri Poincar\'e, Anal. Non Lin\'eaire}},
  2:101--118, 1985.

\bibitem{BLR}
C.~{Bardos}, G.~{Lebeau}, and J.~{Rauch}.
\newblock {Sharp sufficient conditions for the observation, control, and
  stabilization of waves from the boundary}.
\newblock {\em {SIAM J. Control Optim.}}, 30(5):1024--1065, 1992.

\bibitem{BDM}
S.~Benjelloun, L.~Desvillettes, and A.~Moussa.
\newblock Existence theory for the kinetic-fluid coupling when small droplets
  are treated as part of the fluid.
\newblock {\em {J. Hyperbolic Differ. Equ.}}, 11(01):109--133, 2014.

\bibitem{BDGR1}
E.~Bernard, L.~Desvillettes, F.~Golse, and V.~Ricci.
\newblock A derivation of the {V}lasov-{N}avier-{S}tokes model for aerosol
  flows from kinetic theory.
\newblock {\em Commun. Math. Sci.}, 15(6):1703--1741, 2017.

\bibitem{BDGR2}
E.~Bernard, L.~Desvillettes, F.~Golse, and V.~Ricci.
\newblock A derivation of the {V}lasov-{S}tokes system for aerosol flows from
  the kinetic theory of binary gas mixtures.
\newblock {\em Kinet. Relat. Models}, 11(1):43--69, 2018.

\bibitem{BlauFil}
A.~Blaustein and F.~Filbet.
\newblock Concentration phenomena in {F}itz{H}ugh-{N}agumo's equations: A
  mesoscopic approach.
\newblock {\em arXiv preprint arXiv:2201.02363}, 2022.

\bibitem{BM}
W.~Borchers and T.~Miyakawa.
\newblock {L}2 decay for the {N}avier-{S}tokes flow in halfspaces.
\newblock {\em {Math. Ann.}}, 282(1):139--155, 1988.

\bibitem{BrandoCharaf}
L.~Brandolese and C.~Mouzouni.
\newblock A short proof of the large time energy growth for the {Boussinesq}
  system.
\newblock {\em J. Nonlinear Sci.}, 27(5):1589--1608, 2017.

\bibitem{BrandoSchon}
L.~Brandolese and M.~E. Schonbek.
\newblock Large time decay and growth for solutions of a viscous {Boussinesq}
  system.
\newblock {\em Trans. Am. Math. Soc.}, 364(10):5057--5090, 2012.

\bibitem{BreVP1}
Y.~{Brenier}.
\newblock {A {Vlasov}-{Poisson} formulation of the {E}uler equations for
  perfect incompressible fluids}.
\newblock {\em {Rapport de recherche INRIA}}, 1989.

\bibitem{BreVP2}
Y.~{Brenier}.
\newblock {Convergence of the {V}lasov-{P}oisson system to the incompressible
  {E}uler equations}.
\newblock {\em {Commun. Partial Differ. Equations}}, 25(3-4):737--754, 2000.

\bibitem{CH}
K.~Carrapatoso and M.~Hillairet.
\newblock On the derivation of a {S}tokes--{B}rinkman problem from {S}tokes
  equations around a random array of moving spheres.
\newblock {\em {Commun. Math. Phys.}}, 373(1):265--325, 2020.

\bibitem{CC}
J.~A. {Carrillo} and Y.-P. {Choi}.
\newblock {Quantitative error estimates for the large friction limit of
  {V}lasov equation with nonlocal forces}.
\newblock {\em {Ann. Inst. Henri Poincar\'e, Anal. Non Lin\'eaire}},
  37(4):925--954, 2020.

\bibitem{carrillo2006stability}
J.~A. Carrillo and T.~Goudon.
\newblock Stability and asymptotic analysis of a fluid-particle interaction
  model.
\newblock {\em Commun. Partial. Differ. Equ.}, 31(9):1349--1379, 2006.

\bibitem{CGP}
J.-A. Carrillo, T.~Goudon, and P.~Lafitte.
\newblock Simulation of fluid and particles flows: Asymptotic preserving
  schemes for bubbling and flowing regimes.
\newblock {\em {J. Comput. Phys.}}, 227(16):7929--7951, 2008.

\bibitem{Chae}
D.~{Chae}.
\newblock {Global regularity for the 2D {B}oussinesq equations with partial
  viscosity terms}.
\newblock {\em {Adv. Math.}}, 203(2):497--513, 2006.

\bibitem{CMX}
D.~Chae, Q.~Miao, and L.~Xue.
\newblock Global regularity of non-diffusive temperature fronts for the 2d
  viscous {B}oussinesq system.
\newblock {\em arXiv preprint arXiv:2110.06442}, 2021.

\bibitem{FN1}
J.~{Crevat}.
\newblock {Asymptotic limit of a spatially-extended mean-field
  {F}itz{H}ugh--{N}agumo model}.
\newblock {\em {Math. Models Methods Appl. Sci.}}, 30(5):957--990, 2020.

\bibitem{FN2}
J.~{Crevat}, G.~{Faye}, and F.~{Filbet}.
\newblock {Rigorous derivation of the nonlocal reaction-diffusion
  {F}itz{H}ugh--{N}agumo system}.
\newblock {\em {SIAM J. Math. Anal.}}, 51(1):346--373, 2019.

\bibitem{CHW}
H.~Cui, W.~Wang, and L.~Yao.
\newblock Asymptotic analysis for 1d compressible {N}avier-{S}tokes-{V}lasov
  equations.
\newblock {\em {Commun. Pure Appl. Anal.}}, 19(5):2737, 2020.

\bibitem{DP}
R.~{Danchin} and M.~{Paicu}.
\newblock {Les th\'eor\`emes de Leray et de Fujita-Kato pour le syst\`eme de
  Boussinesq partiellement visqueux}.
\newblock {\em {Bull. Soc. Math. Fr.}}, 136(2):261--309, 2008.

\bibitem{DZ}
R.~Danchin and X.~Zhang.
\newblock Global persistence of geometrical structures for the boussinesq
  equation with no diffusion.
\newblock {\em {Commun. Partial Differ. Equations}}, 42(1):68--99, 2017.

\bibitem{D}
L.~Desvillettes.
\newblock Some aspects of the modeling at different scales of multiphase flows.
\newblock {\em {Comput. Methods Appl. Mech. Eng.}}, 199(21-22):1265--1267,
  2010.

\bibitem{DGR}
L.~Desvillettes, F.~Golse, and V.~Ricci.
\newblock The mean-field limit for solid particles in a {N}avier-{S}tokes flow.
\newblock {\em {J. Stat. Phys.}}, 131(5):941--967, 2008.

\bibitem{DPL}
R.~J. DiPerna and P.-L. Lions.
\newblock Ordinary differential equations, transport theory and {S}obolev
  spaces.
\newblock {\em {Invent. Math.}}, 98(3):511--547, 1989.

\bibitem{DWZZ}
C.~R. Doering, J.~Wu, K.~Zhao, and X.~Zheng.
\newblock Long time behavior of the two-dimensional {B}oussinesq equations
  without buoyancy diffusion.
\newblock {\em {Physica D}}, 376:144--159, 2018.

\bibitem{DS}
L.~Dong and Y.~Sun.
\newblock On asymptotic stability of the 3d {B}oussinesq equations without
  thermal conduction.
\newblock {\em arXiv preprint arXiv:2107.10082}, 2021.

\bibitem{E}
L.~Ertzbischoff.
\newblock Decay and absorption for the {V}lasov--{N}avier--{S}tokes system with
  gravity in a half-space.
\newblock {\em arXiv preprint arXiv:2107.02200}, 2021.

\bibitem{ErtzbischoffPHD}
L.~Ertzbischoff.
\newblock Mathematical analysis of some fluid-kinetic systems of equations.
\newblock {\em PhD thesis - Institut Polytechnique de Paris}, 2023.

\bibitem{EHKM}
L.~Ertzbischoff, D.~Han-Kwan, and A.~Moussa.
\newblock Concentration versus absorption for the {V}lasov-{N}avier-{S}tokes
  system on bounded domains.
\newblock {\em Nonlinearity}, 34(10):6843, 2021.

\bibitem{FK}
A.~{Figalli} and M.-J. {Kang}.
\newblock {A rigorous derivation from the kinetic {C}ucker-{S}male model to the
  pressureless {E}uler system with nonlocal alignment}.
\newblock {\em {Anal. PDE}}, 12(3):843--866, 2019.

\bibitem{FLR1}
F.~Flandoli, M.~Leocata, and C.~Ricci.
\newblock The {V}lasov-{N}avier-{S}tokes equations as a mean field limit.
\newblock {\em Discrete Contin. Dyn. Syst. Ser. B}, 24(8):3741--3753, 2019.

\bibitem{FLR2}
F.~Flandoli, M.~Leocata, and C.~Ricci.
\newblock The {N}avier-{S}tokes-{V}lasov-{F}okker-{P}lanck system as a scaling
  limit of particles in a fluid.
\newblock {\em J. Math. Fluid Mech.}, 23(2):Paper No. 40, 39, 2021.

\bibitem{GGJ}
F.~Gancedo and E.~Garc{\'\i}a-Ju{\'a}rez.
\newblock Regularity results for viscous 3d {B}oussinesq temperature fronts.
\newblock {\em {Commun. Math. Phys.}}, 376:1705--1736, 2020.

\bibitem{Giga}
Y.~Giga and H.~Sohr.
\newblock Abstract {L}p estimates for the {C}auchy problem with applications to
  the {N}avier-{S}tokes equations in exterior domains.
\newblock {\em {J. Funct. Anal.}}, 102(1):72--94, 1991.

\bibitem{GHKM}
O.~Glass, D.~Han-Kwan, and A.~Moussa.
\newblock The {V}lasov--{N}avier--{S}tokes system in a 2d pipe: Existence and
  stability of regular equilibria.
\newblock {\em {Arch. Ration. Mech. Anal.}}, 230(2):593--639, 2018.

\bibitem{Gou}
T.~Goudon.
\newblock Asymptotic problems for a kinetic model of two-phase flow.
\newblock {\em Proceedings of the Royal Society of Edinburgh Section A:
  Mathematics}, 131(6):1371--1384, 2001.

\bibitem{goudon2004hydrodynamic1}
T.~Goudon, P.-E. Jabin, and A.~Vasseur.
\newblock Hydrodynamic limit for the {V}lasov-{N}avier-{S}tokes equations. part
  i: Light particles regime.
\newblock {\em {Indiana Univ. Math. J.}}, pages 1495--1515, 2004.

\bibitem{goudon2004hydrodynamic2}
T.~Goudon, P.-E. Jabin, and A.~Vasseur.
\newblock Hydrodynamic limit for the {V}lasov-{N}avier-{S}tokes equations. part
  ii: Fine particles regime.
\newblock {\em {Indiana Univ. Math. J.}}, pages 1517--1536, 2004.

\bibitem{GouPou}
T.~Goudon and F.~Poupaud.
\newblock On the modeling of the transport of particles in turbulent flows.
\newblock {\em ESAIM MATH. Model. Numer. Anal}, 38(4):673--690, 2004.

\bibitem{PigongSchon}
P.~Han and M.~E. Schonbek.
\newblock Large time decay properties of solutions to a viscous {Boussinesq}
  system in a half space.
\newblock {\em Adv. Differ. Equ.}, 19(1-2):87--132, 2014.

\bibitem{HK-quasin}
D.~{Han-Kwan}.
\newblock {Quasineutral limit of the {}Vlasov-{P}oisson system with massless
  electrons}.
\newblock {\em {Commun. Partial Differ. Equations}}, 36(7-9):1385--1425, 2011.

\bibitem{HK-HDR}
D.~{Han-Kwan}.
\newblock {Stabilité, limites singulières et conditions de contrôle
  géométrique en théorie cinétique}.
\newblock {\em {Habilitation à diriger les recherches}}, 2017.

\bibitem{HK}
D.~Han-Kwan.
\newblock Large time behavior of small data solutions to the
  {V}lasov--{N}avier--{S}tokes system on the whole space.
\newblock {\em Probability and Mathematical Physics}, To appear.

\bibitem{HKI}
D.~{Han-Kwan} and M.~{Iacobelli}.
\newblock {The quasineutral limit of the {V}lasov-{P}oisson equation in
  {W}asserstein metric}.
\newblock {\em {Commun. Math. Sci.}}, 15(2):481--509, 2017.

\bibitem{HKM}
D.~Han-Kwan and D.~Michel.
\newblock On hydrodynamic limits of the {V}lasov-{N}avier-{S}tokes system.
\newblock {\em Mem. Amer. Math. Soc}, To appear.

\bibitem{HKMM}
D.~{Han-Kwan}, A.~{Moussa}, and I.~{Moyano}.
\newblock {Large time behavior of the {V}lasov-{N}avier-{S}tokes system on the
  torus}.
\newblock {\em {Arch. Ration. Mech. Anal.}}, 236(3):1273--1323, 2020.

\bibitem{HLP}
G.~H. Hardy, J.~E. Littlewood, and G.~P\'{o}lya.
\newblock {\em Inequalities}.
\newblock Cambridge Mathematical Library. Cambridge University Press,
  Cambridge, 1988.
\newblock Reprint of the 1952 edition.

\bibitem{H}
M.~Hillairet.
\newblock On the homogenization of the {S}tokes problem in a perforated domain.
\newblock {\em {Arch. Ration. Mech. Anal.}}, 230(3):1179--1228, 2018.

\bibitem{H2}
M.~Hillairet.
\newblock Derivation of the {S}tokes--{B}rinkman problem and extension to the
  {D}arcy regime.
\newblock {\em {J. Elliptic Parabol. Equ.}}, pages 1--20, 2021.

\bibitem{HMS}
M.~Hillairet, A.~Moussa, and F.~Sueur.
\newblock On the effect of polydispersity and rotation on the {B}rinkman force
  induced by a cloud of particles on a viscous incompressible flow.
\newblock {\em Kinet. Relat. Models}, 12, 2019.

\bibitem{HmidiK}
T.~Hmidi and S.~Keraani.
\newblock On the global well-posedness of the two-dimensional {B}oussinesq
  system with a zero diffusivity.
\newblock {\em {Adv. Differ. Equ.}}, 12(4):461--480, 2007.

\bibitem{HR}
T.~{Hmidi} and F.~{Rousset}.
\newblock {Global well-posedness for the {N}avier-{S}tokes-{B}oussinesq system
  with axisymmetric data}.
\newblock {\em {Ann. Inst. Henri Poincar\'e, Anal. Non Lin\'eaire}},
  27(5):1227--1246, 2010.

\bibitem{Hof}
R.~M. H{\"o}fer.
\newblock Sedimentation of inertialess particles in {S}tokes flows.
\newblock {\em {Commun. Math. Phys.}}, 360(1):55--101, 2018.

\bibitem{HoferPHD}
R.~M. H{\"o}fer.
\newblock Sedimentation of particle suspensions in {S}tokes flows.
\newblock {\em Universit{\"a}ts-und Landesbibliothek Bonn}, 2020.

\bibitem{HS}
R.~M. {H\"ofer} and R.~{Schubert}.
\newblock {The influence of {E}instein's effective viscosity on sedimentation
  at very small particle volume fraction}.
\newblock {\em {Ann. Inst. Henri Poincar\'e, Anal. Non Lin\'eaire}},
  38(6):1897--1927, 2021.

\bibitem{HL}
T.~Y. {Hou} and C.~{Li}.
\newblock {Global well-posedness of the viscous {B}oussinesq equations}.
\newblock {\em {Discrete Contin. Dyn. Syst.}}, 12(1):1--12, 2005.

\bibitem{HoferInertia}
R.~M. Höfer.
\newblock The inertialess limit of particle sedimentation modeled by the
  {V}lasov-{S}tokes equations.
\newblock {\em {SIAM J. Math. Anal.}}, 50(5):5446--5476, 2018.

\bibitem{Gray}
H.~G. II.
\newblock Dynamics of density patches in infinite prandtl number convection.
\newblock {\em arXiv preprint arXiv:2207.09738}, 2022.

\bibitem{jabin2000macroscopic}
P.-E. Jabin.
\newblock Macroscopic limit of {V}lasov type equations with friction.
\newblock {\em {Ann. Inst. Henri Poincar\'e, Anal. Non Lin\'eaire}},
  17(5):651--672, 2000.

\bibitem{Jab2}
P.-E. {Jabin}.
\newblock {Various levels of models for aerosols.}
\newblock {\em {Math. Models Methods Appl. Sci.}}, 12(7):903--919, 2002.

\bibitem{KW}
I.~Kukavica and W.~Wang.
\newblock Long time behavior of solutions to the 2d {B}oussinesq equations with
  zero diffusivity.
\newblock {\em {J. Dyn. Differ. Equations}}, 32(4):2061--2077, 2020.

\bibitem{Antoine}
A.~{Leblond}.
\newblock {Well-posedness of the Stokes-transport system in bounded domains and
  in the infinite strip}.
\newblock {\em {J. Math. Pures Appl. (9)}}, 158:120--143, 2022.

\bibitem{Lunardi}
A.~Lunardi.
\newblock {\em Interpolation theory}, volume~16.
\newblock Springer, 2018.

\bibitem{Maj-ocea}
A.~{Majda}.
\newblock {\em {Introduction to PDEs and waves for the atmosphere and ocean}},
  volume~9.
\newblock Providence, RI: American Mathematical Society (AMS); New York, NY:
  Courant Institute of Mathematical Sciences, 2003.

\bibitem{MB}
A.~J. {Majda} and A.~L. {Bertozzi}.
\newblock {\em {Vorticity and incompressible flow}}.
\newblock Cambridge: Cambridge University Press, 2002.

\bibitem{Mas}
N.~{Masmoudi}.
\newblock {From {V}lasov-{P}oisson system to the incompressible {E}uler
  system.}
\newblock {\em {Commun. Partial Differ. Equations}}, 26(9-10):1913--1928, 2001.

\bibitem{MSZ}
N.~Masmoudi, B.~Said-Houari, and W.~Zhao.
\newblock Stability of {C}ouette flow for 2d {B}oussinesq system without
  thermal diffusivity.
\newblock {\em arXiv preprint arXiv:2010.01612}, 2020.

\bibitem{Mech}
A.~Mecherbet.
\newblock Sedimentation of particles in {S}tokes flow.
\newblock {\em Kinet. Relat. Models}, 12, 2019.

\bibitem{Mech2}
A.~Mecherbet.
\newblock On the sedimentation of a droplet in {S}tokes flow.
\newblock {\em To appear in Comm Math Sci}, 2020.

\bibitem{MS}
A.~Mecherbet and F.~Sueur.
\newblock A few remarks on the transport-stokes system.
\newblock {\em arXiv preprint arXiv:2209.11637}, 2022.

\bibitem{MV}
A.~Mellet and A.~Vasseur.
\newblock Asymptotic analysis for a {V}lasov-{F}okker-{P}lanck/compressible
  {N}avier-{S}tokes system of equations.
\newblock {\em {Commun. Math. Phys.}}, 281(3):573--596, 2008.

\bibitem{Misch}
S.~Mischler.
\newblock On the trace problem for solutions of the {V}lasov equation.
\newblock {\em Commun. Partial. Differ. Equ.}, 25(7-8):1415--1443, 2000.

\bibitem{MoSu}
A.~Moussa and F.~Sueur.
\newblock On a {V}lasov--{E}uler system for 2d sprays with gyroscopic effects.
\newblock {\em {Asymptotic Anal.}}, 81(1):53--91, 2013.

\bibitem{oro}
P.~J. O'Rourke.
\newblock Collective drop effects on vaporizing liquid sprays.
\newblock Technical report, Los Alamos National Lab., NM (USA), 1981.

\bibitem{S-ocea}
R.~Salmon.
\newblock {\em Lectures on geophysical fluid dynamics}.
\newblock Oxford University Press, 1998.

\bibitem{SY}
Y.~Su and L.~Yao.
\newblock Hydrodynamic limit for the inhomogeneous incompressible
  {N}avier-{S}tokes/{V}lasov-{F}okker-{P}lanck equations.
\newblock {\em {J. Differ. Equations}}, 269(2):1079--1116, 2020.

\bibitem{TWZZ}
L.~Tao, J.~Wu, K.~Zhao, and X.~Zheng.
\newblock Stability near hydrostatic equilibrium to the 2d {B}oussinesq
  equations without thermal diffusion.
\newblock {\em {Arch. Ration. Mech. Anal.}}, 237(2):585--630, 2020.

\bibitem{V-ocea}
G.~K. Vallis.
\newblock {\em Atmospheric and oceanic fluid dynamics}.
\newblock Cambridge University Press, 2017.

\bibitem{Wieg}
M.~Wiegner.
\newblock Decay results for weak solutions of the {N}avier-{S}tokes equations
  on {R}n.
\newblock {\em {J. Lond. Math. Soc., II. Ser.}}, 2(2):303--313, 1987.

\bibitem{will}
F.~A. Williams.
\newblock {\em Combustion theory}.
\newblock Benjamin Cummings, second edition, 1985.

\end{thebibliography}

\end{document}